\documentclass[12pt]{article}
\usepackage{amsmath,amssymb,amsfonts}
\usepackage{amsthm}
\usepackage[mathscr]{euscript}
\usepackage{color}
\usepackage[T1]{fontenc}
\usepackage{graphicx}
\usepackage{eepic}
\usepackage{epic}
\newtheorem{theorem}{Theorem}[section]
\newtheorem{fact}[theorem]{Fact}

\newtheorem{corollary}[theorem]{Corollary}

\newtheorem{example}[theorem]{Example}
\newtheorem{lemma}[theorem]{Lemma}

\newtheorem{proposition}[theorem]{Proposition}

\theoremstyle{definition}
\newtheorem{definition}[theorem]{Definition}
\newtheorem{thmalph}{Theorem}

\newtheorem{probalph}[thmalph]{Problem}

\theoremstyle{remark}
\newtheorem{remark}[theorem]{Remark}

\newcommand{\la}{\mathbb \lambda}

\newcommand{\tpi}{\tilde{\pi}}
\newcommand{\uxi}{\underline{\xi}}

\newcommand{\ux}{\underline{x}}
\newcommand{\uz}{\underline{z}}

\newcommand{\eps}{\epsilon}

\newcommand{\mP}{\mathbb P}
\newcommand{\mS}{\mathbb S}

\newcommand{\mC}{\mathbb C}

\newcommand{\mJ}{\mathbb J}

\newcommand{\lC}{{\cal C}}

\newcommand{\lD}{{\cal D}}

\newcommand{\lN}{{\cal N}}
\newcommand{\lO}{{\cal O}}

\newcommand{\mR}{\mathbb R}
 \newcommand{\mN}{\mathbb N}
\newcommand{\be}{\begin{eqnarray}}
\newcommand{\ee}{\end{eqnarray}}
\newcommand{\bd}{\begin{definition}}
\newcommand{\ed}{\end{definition}}
\newcommand{\br}{\begin{remark}}
\newcommand{\er}{\end{remark}}

\newcommand{\gog}{{\mathfrak g}}
\newcommand{\gom}{{\mathfrak m}}

\newcommand{\gon}{{\mathfrak n}}
\newcommand{\gop}{{\mathfrak p}}

\newcommand{\goh}{{\mathfrak h}}

\newcommand{\gol}{{\mathfrak l}}

\newcommand{\sh}{{\fam2 O}}

\newcommand{\fn}{{\fam2 C}}

\newcommand{\V}{\mathcal{V}}

\newcommand{\bl}{\begin{lemma}}
\newcommand{\el}{\end{lemma}}

\newcommand{\bp}{\begin{picture}}
\newcommand{\ep}{\end{picture}}
\newcommand{\bi}{\begin{itemize}}
\newcommand{\ei}{\end{itemize}}
\newcommand{\bq}{\begin{quotation}}
\newcommand{\eq}{\end{quotation}}
\newcommand{\pa}{\partial}

\newcommand{\Pol}{\operatorname{Pol}}

\newcommand{\ad}{\operatorname{ad}}

\newcommand{\Ind}{\operatorname{Ind}}
\newcommand{\Id}{\operatorname{Id}}
\newcommand{\Hom}{\operatorname{Hom}}

\newcommand{\ch}{{\,\vee}}

\usepackage{hyperref}%

\begin{document}
\date{}
\title{Branching laws for Verma modules and applications in parabolic geometry. I}
\author{ Toshiyuki Kobayashi,
         Bent \O rsted, \\
         Petr Somberg, 
         Vladimir Sou\v cek}

\maketitle

\abstract 

We initiate a new study of differential
operators with symmetries and combine this with the study of
branching laws for Verma modules
 of reductive Lie algebras. 
By the criterion for discretely decomposable
 and multiplicity-free restrictions
 of generalized Verma modules
 [T. Kobayashi,
\href{http://dx.doi.org/10.1007/s00031-012-9180-y}
{Transf. Groups (2012)}],
 we are brought to natural settings
 of parabolic geometries
 for which there exist unique equivariant differential operators
 to submanifolds.  
Then we apply a new method (F-method)
 relying on  the  Fourier transform
 to find singular vectors in generalized Verma modules,
which significantly simplifies
 and generalizes many preceding works.  
In certain cases,
 it also determines the Jordan--H{\"o}lder series
 of the restriction 
 for singular parameters.  
The F-method yields an explicit formula
 of such unique operators,
 for example,
 giving an intrinsic and new proof of
 Juhl's conformally invariant differential operators
[Juhl, \href{http://dx.doi.org/10.1007/978-3-7643-9900-9}
{Progr. Math. 2009}]
 and its generalizations to spinor bundles.
This article is the first in the series,
 and the next ones include
 their extension
 to curved cases
 together with more applications
 of the F-method
 to various settings in parabolic geometries.  
\vskip 0.8pc
{{\bf {Key words}}: F-method, branching law,  conformal geometry, parabolic geometry, 
equivariant differential operator, Verma module, symmetric pairs }.
\vskip 0.8pc
{{\bf {2010 MSC}}: 
Primary 
53A30 
;
Secondary
22E47, 
33C45, 
58J70.  
}

\endabstract

\tableofcontents

\section{Introduction}
 \numberwithin{equation}{section}
 \setcounter{equation}{0}
Let $G'\subset G$ be a pair of real reductive Lie groups. The main objects of this paper are $G'$-equivariant differential operators between two homogeneous vector bundles over two real flag manifolds 
 $N=G'/P'$ and 
 $M=G/P,$ where $N$ is a submanifold of $M$ corresponding to $G'\subset G.$ We provide a new method 
 (\textquotedbl{F-method}\textquotedbl) for constructing explicitly such operators, and demonstrate 
its effectiveness in several concrete examples.

On the algebraic level --- in the dual language of
homomorphisms between (generalized) Verma modules, 
the whole construction is connected
 to natural questions in
representation theory, namely branching laws; 
 no attempts at a systematic approach
to branching laws for Verma modules has been made
 until quite recently,
and our results might be of independent interest
from this point of view. 
Restricting Verma modules
 to reductive subalgebras appears manageable
 at a first glance,
 however, it involves sometimes wild
problems
 such as the effect
 of \textquotedbl{hidden continuous spectrum}\textquotedbl\
 as was revealed in \cite{k2}.  
Nevertheless
 there is a considerably rich family
 of examples
with good behaviour
 such as discretely decomposable
 and multiplicity-free restrictions \cite{k1},
 which are at the same time particularly
important for our geometric purposes.

Some of the operators we construct 
for example, 
 (powers of) the wave operator and
Dirac operator
 appeared previously
 in physics.  
Since a large amount of natural differential operators
 have been already found in parabolic geometries, it is
worth pointing out,
 that the ones treated as a prototype here are
exactly the ones that are the hardest to find by the previous
methods (essentially coming from the BGG resolution).  
Further we extend this prototype
 in two folds
 to arbitrary signatures
 in pseudo-Riemannian manifolds (Theorem \ref{thm:4.3})
 and to Dirac operators (Theorem \ref{thm:5.6})
 by the new method.  
We work primarily in the model case situation
 where
the manifold is a real flag manifold, 
but we see in the second part of the series \cite{KOSS2}
 that (as seen for example in the case of
conformal geometry)  it is both possible and
interesting to extend to the \textquotedbl{curved case}\textquotedbl\ of
manifolds equipped with the corresponding parabolic
geometry.

The results we are going to present are inspired by geometrical
considerations. In particular, they correspond to
differential invariants (of higher order in general)
in the case of models for parabolic geometries.
To be precise, let $G$ be a real reductive Lie group, $P$ a parabolic subgroup of $G,$ and $G'$ a reductive subgroup of $G$ such that $P':=P\cap G'$ is a parabolic subgroup of $G'.$ We consider $G'$-equivariant differential operators acting on sections of homogeneous vector bundles
over homogeneous models over $G'/P'$ and $G/P.$  
In effect, what happens is that initial sections on $G/P$ are differentiated
and then restricted to the submanifolds $G'/P'$, and this combined operation
commutes with the action of  the group $G'$.

Explicit formulae for invariant differential operators
constructed in the paper are described in the simplest possible coordinates, i.e., in the noncompact picture.
In principle, there are methods (based on factorization identities) how to compute explicit form
of the differential operators in compact picture but the work needed to do so is nontrivial. An example of such computation can be found
in \cite[Chapt. 5.2]{ju}.
  
Our language chosen for presenting these results is algebraic, 
however, relies at a stage on certain analytic techniques.   
The first step
 is to 
 translate geometrical problems into branching problems 
 of generalized Verma modules for the Lie algebra of $G$
 induced from $P$
 when restricted to the Lie algebra of $G'$. 
Let us recall
 what is known and what is not known.

The {\it{existence}}
 of equivariant differential operators
 is assured
 by the discrete decomposability (Definition \ref{def:3.1})
 of the restriction in the BGG category ${\mathcal{O}}$.  
A general theory of discretely decomposable restrictions
 in the BGG category
$\mathcal{O}$
 as well as
 in the category of Harish-Chandra modules
 was established
 in \cite{kob, kob2, kob3, k2}. 
Moreover,
 the {\it{uniqueness}} of such operators
 is guaranteed
 by the multiplicity-freeness 
 of the restriction.  
An explicit formula
 of the branching law
 for reductive symmetric pairs $(\gog,\gog')$
 was proved in \cite{k1,k2}, 
 which includes the classical Hua--Kostant--Schmid formula
 and also the decomposition
 of the tensor product
  of two modules
 as special cases.  
A short summary is given in Section \ref{sec:3}
 in a way
 that we need.
These general results help us to single out
 appropriate geometric settings
 for which we could expect to construct natural equivariant operators
 to submanifolds,
 however, 
 we need another idea
 to answer
 the following delicate algebraic problems
 of branching laws, 
 which are closely related
 with our geometric interest
 in finding explicit formulae
 equivariant operators
 in parabolic geometries.  
In what follows,  $M^\gog_\gop(\la)\equiv M^\gog_\gop(F_\la)$ 
denotes the $\gog$-module
 induced from an irreducible finite-dimensional $\gop$-module $F_\la$ with
 highest weight $\la$.
\vskip 2mm
\begin{probalph}
\label{prob:A}
Find precisely
 where irreducible ${\mathfrak {g}}'$-submodules
 are located  in a generalized Verma module $M^\gog_\gop(\la)$ of $\gog$.
\end{probalph}

\vskip 2mm
\begin{probalph}
\label{prob:B}
Find the Jordan--H{\"o}lder series
 of a generalized Verma module
 $M^\gog_\gop(\la)$
 of ${\mathfrak {g}}$
 regarded as a module
 of a reductive subalgebra ${\mathfrak {g}}'$
 by the restriction.  
\end{probalph}

\vskip 2mm

Problem \ref{prob:A} is to ask for an
{explicit description of {\it{${\mathfrak {g}}'$-singular vectors}}
 (i.e., vectors annihilated by the action of the nilpotent radical
of $\gop'=\gog'\cap\gop$, see Section \ref{sec:3}
 for the definition)
 in the generalized Verma module $M^\gog_\gop(\la)$) of $\mathfrak g$,
and in turn, is equivalent to our geometric question,
 namely, to construct equivariant
 differential operators
 explicitly from real flag varieties
 to subvarieties (see Theorem \ref{folkloreg}).  
Problem \ref{prob:B} concerns 
 with the case
 where $M^\gog_\gop(\la)$ is not completely reducible,
 in particular, for non-generic parameter $\lambda$.  
It should be noted that even in the case $\mathfrak{g}=\mathfrak{g}'$,
 Problem \ref{prob:B} is already difficult and unsolved in general.
Furthermore, complete reducibility as a $\gog'$-module is
another thing than complete reducibility as a $\gog$-module, and 
it seems  that Problem \ref{prob:B} has never been studied
 before in the case where ${\mathfrak {g}}' \subsetneqq {\mathfrak {g}}$
 (even for Lie algebras of small dimensions).
The new ingredient of Problem \ref{prob:B} is to understand how non-trivial 
 ${\mathfrak {g}}$-extensions occurring in $M^\gog_\gop(\la)$ behave
 when restricted to the subalgebra ${\mathfrak {g}}'$.  

We are interested in Problems \ref{prob:A} and \ref{prob:B},
 in particular, 
when we know a priori the restriction 
 $M^\gog_\gop(\la)|_{\gog'}$ is isomorphic to  a multiplicity-free direct sum
 of irreducible ${\mathfrak {g}}'$-modules for generic parameter $\lambda$.  

In the category ${\mathcal{O}}$, 
every irreducible ${\mathfrak {g}}'$-submodule
 contains a singular vector,
 and conversely,
 every singular vector generates
 a ${\mathfrak {g}}'$-submodule
 of finite length.  
Thus the structure
 of the set of all singular vectors
 is a key to the above mentioned problems.  
In the case
 of conformal densities,
 singular vectors
 were found by using the recurrence relations
 in certain generalized Verma modules
 by A. Juhl \cite{ju}.  
However,
 it seems hard
 to apply such a combinatorial method
 in a more general setting
 due to its computational complexity.

Our method 
 to attack Problems \ref{prob:A} and \ref{prob:B} is based 
on the \textquotedbl{Fourier transform}\textquotedbl\ of generalized Verma modules;
we call it {\it the F-method}.  
The idea is to characterize
 the set of all singular vectors
 by means of a system
 of partial differential equations
 on the Fourier transform side.  
It was first suggested by T. Kobayashi, March 2010, 
 with a number of new examples.   
In contrast to the existing combinatorial techniques
 to find singular vectors, 
 the F-method is more conceptual.

For example, 
the coefficients
 of Juhl's families of equivariant differential operators for the conformal group
 (see \eqref{acoef})
 coincide with those of the Gegenbauer
 polynomials.  
This was discovered 
 by Juhl \cite{ju},
 but the combinatorial proof there based on recurrence relations
 did not explain the origin of the special functions 
 in formulae.  
Our new method is completely different and explains their appearance
in a natural way (see Section \ref{sec:4}).

The F-method itself is briefly described in Section 2.
The key idea of the F-method is 
 to take the Fourier transform 
 of Verma modules
after realizing them 
 in the space of distributions supported
 at the origin 
 on the flag variety.  
Then we can transfer the algebraic branching problem for generalized Verma modules into an analytic problem,
 to find polynomial solutions
 to a system of partial differential equations. 
In the setting we consider,
 it leads to an ordinary differential equation
(due to symmetry involved).  
The resulting second-order differential equations control
 all the family
 of equivariant differential operators (of arbitrarily high order).
The polynomial solutions are the Fourier transform of 
singular vectors.  
Hence this new method offers a uniform and effective tool
to find explicitly singular vectors in many different cases. 

In Section 3, we discuss a class of branching problems for modules
in the parabolic BGG category $\lO^{\gop}$ having a discrete decomposability property
with respect to reductive subalgebras $\gog'$.  
Moreover, one of our guiding principles is
to focus on multiplicity-free cases which were obtained systematically
 in \cite{k1,k2}
 by two methods
--- 
 by visible actions
 on complex manifolds
 and by purely algebraic methods.  
Branching rules
are given in terms of the Grothendieck group 
of the category $\lO^{\gop}$, and they give
geometric settings
 where we shall apply the F-method.  

The rest of the paper contains applications of the F-method for descriptions of the space of all 
singular vectors in particular cases of conformal geometry. 
It contains a complete answer
 to Problems \ref{prob:A} and \ref{prob:B}
 for generalized Verma modules
 of scalar type in the case
 where $(G,G') = (SO_0(p,q),SO_0(p,q-1))$,
 see Theorems \ref{basis} and \ref{T.3.8},  
respectively.  
The explicit construction
 of equivariant  differential operators
 for the particular examples
 of pseudo-Riemannian manifolds
 of arbitrary signature $(p,q)$
 is given in Theorem \ref{thm:4.3},
 extending a theorem of Juhl.  
A further generalization
 to spinor-valued sections
 is discussed in Section \ref{sec:5}, 
and explicit formulae
 of equivariant differential operators 
 for the conformal group
 are given in Theorem \ref{thm:5.6}
 by using the Dirac operator
 and the coefficients of Gegenbauer polynomials.  
Again the main machinery is the F-method.

As we already emphasized, our original motivation for the study of branching rules 
for generalized Verma modules came from differential geometry.  
In fact,
 there is a substantial relation of the curved version
of the Juhl family and {a notion} of $Q$-curvature and conformally invariant
powers of the Laplace operator. 
In the second part
 of the series \cite{KOSS2}, 
 we construct the curved version of the Juhl family 
 and its generalization 
 by using the result
 of this article
 and the ideas of semi-holonomic Verma modules. 
 
To summarize, we have in this paper 
 provided a new method,
 and some new results concerning the relation
between several important topics in representation theory
 and parabolic geometry, namely branching laws for generalized Verma modules 
and the construction of equivariant  differential operators 
 to submanifolds. 
In the second part of the series, \cite{KOSS2},
 we give further applications
 of the F-method to some other examples
 of parabolic geometries 
with commutative nilradical,
 e.g., 
 the projective geometry, Grassmannian 
geometry and Rankin--Cohen brackets  as an example of branching 
rules for the symmetric pair $(G\times G,\Delta (G)),$ where $
\Delta:G\to G \times G$ is a diagonal embedding, 
 and discuss their \textquotedbl{curved versions}\textquotedbl.  

In the paper, we use the following notation:
 $\mN=\{0,1,2,\cdots \}$, $\mN_+=\{1,2,\cdots \}$.

\section{Problems and methods for their solutions}
\label{sec:2 correction}
 \numberwithin{equation}{section}
 \setcounter{equation}{0}

The first aim of this section is to explain in more details the connection between geometric
and algebraic side of the problem and
a method to find ${\mathfrak {g}}'$-singular vectors in Verma modules
 of ${\mathfrak {g}}$, 
 where ${\mathfrak {g}}' \subset {\mathfrak {g}}$ are
a pair of complex reductive Lie algebras.  The second aim 
is   to discuss the general idea of a new approach
how to describe equations for singular vectors by using Fourier analysis.  
The main advantage of the method, 
which we call \textquotedbl{F-method}\textquotedbl , 
is
 that a combinatorially complicated 
 problem
 of finding singular vectors
 by the existing methods
 is converted to a more conceptual question
 to find polynomial solutions
 of a certain system of partial or ordinary 
differential equations, 
 see \eqref{eqn:phi}. 
Explicit examples showing how the method works   in various situations (related in problems in differential geometry)
 can be found in the latter
  half of the paper, 
 in the second part \cite{KOSS2}
 of the series, and \cite{KP2}. 
 
\subsection{Two dual faces of the problem}
Let $G' \subset G$ be real reductive Lie groups,
 and ${\mathfrak{g}}' \subset {\mathfrak {g}}$
 their complexified Lie algebras.  
In the paper, we are studying two closely related problems.
On the side of geometry, we are going to construct intertwining differential operators between principal series representations
 of the two groups $G$ and $G'$.
On algebraic side, we are going to construct homomorphisms between generalized Verma modules
 of the two Lie algebras ${\mathfrak{g}}'$ and ${\mathfrak {g}}$. 
The relation between these geometric and algebraic sides is 
 classically known when $G' = G$ (see 
 $\cite{BS}$, 
 for instance). 
We generalize it to the case $G' \ne G$
 in connection with branching problems
 in representation theory.  
The treatment here
 is based on recent publications
 \cite{CSS1, CSS2} by using jet bundles.  
A self-contained account (in a more general situation)
 by a different approach
 can be found in \cite[Sect. 2]{KP}.

Let $G$ be a connected real reductive Lie group
 with Lie algebra $\gog({\mathbb{R}})$.  
Let $x \in {\mathfrak {g}}({\mathbb{R}})$ be a hyperbolic element.  
This means that $\ad(x)$ is diagonalizable and its
eigenvalues are all real. 
Then we have the following Gelfand--Naimark decomposition
$$
\gog({\mathbb{R}})
=\gon_-({\mathbb{R}}) + \gol ({\mathbb{R}})+ \gon_+({\mathbb{R}}),
$$
according to the negative, zero,
 and positive eigenvalues of $\ad(x)$.  
The subalgebra $\gop({\mathbb{R}}):=\gol({\mathbb{R}}) + \gon_+({\mathbb{R}})$
 is a parabolic subalgebra of $\gog({\mathbb{R}})$, 
 and its normalizer $P$ in $G$
 is a parabolic subgroup of $G$.  
Subgroups $N_\pm\subset G$ are defined
 by $N_\pm=\exp \gon_\pm({\mathbb{R}}).$

 Given a complex finite-dimensional $P$-module $V$,
 we consider the unnormalized induced representation $\pi$
 of $G$ on the space $\operatorname{Ind}_P^G(V)$
 of smooth sections
 for the homogeneous vector bundle
 ${\mathcal{V}}:=G \times_P V \to G/P$.  
We can identify this space
 with 
$$
\lC^\infty(G,V)^P:=\{f\in\lC^\infty(G,V):f(g\,p)=p^{-1}\cdot f(g),\,g\in G,p\in P\}.
$$
 
Moreover,
we shall also need the space $J^k_e(G,V)^P$ of $k$-jets in $e\in G$
of $P$-equivariant maps and its projective limit
$$
J^\infty_e(G,V)^P={\lim_{\longrightarrow}}{}_{\, k} J^k_e(G,V)^P.
$$

Let $U(\gog)$
 denote the universal enveloping algebra
 of the complexified Lie algebra ${\mathfrak {g}}$
 of ${\mathfrak {g}}({\mathbb{R}})$.  
Let $V^\ch$ denote the contragredient representation.  
Then $V^\ch$ extends to a representation
 of the whole enveloping algebra $U(\gop).$
The generalized Verma module $M^\gog_\gop(V^\ch)$ is  defined by
$$
M^\gog_\gop(V^\ch):= 
U(\gog)\otimes_{U(\gop)}V^\ch.  
$$
It is a well-known fact that there is a non-degenerate invariant pairing between
$J^\infty_e(G,V)^P$  and $M^\gog_\gop(V^\ch).$
In what follows, we present a version better adapted for our needs. All details
 of the proof  of the version of the claim presented below can be found in
\cite[App.1]{CSS2},
 which is an extended version of the paper $\cite{CSS1}.$

\begin{fact}\label{folklore}
Let $G$ be a connected semisimple Lie group
 with complexified Lie algebra $\gog$, 
 and $P$ a parabolic subgroup of $G$
 with complexified Lie algebra $\gop.$
Suppose further that $V$ is a finite-dimensional $P$-module and
$V^\ch$ its dual.
Then there is a $({\mathfrak {g}}, P)$-invariant  
 pairing between  
 $J^\infty_e(G,V)^P$  and   $M^\gog_\gop(V^\ch),$ which identifies
 $M^\gog_\gop(V^\ch)$ with the space of all linear maps from  
 $J^\infty_e(G,V)^P$ to $\mC$ that factor through $J^k_e(G,V)^P$ for some $k.$
 
\end{fact}

The statement above has the following classical corollary
 (see \cite{BS} for instance), which
explains the relation between the geometrical problem of finding  
$G$-equivariant differential operators between induced representations
and the algebraic problem of finding homomorphisms between generalized
Verma modules.

\begin{corollary}
\label{cor:2.2}
Let $V$ and $V'$ be two finite-dimensional $P$-modules.
Then the space of $G$-equivariant differential operators
from $\Ind_P^G(V)$ to  $\Ind_P^G(V')$ is isomorphic to
 the space of $(\gog,P)$-homomorphisms
from $M^\gog_\gop((V')^\ch)$ to $M^\gog_\gop(V^\ch)$.
\end{corollary}
 
We now discuss a generalization of Corollary \ref{cor:2.2}
 to the 
case of homogeneous vector bundles over {\it different} flag manifolds, which is formulated and proved below.
Suppose that, in addition to the pair $P\subset G$ used
above, we  consider another pair $P'\subset G',$
 such that $G'$ is a reductive
subgroup of $G$ and $P'=P \cap G'$ is a parabolic subgroup of $G'.$
We shall see
 that this happens if
 ${\mathfrak {p}}$ is ${\mathfrak {g}}'$-compatible
 in the sense of Definition \ref{def:3.2}.
 In this case, 
 ${\mathfrak {n}}_+':={\mathfrak {n}}_+ \cap {\mathfrak {g}}'$
 is the nilradical 
 of ${\mathfrak {p}}'$, 
 and we set $L'=L \cap G'$
 for the corresponding Levi subgroup in $G'$. 
 
For any smooth vector bundle 
$\V\to M,$ there exists a unique (up to isomorphism) vector bundle
$J^k(\V)$ over $M$ (called the $k$-th jet prolongantion of $\V$)
together with the canonical differential operator
$$
J^k:\fn^\infty(M,{\mathcal{V}})\rightarrow \fn^\infty(M,J^k{\mathcal{V}})  
$$ of order $k.$
Recall that a linear operator $D:\fn^\infty(M,{\mathcal{V}})
\rightarrow\fn^\infty(M,{\mathcal{V}}')$ between two smooth vector bundles
over $M$ is called a differential operator of order at most $k,$
if there is a bundle morphism $Q:J^k{\mathcal{V}}\to{\mathcal{V}}'$
such that $D=Q\circ J^k.$
We need a generalization of this definition to the case of a linear operator acting between vector fibre bundles over two {\it different} smooth manifolds.

\begin{definition}
\label{diffop} 
Fix $k\in\mN$.
Let  $p:N\rightarrow M$ be a smooth map
 between two smooth manifolds and
let ${\mathcal{V}}\rightarrow M$ 
 (respectively, 
 ${\mathcal{V}}'\rightarrow N$)
 be two smooth vector bundles.

A linear map $D:\fn^\infty(M,{\mathcal{V}})
\rightarrow\fn^\infty(N,{\mathcal{V}}')$
 is said to be a differential
operator of order $k,$ if there exists a bundle map
$Q:\fn^\infty(N,p^*(J^k{\mathcal{V}}))\rightarrow\fn^\infty(N,{\mathcal{V}}')$
such that
$$
D=Q\circ p^*\circ J^k.
$$
\end{definition}

An alternative definition of a differential operator
can be based on suitable local properties.
For operators between bundles over the same manifold, the relation between both definitions is contained in the classical theorem of Peetre ($\cite{Peetre}$).  For a more general situation (including the case of an linear operator between bundles over different manifolds),
 an analogue of the Peetre theorem also holds, 
 see \cite[Sect.2]{KP} and 
\cite[Chapt. 19]{KMS}. 
 
 Now we can formulate and prove a generalization of Corollary \ref{cor:2.2} in a more general situation
as follows:

\begin{theorem}\label{folkloreg}
The set of all $G'$-equivariant differential operators from
$\Ind_P^G(V)$ to $\Ind_{P'}^{G'}(V')$ is in one-to-one
  correspondence with the space of all $(\gog',P')$-homomorphisms
from  $M^{\gog'}_{\gop'}({V'}^\ch)$ to $M^\gog_\gop(V^\ch).$
\end{theorem}

 
\begin{proof}
The inclusion $i:G'\to G$ induces a smooth map $i:G'/P'\to G/P.$
The fibers of  $J^k({\mathcal{V}})$ and $i^*(J^k({{V}}))$ over $o\in G'/P'$ are both
isomorphic to  $J^k_e(G,{{V}})^P, $ hence $G'$-equivariant bundle maps from 
$i^*(J^k({\mathcal{V}}))$ to $V'$
over $G'/P'$ are
in bijective correspondence with elements in 
$\Hom_{P'}(J^k(G,{{V}})^P,V').$  It implies that $G'$-equivariant differential operators
from $\Ind_P^G(V)$ to $\Ind_{P'}^{G'}(V')$ are in one-to-one correspondence with
$P'$-homomorphisms from $J^\infty_e(G,{{V}})^P$ to $V'$ that factor through $J^k_e(G,V)^P$ for some $k.$
By using the pairing in  Fact \ref{folklore},
 such homomorphisms are in bijective correspondence
 with elements in 
$
(M^\gog_\gop(V^\ch)\otimes V')^{P'}
 \simeq \Hom_{P'}((V')^\ch,M^\gog_\gop(V^\ch)).
$
Finally, the Frobenius reciprocity gives us the bijective correspondence
between the spaces $\Hom_{P'}((V')^\ch,M^\gog_\gop(V^\ch))$
and $\Hom_{(\gog',P')}(M^{\gog'}_{\gop'}({V'}^\ch), M^\gog_\gop(V^\ch)).$
\end{proof}

A detailed account of Theorem \ref{folkloreg}
 (in a more general setting)
 with a different proof
 can be found in \cite{KP}.  
See also \cite[Chap. 3]{KSpeh} for the general perspectives
 of {\it{continuous}} symmetry breaking operators
 which include $G'$-equivariant
 {\it{differential}} operators
 as a special case.



\subsection{F-method}
Let us recall that we now consider a general setting
 with a given pair $(P,G)$  together with  another pair $(P',G')$,
 such that $G'\subset G$ is a reductive
subgroup of $G$, $P=LN_+$ is a Levi decomposition of a (real)
parabolic subgroup of $G$, and $P'=P \cap G'$ is a parabolic subgroup of $G'$
with Levi decomposition $L'N_+'\equiv (L\cap G')(N_+\cap G')$. We write
$\gog,\gog',\gop,\gop',\gol,\gol',\gon_+,$ and $\gon_+'$ for the 
complexified Lie algebras of $G,G',P,P',L,L',N_+,$ and $N_+'$, respectively. 
Then ${\mathfrak {n}}_+'={\mathfrak {n}}_+ \cap {\mathfrak {g}}'$
 is the nilradical of ${\mathfrak {p}}'$.  
A usual classical setting 
 is that $G=G'$ and $P=P'$.  

As explained in Theorem \ref{folkloreg},
 a study of intertwining differential operators between principal series of representations
 of the two groups $G$ and $G'$ can be translated to a study of homomorphisms
between generalized Verma modules 
 of the two Lie algebras ${\mathfrak {g}}'$ and ${\mathfrak {g}}$.  
By the universality
 of the tensor product,
 the latter homomorphisms are characterized by the image of the highest weight vectors
 with respect of the parabolic subalgebra $\gop',$ which are
sometimes referred to as {\it singular vectors}. These vectors are annihilated by the nilradical $\gon_+'.$

 The whole procedure of the F-method to find explicit singular 
vectors may be divided into the following three main steps.

\vskip 2mm
 \noindent
{\bf Step 1.} Computation of the infinitesimal action $d\pi(X)$ for $X\in \gon_+({\mathbb{R}})$
 on a chosen principal series representation of $G$ (in the non-compact picture).

\vskip 2mm 
 \noindent
{\bf Step 2. }
Computation of the dual infinitesimal action $d\pi^\ch(X)$ for $X\in \gon_+({\mathbb{R}})$ on the dual space 
$\lD'_{[o]}(N_-,V^\ch)$
of distributions
on $N_-$ with values in $V^\ch$ with support in $[o].$
This space is isomorphic with $M^{\gog}_{\gop}(V^\ch)$ as $\gog$-modules.

\vskip 2mm
\noindent
{\bf Step 3.}
Suppose now that $N_-$ is commutative and 
 we identify $N_-$ with the Lie algebra ${\mathfrak {n}}_-(\mR)$
 via the exponential map.  
The Fourier transform defines an isomorphism 
$F\otimes \Id_{V^\ch}:\lD'_{[o]}(N_-,V^\ch)\to {\operatorname{Pol}}[\gon_+]
\otimes V^\ch$
and the representation $d\pi^\ch$ induces
  the action $d\tilde{\pi}$ of $\gog$
on ${\operatorname{Pol}}[\gon_+] \otimes V^\ch.$


\subsection{Realization of F-method.}

We shall describe now these three steps in more details.

\vskip 1mm

\noindent
{\bf Step 1.}
The fibration  $p:G\to G/P$ is
 a principal fiber bundle with the structure group $P$ over
 the compact manifold $G/P$. 
The manifold $p(N_-\,P)$ is an open dense subset
 of $G/P$,  
sometimes referred to as the big Schubert cell or the open Bruhat cell of $G/P.$ 
It is naturally identified with $N_-$.
Let $o:= e\cdot P\in M\subset G/P$.  
The exponential map
$$
\phi:\gon_-(\mathbb{R})\to G/P, 
\quad
\phi (X):=\exp(X)\cdot o\in G/P
$$
 gives the natural identification
of the vector space $\gon_-(\mathbb{R})$ with 
the open Bruhat cell $N_-\simeq p(N_-\,P)$.

Let $V$ be an irreducible complex finite-dimensional $P$-module, and let us
 consider the corresponding
induced representation $\pi$ of $G$ on $\Ind_P^G(V)\simeq \lC^\infty(G,V)^P.$
The representation of $G$ on $\Ind_P^G(V)$ will be denoted
 by $\pi$, and the infinitesimal representation $d\pi$
 of its complexified Lie algebra $\gog$
will be considered  in the non-compact picture as follows:   
We identify the space of equivariant smooth maps 
$
\lC^\infty(N_-\,P,V)^P
$
with the space $\lC^\infty(N_-,V)$
 as follows:
A function $f\in\lC^\infty(N_-,V)$ corresponds to $\tilde{f}\in\lC^\infty(N_-\,P,V)^P$
defined by $\tilde{f}(n_-\,p)=p^{-1}\cdot f(n_-),\, n_-\in N_-,\,p\in P.$
 The induced representation $d\pi$ defines by restriction the representation of $\gog$ on the space $\lC^\infty(N_-,V).$

An actual computation of the representation $d\pi(Z), Z\in \gon_+({\mathbb{R}})$ 
can be carried out by the usual scheme:
For a given $Z\in \gon_+({\mathbb{R}})$,  
 we consider the one-parameter subgroup $n(t)=\exp(tZ)\in N_+$ and 
rewrite the product $n(t)^{-1}x$,
 for $x\in N_-$
 and for small $t\in{\mathbb R}$ as 
$$
n(t)^{-1}x=\tilde{x}(t)p(t),\,\tilde{x}(t)\in N_-,\, p(t)\in P.
$$
Then, 
 for $f\in\lC^\infty(N_-,V),$
 we have 
\begin{eqnarray}\label{dpi}
[d\pi(Z)f](x)
=
\frac{d}{dt}\big|_{t=0}(p(t))^{-1}\cdot f(\tilde{x}(t)).
\end{eqnarray}
\noindent
{\bf Step 2.}
A simple way how to describe the dual representation $d\pi^\ch$
on the corresponding generalized Verma module is to use a well-known
identification of generalized Verma modules with the spaces of
distributions supported at the origin of $G/P$
 or on the open Bruhat cell $N_-$.  
It goes as follows.

Let $\lD'_{[o]}(N_-,V^\ch)$ denote the space of $V^\ch$-valued distributions on $N_-$ with 
support in the point $\{o\}.$ The Lie algebra $\gog$
acts on this space by the dual action $d\pi^\ch:$
$$
d\pi^\ch(X)(T)(f)=-T(d\pi(X)(f)), X\in\gog, f\in\lC^\infty(N_-,V).
$$
The action can be extended to the action of $U(\gog)$ by
$$
d\pi^\ch(u)(T)(f)=-T(d\pi(X)(u^o)), X\in\gog, f\in\lC^\infty(N_-,V),
$$
where the map $u\to u^o$ is the antiautomorphism of $U(\gog)$ 
acting as $X\mapsto -X$ on $\gog.$

The space  $\lD'_{[o]}(N_-,V^\ch)$ can be identified with a suitable
generalized Verma module:

\vskip 1mm
\noindent
{\bf Fact 2.2}
The linear map
$$
\phi:U(\gog)\otimes_{U(\gop)}V^\ch\to \lD'_{[o]}(N_-,V^\ch)
$$
determined by
$$
\phi(u\otimes v^\ch):f\mapsto \langle v^\ch, (d\pi(u^o)f)(o) \rangle
$$
is a $U(\gog)$-module isomorphism.

The proof of this claim can be found in \cite{BZ, KP}. 

\vskip 5mm
\noindent
{\bf Step 3.}
Suppose now that the unipotent group $N_-$ is commutative.  
We identify $\gon_-(\mR)$ with $N_-$
 via the exponential map.  The Fourier transform gives an isomorphism
\begin{eqnarray}\label{eqn:FT}
{\mathcal F}: \lD'_{[0]}(\gon_-({\mathbb R}))
 \stackrel{\sim}{\longrightarrow} \Pol[\mathfrak {n}_+],\quad T\mapsto {\mathcal F}(T)
\end{eqnarray}
defined by
$$
{\mathcal F}(T)(\xi)=T_x(e^{i\langle x,\xi\rangle}),\,
\mbox{for}\,\,\, x\in\gon_-({\mathbb R}),\xi\in\gon_+,
$$
where $\langle x,\xi\rangle$ is given by the Killing form on $\gog.$

If we consider the space $\lD'_{[0]}(\gon_-(\mR))$ as a convolution algebra
 with the delta function being the unit, then ${\mathcal F}$ is an algebra isomorphism
 mapping the delta function to the constant polynomial $1.$

 The isomorphism (\ref{eqn:FT}) can be extended to distributions with values 
 in $V^\ch$ by
 $$
 {\mathcal F}\otimes \Id_{V^\ch}:\lD'_{[o]}(N_-,V^\ch)
 \to \Pol[\mathfrak {n}_+]\otimes V^\ch.
 $$
 The Fourier transform $ {\mathcal F}\otimes \Id_{V^\ch}$ can be then used to define the action $d\tilde{\pi}$
of $\gog$ on the space $\Pol[\mathfrak {n}_+]\otimes V^\ch.$
Elements of $\gon_+$ act  by differential operators of second order as ${\gon_+}$ is commutative.
Let ${\mathcal F}^{-1}$ be the inverse Fourier transform, and we set
$\varphi:=\phi^{-1}\circ ({\mathcal F}^{-1}\otimes \mathrm{Id}_{V^\ch})$.
Then $\varphi$ gives a bijection 
\begin{eqnarray}\label{eqn:phiPM}
\varphi: \Pol[\mathfrak {n}_+]\otimes V^\ch
\stackrel{\sim}{\longrightarrow} 
U(\gog)\otimes_{U(\gop)}V^\ch.
\end{eqnarray}

 \subsection{Singular vectors in F-method}

\begin{definition}
Let $V$ be any irreducible finite-dimensional $\gop$-module.
Let us define the $L'$-module
\begin{eqnarray}
M_\gop^\gog(V^\ch)^{\gon'_+}
 :=
 \{v\in M^\gog_\gop(V^\ch):d\pi^\ch (Z)v=0
\textrm{  for any } Z\in \gon'_+\}.
\end{eqnarray}
\end{definition}

 The space $M_\gop^\gog(V^\ch)^{\gon'_+}$ of singular vectors  is a principal object of our interest.
For $G'=G,$  the space $M_\gop^\gog(V^\ch)^{\gon'_+}$  is
 of finite-dimension.
Note that for $G' \subsetneqq G,$ the space $M_\gop^\gog(V^\ch)^{\gon'_+}$ is
 infinite-dimensional
 but it is still completely reducible
 as an ${\mathfrak {l}}'$-module. 
Let us decompose $M_\gop^\gog(V^\ch)^{\gon'_+}$
 into irreducible ${\gol}'$-submodules
 and take $W^{\ch}$ to be one of its irreducible submodules. 
Then we get a $\gog'$-homomorphism from $M_{\gop'}^{\gog'}(W^\ch)$
to $M_\gop^\gog(V^\ch).$ 

If $V$ is a $P$-module, then $M_\gop^\gog(V^\ch)$ carries a $(\gog,P)$-module
structure, and therefore, $M_\gop^\gog(V^\ch)^{\gon_+'}$ becomes an $L'$-module.
In this case, we consider irreducible submodules of $L'$ in 
$M_\gop^\gog(V^\ch)^{\gon_+'}$ for $W^\ch$, and regard $W^\ch$ as a $P'$-module
by letting $N_+'$ act trivially. Then we get a $(\gog',P')$-homomorphism
from $M_{\gop'}^{\gog'}(W^\ch)$ to $M_{\gop}^{\gog}(V^\ch)$ via the canonical 
isomorphisms:
\begin{eqnarray}
\mathrm{Hom}_{\gog',P'}(M_{\gop'}^{\gog'}(W^\ch), M_{\gop}^{\gog}(V^\ch)) &\simeq &
\mathrm{Hom}_{\gop',P'}(W^\ch, M_{\gop}^{\gog}(V^\ch))
\nonumber \\ \label{eqn:WMpV}
&\simeq &
\mathrm{Hom}_{L'}(W^\ch, M_{\gop}^{\gog}(V^\ch)^{\gon_+'}).
\end{eqnarray}
Dually,
 in the language 
of differential
operators, we shall get an invariant differential operator from
$\Ind_P^G(V)$ to $\Ind_{P'}^{G'}(W)$ 
 by Theorem \ref{folkloreg}.  
So the knowledge of all irreducible
summands $W^\ch$ of the $L'$-module $M_\gop^\gog(V^\ch)^{\gon'_+}$ gives the knowledge
of all possible targets $\Ind_{P'}^{G'}(W)$ for equivariant differential operators on $\Ind_P^G(V).$

In the F-method,
 we then realize the space $M_\gop^\gog(V^\ch)^{\gon'_+}$  in the space of polynomials on 
$\gon_+$
 with values in $V^\ch$ with action  $d\tpi.$
 It can be done efficiently using the Fourier transform
 as follows:  
 
\begin{definition}  
We define
\begin{align}
 \mbox{\rm Sol}\nonumber
 \equiv& \mbox{\rm Sol}(\mathfrak{g},\mathfrak{g}';V^{\ch})
\\ \label{eqn:sol2l}
 :=&\{f \in \Pol[\mathfrak {n}_+]\otimes V^\ch :
d\tpi(Z) f = 0 \textrm{ for any } Z\in\gon'_+ 
\}.
\end{align}
\end{definition}

The inverse Fourier transform gives
 an $L'$-isomorphism
\begin{equation}
\label{eqn:phi}
\varphi: \mbox{\rm Sol}(\mathfrak{g},\mathfrak{g}';V^{\ch})
 \overset \sim \to
 M_{{\mathfrak {p}}}^{{\mathfrak {g}}}(V^{\ch})^{{\mathfrak {n}}_+'}.  
\end{equation}
An explicit form
of the action $d\tpi(Z)$ leads to 
a (system  of) differential equation for elements in $ \mbox{\rm Sol}$
which makes it possible to describe its structure completely
 in some particular cases of interest. 
We shall see in Sections \ref{sec:4} and \ref{sec:5}
 in certain settings 
the full understanding of the structure of the set $ \mbox{\rm Sol}$
as an $L'$-module
gives complete classification of $\gog'$-homomorphisms
from $M^{\gog'}_{\gop'}(W^\ch)$ to $M^\gog_\gop(V^\ch)$
and answers Problems \ref{prob:A} and \ref{prob:B}.

The transition from $M_\gop^\gog(V^\ch)^{\gon'_+}$ to $ \mbox{\rm Sol}({\mathfrak {g}}, {\mathfrak {g}}';V^{\vee})$ is the key point of the F-method. 
It transforms the algebraic problem
 of finding singular vectors
 in generalized Verma modules 
 (Problem \ref{prob:A}) into an analytic problem
 of solving certain differential equations.
It turns out
 that the F-method is often more efficient
 than other existing algebraic methods
 in finding singular vectors.
Furthermore,
 the F-method clarifies why the combinatorial formula
 appearing in the coefficients
 of intertwining differential operators
 in the example of Juhl \cite[Chapter 5]{ju}
 are related to those of the Gegenbauer polynomials.  
It also reduces substantially the amount of computation needed
 and gives a complete description of the set
of singular vectors
 (e.g. Theorem \ref{basis});
 finally it offers a systematic and effective tool for investigation of singular vectors in
many cases
 (e.g. Theorem \ref{T.3.8}). 
It will be illustrated 
in Sections 4 and 5
 of this article
 as well as
 in the second part of the series
 with a series of different examples.
 
\section{Discretely decomposable branching laws}
\label{sec:3}

\numberwithin{equation}{section}
\setcounter{equation}{0}

Suppose that $\gog\supset \gog'$ are a pair
of complex reductive Lie algebras, and that $X$ is  an irreducible
$\gog$-module. It often happens that the restriction $X|_{\gog'}$
does not contain any irreducible
$\gog'$-module (\cite{KM,KP}). On the other hand, for $X=M^{\gog}_{\gop}(V^\ch)$
belonging to the parabolic BGG category $\lO^{\gop}$ (see below), the restriction $X|_{\gog'}$ needs to contain some irreducible $\gog'$-module for the existence of nonzero $G'$-equivariant differential operators by  Theorem \ref{folkloreg}. This algebraic property is said to be \textquotedbl{discretely decomposable restrictions}\textquotedbl\  in representation theory, which gives a certain constraint on the triple $(\gog,\gog',\gop)$
(Proposition \ref{prop:3.1}).

Further, 
 the uniqueness (up to scaling) of equivariant differential operators
 to submanifolds
 is assured if (\ref{eqn:WMpV}) is one-dimensional, or if the branching law of the restriction
$X|_{\gog'}$ is multiplicity free, by Theorem \ref{folkloreg}.  

In this section we fix notation for the parabolic BGG
 category ${\mathcal{O}}^{\mathfrak{p}}$,
 and summarize the algebraic framework 
 on discretely decomposable restrictions
 and multiplicity-free theorems
 in branching laws
 that were established in \cite{kob,k1,k2}.  
These algebraic results are 
 a guiding principle
 in this current article
 and in the second part \cite{KOSS2}
 of the series for finding appropriate settings
 in parabolic geometry,
 where one could expect to 
 obtain explicit formulas
 of equivariant   differential operators.

\subsection{Category ${\mathcal {O}}$
 and ${\mathcal{O}}^{\mathfrak {p}}$}
\label{subsec:oop}
We begin with a quick review of the parabolic BGG category
$\mathcal{O}^{\mathfrak{p}}$
 (see \cite{H} for an introduction to this area).   

Let ${\mathfrak {g}}$ be a semisimple Lie 
 algebra over $\mathbb C$, 
 and ${\mathfrak {j}}$ a Cartan subalgebra.  
We write $\Delta \equiv \Delta({\mathfrak {g}},
{\mathfrak {j}})$
 for the root system, 
 ${\mathfrak {g}}_{\alpha}$
 ($\alpha \in\Delta$) for the root space, 
 and $\alpha^{\ch}$ for the coroot,
 and 
 $W\equiv W({\mathfrak {g}})$
 for the Weyl group for the root system 
 $\Delta({\mathfrak {g}}, {\mathfrak {j}})$.  
We fix a positive system $\Delta^+$, 
 write  $\rho\equiv \rho({\mathfrak{g}})$ for 
 half the sum of the positive roots,  
 and define a Borel subalgebra
 ${\mathfrak {b}}={\mathfrak {j}}
+{\mathfrak {n}}$
 with nilradical  ${\mathfrak {n}}
:= \oplus_{\alpha\in \Delta^+}
{\mathfrak {g}}_{\alpha}$.   
The Bernstein--Gelfand--Gelfand category
 ${\mathcal {O}}$ 
 (BGG category for short) is defined
 to be the full subcategory 
 of ${\mathfrak {g}}$-modules
 whose objects are finitely generated ${\mathfrak {g}}$-modules $X$
 such that 
 $X$ are ${\mathfrak {j}}$-semisimple 
 and locally ${\mathfrak {n}}$-finite \cite{BGG}.

Let $\mathfrak{p}$ be a parabolic subalgebra
 containing ${\mathfrak {b}}$,
and ${\mathfrak{p}} = {\mathfrak{l}} + {\mathfrak{n}}_+$
 its Levi decomposition 
 with ${\mathfrak{j}} \subset {\mathfrak{l}}$.
We set 
$
     \Delta^+({\mathfrak{l}})
   :=\Delta^+ \cap \Delta({\mathfrak{l}}, {\mathfrak{j}}),
$
and define
\[
{\mathfrak {n}}_-({\mathfrak {l}})
  :=\sum_{\alpha\in \Delta^+ ({\mathfrak {l}})}
  {\mathfrak {g}}_{-\alpha}.
\]
The parabolic BGG category 
${\mathcal{O}}^{\mathfrak {p}}$
 is the full subcategory 
 of ${\mathcal{O}}$
 whose objects $X$ are locally 
 ${\mathfrak {n}}_-({\mathfrak {l}})$-finite.  
We note that ${\mathcal{O}}^{\mathfrak{b}}={\mathcal{O}}$
 by definition.

The set of $\lambda$ for 
 which $\lambda|_{{\mathfrak{j}} \cap [{\mathfrak {l}},{\mathfrak {l}}]}$
 is dominant integral is denoted by 
\[
  \Lambda^+({\mathfrak {l}})
:=
 \{\lambda  \in {\mathfrak {j}}^*
  :
   \langle \lambda,\alpha^\ch\rangle
   \in {\mathbb{N}}
   \text{ for all }\alpha\in \Delta^+({\mathfrak {l}})\}. 
\]
We write $F_\lambda$ for the finite-dimensional simple ${\mathfrak{l}}$-module
 with highest weight $\lambda$,
 inflate $F_\lambda$ to a ${\mathfrak {p}}$-module
 via the projection 
$
   {\mathfrak {p}}\to {\mathfrak {p}}/{\mathfrak {n}}_+\simeq {\mathfrak {l}},
$ 
 and define the generalized Verma module by 
\begin{equation}
\label{eqn:NGpF}
  M_{\mathfrak {p}}^{\mathfrak{g}} (\lambda) 
  \equiv M_{\mathfrak {p}}^{\mathfrak {g}}({F_{\lambda}})
:= U({\mathfrak {g}}) \otimes_{U({\mathfrak{p}})} F_{\lambda}.  
\end{equation}
Then $M_{\mathfrak {p}}^{\mathfrak {g}}({\lambda})\in
{\mathcal{O}}^{\mathfrak {p}}$,
and any simple object in $\mathcal{O}^{\mathfrak {p}}$ is the quotient
of some $M_{\mathfrak {p}}^{\mathfrak {g}}({\lambda})$.
We say $M_{\mathfrak {p}}^{\mathfrak {g}}({\lambda})$ is
 of {\it{scalar type}}
 if $F_{\lambda}$ is one-dimensional, or equivalently,
 if $\langle \lambda, \alpha^\ch\rangle =0$
 for all $\alpha \in \Delta({\mathfrak {l}})$.

If $\lambda \in \Lambda^+({\mathfrak {l}})$
 satisfies 
\begin{equation}
\label{eqn:anti}
   \langle \lambda + \rho , \beta^\ch\rangle \not\in {\mathbb{N}}_+
  \text{ for all } \beta \in \Delta^+ \setminus \Delta({\mathfrak {l}}), 
\end{equation}
then $M_{\mathfrak {p}}^{\mathfrak {g}}({\lambda})$ is simple, 
 see \cite{BD}.

Let ${\mathfrak {Z}}({\mathfrak {g}})$ be the center
 of the enveloping algebra $U({\mathfrak {g}})$, 
and we parameterize maximal ideals of 
 ${\mathfrak {Z}}({\mathfrak {g}})$
 by the Harish-Chandra isomorphism:  
\[
   {\operatorname{Hom}}_{{\mathbb{C}}{\text{-alg}}}
   ({\mathfrak {Z}}({\mathfrak {g}}), {\mathbb{C}})
   \simeq {\mathfrak {j}}^*/W,
   \quad
   \chi_{\lambda} \leftrightarrow \lambda.  
\]
In our normalization, the trivial one-dimensional representation
has a ${\mathfrak {Z}}({\mathfrak {g}})$-infinitesimal
character $\rho \in {\mathfrak {j}}^*/W$.
Then the generalized Verma module $M_{\mathfrak {p}}^{\mathfrak {g}}(\lambda)$
has a ${\mathfrak {Z}}({\mathfrak {g}})$-infinitesimal 
character
$
    \lambda + \rho \in {\mathfrak {j}}^*/W.  
$

We denote by ${\mathcal{O}}^{\mathfrak {p}}_{\lambda}$
 the full subcategory of ${\mathcal{O}}^{\mathfrak {p}}$
 whose objects have generalized ${\mathfrak {Z}}({\mathfrak {g}})$-infinitesimal 
characters $\lambda \in {\mathfrak {j}}^{\ast}/W$, 
 namely,
\[
  {\mathcal{O}}_{\lambda}^{\mathfrak {p}}
  =\bigcup_{n=1}^{\infty}
   \{X \in {\mathcal{O}}^{\mathfrak {p}}:
     (z-\chi_{\lambda}(z))^n v=0
    \quad\text{for any }v \in X
    \text{ and }z \in  {\mathfrak {Z}}({\mathfrak {g}})
   \}.  
\]
Any ${\mathfrak {g}}$-module
 in ${\mathcal{O}}^{\mathfrak {p}}$
 is a finite direct sum of ${\mathfrak {g}}$-modules
 belonging to some ${\mathcal{O}}_{\lambda}^{\mathfrak {p}}$.  
Let $K({\mathcal{O}}_{\lambda}^{\mathfrak {p}})$ be 
 the Grothendieck group 
 of ${\mathcal{O}}_{\lambda}^{\mathfrak {p}}$,
 and set 
\[
     K({\mathcal{O}}^{\mathfrak {p}})
     :=
     {\prod_{\lambda \in {\mathfrak {j}}^*/W}}^{\!\!\!\!\prime}
     K({\mathcal{O}}_{\lambda}^{\mathfrak {p}}), 
\]
where ${\prod}'$ denotes the direct product
 for which the components are zero 
 except for countably 
many constituents.  
Then $K({\mathcal{O}}^{\mathfrak {p}})$
 is a free ${\mathbb{Z}}$-module
 with basis elements $\operatorname{Ch}(X)$
 in one-to-one correspondence with simple modules $X \in {\mathcal{O}}^{\mathfrak {p}}$.  
We note that $K({\mathcal{O}}^{\mathfrak {p}})$
 allows a formal sum of countably many $\operatorname{Ch}(X)$, 
 and contains the Grothendieck group
 of ${\mathcal{O}}_{\lambda}^{\mathfrak {p}}$ as a subgroup.

\subsection{Discretely decomposable 
 branching laws for ${\mathcal {O}}^{\mathfrak {p}}$}

Retain the notation of Section \ref{subsec:oop}.  
Let ${\mathfrak {g}}'$ be a reductive subalgebra
 of ${\mathfrak {g}}$. 
Our subject here is
 to understand the ${\mathfrak {g}}'$-module
 structure of a ${\mathfrak {g}}$-module
 $X \in {\mathcal {O}}^{\mathfrak {p}}$,
 to which we simply refer 
 as the {\it{restriction}}
 $X|_{{\mathfrak {g}}'}$.  
We allow the case 
 where ${\mathfrak {g}}'$
 is not of maximal rank in ${\mathfrak {g}}$.  
This question might look easy
 in the category ${\mathcal{O}}$
 at first glance,
 however, 
 the restriction $X|_{{\mathfrak {g}}'}$
 behaves surprisingly
 in a various
 (and sometimes \textquotedbl{wild}\textquotedbl) manner
 even when $({\mathfrak {g}}, {\mathfrak {g}}')$
 is a reductive symmetric pair.  
In particular,
 it may well happen 
 that the restriction $X|_{{\mathfrak {g}}'}$
 does not contain any simple module of ${\mathfrak {g}}'$,
 which may be considered as the effect  
 of \textquotedbl{hidden continuous spectrum}\textquotedbl\
 (see \cite{k2}).  

In order to exclude \textquotedbl{hidden continuous spectrum}\textquotedbl , 
 we introduced the following notion:
\begin{definition}
[{\cite[Definition 1.1.]{kob3}}]
\label{def:3.1}
A ${{\mathfrak {g}}'}$-module $X$
 is {\it{discretely decomposable}} 
 if there exists an increasing sequence
 of ${\mathfrak {g}}'$-modules
 $X_j$ of finite length ($j \in {\mathbb{N}}$)
 such that $X=\cup_{j=0}^{\infty} X_j$.  
\end{definition}

For an irreducible ${\mathfrak {g}}$-module $X$, 
 the restriction $X|_{{\mathfrak {g}}'}$ contains
 an irreducible ${\mathfrak {g}}'$-module
 if and only if 
 the restriction $X|_{{\mathfrak {g}}'}$ is
 discretely decomposable
 \cite[Sect.1]{kob3}.  
Applying this to $X=M^\gog_\gop(V^\ch)$, we see from Theorem 
\ref{folkloreg} that the concept of 
 \textquotedbl{discretely decomposable restrictions}\textquotedbl\
 exactly guarantees the existence
 of our main objects,
 namely equivariant differential operators
 between two real flag varieties.  

We then ask 
 for which triple 
$
     {\mathfrak {g}}' \subset {\mathfrak {g}}
     \supset {\mathfrak {p}}
$ 
 the restriction $X|_{\mathfrak {g}'}$
 is discretely decomposable
 as a ${\mathfrak {g}}'$-module
 for any $X \in {\mathcal{O}}^{\mathfrak {p}}$.  
A criterion for this was established
 in \cite{k2}
 by using an idea of ${\mathcal{D}}$-modules
 as follows:
Let $G$ be the group $\operatorname{Int}(\gog)$
 of inner automorphisms of ${\mathfrak {g}}$, 
$P\subset G$ the parabolic subgroup of $G$ with Lie algebra 
$\gop$, 
 and $G^\prime\subset G$ a reductive subgroup with Lie algebra 
$\gog^\prime\subset\gog$.

\begin{proposition} 
\label{prop:3.1}
If $G^\prime P$ is  closed in $G,$ then the restriction 
$X|_{\gog^\prime}$ is discretely decomposable for any 
$X\in \sh^\gop$.
The converse statement also holds
 if $(G,G')$ is a symmetric pair.  
\end{proposition}
\begin{proof}
See \cite[Proposition 3.5
 and Theorem 4.1]{k2}.   
\end{proof}

Let us consider a simple sufficient condition
 for the closedness of $G'P$ in $G$, 
 which will be fulfilled
 in all the examples discussed in this article. To that aim, 
let $E$ be a hyperbolic element of $\gog$
 defining a parabolic subalgebra
 $\gop(E)=\gol(E)+\gon(E)$, 
namely, 
 ${\mathfrak {l}}(E)$ and ${\mathfrak {n}}(E)$
 are the sum of eigenspaces
 of $\operatorname{ad}(E)$
 with zero, positive eigenvalues.  

\begin{definition}
[{\cite[Definition 3.7]{k2}}]
\label{def:3.2}
A parabolic subalgebra $\gop$
 is $\gog^\prime$-{\it{compatible}} if there exists a
hyperbolic element $E^\prime\in\gog^\prime$ such that $\gop=\gop(E^\prime)$. 
\end{definition}

If $\gop=\gol +\gon_+$ is $\gog^\prime$-compatible, then $\gop^\prime :=\gop\cap\gog^\prime$
becomes a parabolic subalgebra of $\gog^\prime$ with 
the following Levi decomposition:
$$
\gop^\prime =\gol^\prime +\gon_+^\prime :=(\gol\cap\gog^\prime)+(\gon_+\cap\gog^\prime), 
$$
and $P^\prime:=P \cap G^{\prime}$
 becomes a parabolic subgroup of $G^{\prime}$.  
Hence,  
 $G^{\prime}/P^{\prime}$ becomes automatically 
 a closed submanifold of $G/P$,
 or equivalently,
 $G' P$ is closed in $G$.  
Here is a direct consequence of Proposition \ref{prop:3.1}:  

\begin{proposition}[{\cite[Proposition 3.8]{k2}}]
If $\gop$ is $\gog^\prime$-compatible, 
then the restriction $X|_{\gog^\prime}$ is discretely decomposable for any 
$X\in \sh^\gop$, and any $X_j$ in Definition \ref{def:3.1} belongs to the
parabolic BGG category ${\fam2 O}^{\gop'}$ for $\gog'$-modules.
\end{proposition}

Let $\gop$ be a $\gog^\prime$-compatible parabolic subalgebra,
 and keep the above notation. 
We denote by $F^\prime_\mu$ a finite-dimensional simple $\gol^\prime$-module with highest weight 
$\mu\in\Lambda^+(\gol^\prime)$. 
The $\gol^\prime$-module structure 
on the opposite nilradical $\gon_-$ descends to $\gon_-/(\gon_-\cap \gog^\prime)$, 
and consequently extends to the symmetric tensor algebra 
 $S(\gon_-/(\gon_-\cap \gog^\prime))$. 
We set 
$$
m(\lambda ,\mu):=\dim \Hom_{\gol^\prime}(F^\prime_\mu, F_\lambda|_{\gol^\prime}\otimes 
S(\gon_-/(\gon_-\cap \gog^\prime))).
$$
The following identity is a key step 
 to find branching laws
 (in a generic case) 
 for the restriction $X|_{{\mathfrak {g}}'}$
 for $X \in \sh^\gop$: 

\begin{theorem}[{\cite[Proposition 5.2]{k2}}]
\label{groth}
 Suppose that $\gop =\gol +\gon_+$ is 
a $\gog^\prime$-compatible parabolic subalgebra of $\gog$, 
 and $\lambda\in\Lambda^+(\gol)$.
Then 
\begin{enumerate}
\item[{\rm{1)}}]
$m(\lambda ,\mu)<\infty$ for all $\mu\in\Lambda^+(\gol^\prime)$.
\item[{\rm{2)}}]
We have the following identity in $K(\sh^\gop)$: 
$$
M^\gog_\gop(\lambda)|_{\gog^\prime}\simeq \bigoplus_{\mu\in\Lambda^+(\gol^\prime)}
m(\lambda ,\mu)M^{\gog^\prime}_{\gop^\prime}(\mu)
$$
\end{enumerate}
for any generalized Verma modules
 $M^\gog_\gop(\lambda)$ and $M^{\gog^\prime}_{\gop^\prime}(\mu)$
defined respectively by 
$
M^\gog_\gop(\lambda)=U(\gog)\otimes_{U(\gop)}F_\lambda ,\,\, 
M^{\gog^\prime}_{\gop^\prime}(\mu)=U(\gog^\prime)\otimes_{U(\gop^\prime)}F^\prime_\mu.
$
\end{theorem}

Finally,
 we highlight the multiplicity-free case, 
 namely, 
 when $m(\lambda, \mu)\le 1$
 and give a closed formula
 of branching laws.  
Suppose now that ${\mathfrak {p}} = \mathfrak{l} + \mathfrak{n}_+$ 
is a parabolic subalgebra
such that the nilradical $\mathfrak{n}_+$ is abelian.
We write
$\mathfrak{g} = \mathfrak{n}_- + \mathfrak{l} + \mathfrak{n}_+$
for the Gelfand--Naimark decomposition.  
Let $\theta$ be an endomorphism of ${\mathfrak {g}}$
 such that $\theta|_{{\mathfrak {l}}}=\operatorname{id}$
 and $\theta|_{{\mathfrak {n}}_+ + {\mathfrak {n}}_-}=-\operatorname{id}$.  
Then $\theta$ is an involutive automorphism
 of ${\mathfrak {g}}$
 because ${\mathfrak {n}}_+$
 is abelian.  

Suppose $\tau$ is another involutive automorphism
 of the complex Lie algebra ${\mathfrak {g}}$ 
 such that $\tau {\mathfrak {l}}={\mathfrak {l}}$
 and $\tau {\mathfrak {n}}_{\pm}={\mathfrak {n}}_{\pm}$.  
Then $\tau \theta = \theta \tau$
 and the parabolic subalgebra
 ${\mathfrak {p}}$ is ${\mathfrak {g}}^{\tau}$-compatible.  
We take a Cartan subalgebra $\mathfrak{j}$ of $\mathfrak{l}$
 such that $\mathfrak{j}^\tau$
is a maximal abelian subspace of ${\mathfrak {l}}^{\tau}$.  
Here, 
 for a subspace $V$ in ${\mathfrak {g}}$, 
 we write 
$
     V^{\pm \tau}:=\{v \in V: \tau v = \pm v\}
$
 for the $\pm 1$ eigenspaces of $\tau$, 
respectively.  
Then 
\[
{\mathfrak {g}}^{\tau\theta}:= {\mathfrak {l}}^{\tau}
                  + {\mathfrak {n}}_-^{-\tau}
                  + {\mathfrak {n}}_+^{-\tau}
\]
is a reductive subalgebra of 
 ${\mathfrak {g}}$.  
We write ${\mathfrak {g}}^{\tau\theta} = \bigoplus_i {\mathfrak {g}}_i^{\tau \theta}$
 for the decomposition 
 into simple or abelian ideals,
 and decompose 
 ${\mathfrak {n}}_{-}^{-\tau}=\bigoplus_i{\mathfrak {n}}_{-,i}^{-\tau}$
 correspondingly.  
Each ${\mathfrak {n}}_{-,i}^{-\tau}$
 is a ${\mathfrak {j}}^{\tau}$-module,
 and we denote by 
$\Delta({\mathfrak {n}}_{-,i}^{-\tau}, {\mathfrak {j}}^{\tau})$
 the set of weights of ${\mathfrak {n}}_{-,i}^{-\tau}$
 with respect to ${\mathfrak {j}}^{\tau}$.  
(We note that
 ${\mathfrak {n}}_{-,i}^{-\tau}=\{0\}$
 except for a single $i$
 in the case where we shall treat,
 and thus we may simply replace
 ${\mathfrak {n}}_{-,i}^{-\tau}$
 by ${\mathfrak {n}}_{-}^{-\tau}$ below
 for actual computations below.)

The roots $\alpha$ and $\beta$ are said to be 
 {\it{strongly orthogonal}}
 if neither $\alpha + \beta$ nor $\alpha - \beta$ is a root.  
We take a maximal set of strongly orthogonal roots 
$\{\nu_1^{(i)}, \cdots, \nu_{k_i}^{(i)}\}$ 
 in  $\Delta({\mathfrak{n}}_{-,i}^{-\tau}, {\mathfrak{j}}^{\tau})$ 
 inductively as follows:
\begin{enumerate}
\item[1)]
$\nu_1^{(i)}$ is the highest root
 of $\Delta({\mathfrak {n}}_{-,i}^{-\tau}, {\mathfrak {j}}^{\tau})$.  
\item[2)]
 $\nu_{j+1}^{(i)}$ is the highest root among the elements
 in $\Delta({\mathfrak{n}}_{-,i}^{-\tau}, {\mathfrak{j}}^{\tau})$
 that are strongly orthogonal to $\nu_1^{(i)}, \cdots, \nu_{j}^{(i)}$
 ($1 \le j \le k_i-1$).   
\end{enumerate}

Then we recall from \cite{k1,k2}
 the multiplicity-free branching law:

\begin{theorem}
\label{thm:3.5}
Suppose that ${\mathfrak {p}}$, 
  $\tau$, 
 and $\lambda$ are as above.  
Then the generalized Verma module
 $M_{\mathfrak {p}}^{\mathfrak {g}}({\lambda})$
 decomposes into a multiplicity-free direct sum of generalized Verma modules
 of ${\mathfrak {g}}^{\tau}$
 in $K({\mathcal{O}}^{\mathfrak {p}})$: 
\begin{equation}\label{eqn:mfbr}
  M_{\mathfrak {p}}^{\mathfrak {g}}(\lambda)|_{\mathfrak {g}^{\tau}}
  \simeq
  \bigoplus 
  M_{\mathfrak {p}^{\tau}}^{\mathfrak {g}^{\tau}}
  (\lambda|_{\mathfrak {j}^{\tau}}+\sum_i \sum_{j=1}^{k_i} a_j^{(i)} \nu_j^{(i)}).  
\end{equation}
Here the summation is taken
 over the following subset
 of ${\mathbb{N}}^k$
 ($k=\sum k_i$) defined by 
\[
\prod_i A_i, 
\qquad
 A_i:=\{(a_j^{(i)})_{1 \le j \le k_i} \in {\mathbb{N}}^{k_i}
:
a_1^{(i)} \ge \cdots \ge a_{k_i}^{(i)} \ge 0
\}.  
\]
The restriction 
$M_{\mathfrak {p}}^{\mathfrak {g}}(\lambda)|_{\mathfrak {g}^{\tau}}$
 is actually
 a direct sum in the parabolic BGG category
 ${\mathcal{O}}^{\mathfrak {p}}$
 if $\lambda$ is sufficiently negative,
 or more generally, 
 if the following two conditions are satisfied
 for $\{a_{j}^{(i)}\}$ in the above range: 
\begin{align}
&\langle \lambda|_{\mathfrak {j}^{\tau}}+\rho({\mathfrak {g}}^{\tau})
+
\sum_i \sum_{j=1}^{k_i}a_{j}^{(i)}\nu_{j}^{(i)}, 
\beta^{\ch}\rangle  \notin {\mathbb{N}}_+
\quad \text{for all }\beta \in \Delta({\mathfrak {n}}_+^{\tau}, {\mathfrak {j}}^{\tau}),
\label{eqn:36a}
\\
&  \lambda|_{{\mathfrak{j}}^{\tau}} 
  +
  \rho({{\mathfrak{g}}^{\tau}})
  + 
\sum_i \sum_{j=1}^{k_i}a_{j}^{(i)}\nu_{j}^{(i)}
\text{
 are all distinct in 
 $({\mathfrak{j}}^{\tau})^*/W({\mathfrak{g}}^{\tau})$.  }
\label{eqn:36b}
\end{align}

\end{theorem}

\begin{proof}
The formula \eqref{eqn:mfbr} was proved 
 in \cite[Theorem 8.3]{k1}
 (in the framework of holomorphic discrete series representations)
 and in \cite[Theorem 5.2]{k2}
 (in the framework of generalized Verma modules)
 under the assumption that $\lambda$ is sufficiently negative
 and that ${\mathfrak {g}}^{\tau \theta}$ is simple.  
The latter proof shows in fact 
 that the identity \eqref{eqn:mfbr} holds 
 in $K({\mathcal{O}}^{\mathfrak {p}})$ for all $\lambda$.  
Since two modules
 with different infinitesimal characters
 do not have extension,
 the last statement follows.  
\end{proof}
\begin{remark}
In the case ${\mathfrak {l}}={\mathfrak {g}}^{\tau}$, 
 each summand of the right-hand side
 is finite-dimensional
 because the parabolic subalgebra
 ${\mathfrak {p}}^{\tau}$
 coincides with ${\mathfrak {g}}^{\tau}$.  
In this very special case,
 Theorem \ref{thm:3.5} for sufficiently negative $\lambda$
 was proved earlier by B. Kostant and W. Schmid.  
We also note that Theorem \ref{thm:3.5} includes the decomposition 
 of the tensor product of two representations for generic parameters because
 the pair $({\mathfrak{g}} \oplus {\mathfrak{g}}, 
\operatorname{diag}({\mathfrak{g}}))$
 is regarded as an example
 of a symmetric pair. 
\end{remark}


\section{Conformal geometry with arbitrary signature}
\label{sec:4}
 \numberwithin{equation}{section}
 \setcounter{equation}{0}
In Sections \ref{sec:4} and \ref{sec:5}, 
 we illustrate the F-method
 by examples
 of pseudo-Riemannian manifolds
 with arbitrary signatures,
 whose symmetries are given by
\[
(G,G')=(SO_o(p,q),SO_o(p,q-1))
\,\,\text{ or }\,\,(Spin_o(p,q),Spin_o(p,q-1)).  
\]

In answer to the Problem \ref{prob:A} posed in Introduction,
 we see 
 that the F-method brings us to the Gegenbauer differential equation (\ref{eqn:Gdiff})
 in this case,
 and prove that all singular vectors
 can be described
 by using the classical orthogonal polynomials.  
In turn, 
 these orthogonal polynomials yield
 a generalization
 of Juhl's equivariant differential operators
 between sections of line bundles
 over two conformal manifolds (of different dimensions), 
 of which the original form  
 was constructed by completely different
 (combinatorial) techniques in \cite{ju}.

Concerning Problem \ref{prob:B} in Introduction, by using
the explicit singular vectors obtained
 by the F-method, we can determine the Jordan--H{\"o}lder series
 of the generalized Verma modules
 for $\gog$ with exceptional (discrete) values of parameters,
 when we restrict them to the reductive subalgebra
 ${\mathfrak {g}}'$,  
see Theorem \ref{T.3.8}.  

\subsection{Notation}\label{3.1}
Let $p\geq 1$ and $q\geq 2$.
We set  $n=p+q-2$ and
\[
\epsilon_i:=
\begin{cases}
1 
&
(1 \le i \le p-1), 
\\
-1
\qquad
&(p \le i \le n).  
\end{cases}
\]
We note $\epsilon_n=-1$
 because $q \ge 2$.  
Let us consider the quadratic form 
\begin{equation}
\label{eqn:pq}
2x_0\,x_{n+1}+\sum_{i=1}^{n}\epsilon_i x_i^2 \quad\text{for}\quad x=(x_0,\cdots ,x_{n+1}).
\end{equation}
on ${\mathbb{R}}^{p+q}\simeq\mathbb{R}^{n+2}$, 
 and set $G:=SO_o(p,q)$, 
 the identity component
 of the group preserving the
 quadratic form.  
Then the group $G$ preserves
 the null cone 
\begin{align}
\lN=\lN_{p,q}:=
\{ x=(x_0,\cdots ,x_{n+1})\in\mR^{p,q}\setminus \{0\}: 2x_0\,x_{n+1}+\sum_{i=1}^{n}\epsilon_i x_i^2=0
\}.
\end{align}
We define the parabolic subgroups 
 $P$ and $P_-$ to be the isotropy subgroups
of the line in the null cone $\lN$
 generated by $e_0={}^t\!(1,0,\ldots,0)$ and 
 $e_{n+1}={}^t\!(0,\ldots,0, 1)$, 
 respectively,
 and set $L := P \cap P_-$.
The homogeneous space $G/P$ is the  projective null cone $\mP\lN$ 
with its conformal structure.
We write $P=LN_+ = MAN_+$ for the  Langlands decomposition.
Let $\{E_j\}_{j=1,\cdots,n}$
 be the standard basis of the Lie algebra ${\mathfrak {n}}_+(\mathbb{R})$ of $N_+$,
 which we identify with ${\mathbb{R}}^n$, 
 and likewise ${\mathfrak {n}}_-(\mathbb{R})$
 with ${\mathbb{R}}^n$
 by using the standard basis:
 \begin{eqnarray}\label{eqn:Ncoord}
 \gon_+(\mathbb{R})\simeq \{Z:Z=(z_1,\ldots,z_n)\},
\quad
 \gon_-(\mathbb{R})\simeq\{X:{}^t\! X=(x_1,\ldots,x_n)\}.  
 \end{eqnarray}

The group $M$ is isomorphic to $SO(p-1,q-1)$,
 and acts on ${\mathfrak {n}}_+(\mathbb{R}) \simeq {\mathbb{R}}^n
 (={\mathbb{R}}^{p+q-2})$ 
 as the natural representation,
 preserving the quadratic form 
$\sum_{i=1}^{n}\epsilon_ix_i^2$.  
Denote by $\mJ$ the $n \times n$ matrix
 of this quadratic form with elements $\epsilon_i$ on the diagonal.

Elements in $G$ can be  written as block matrices with respect to the 
direct sum decomposition
\begin{equation}
\label{vectrac}
     \mR^{n+2}=\mR e_0\oplus 
\sum_{j=1}^{n}{\mathbb{R}}e_j
\oplus\mR e_{n+1}.  
\end{equation}
Then elements in the real parabolic subgroup $P$ are given by block triangular matrices
\begin{equation}
\label{eqn:pma}
  p=
\left(
\begin{array}{ccc}
\epsilon(m) a& \star & \star \\
0 & m & \star\\
0 & 0 & \epsilon(m) a^{-1} 
\end{array}
\right)
\end{equation}
with $a\in\mR_+,$ $m \in SO(p-1,q-1).$ Here $\epsilon(m)=+1$
 or $-1$
according to whether $m$ belongs to the identity component $SO_o(p-1,q-1)$
 or not.   
In the coordinates (\ref{eqn:Ncoord}) we have 
\begin{eqnarray}\label{2matrices}
n=\exp Z=
\begin{pmatrix}
1&Z&-\frac{|Z|^2}{2}\\
0&\Id&-\mJ {}^t\! Z\\
0&0&1
\end{pmatrix}\in N_+ ,
\,\,
x=\exp X=\begin{pmatrix}
1&0&0\\
X&\Id&0\\
-\frac{|X|^2}{2}&-{}^t\! X\mJ&1
\end{pmatrix}
\in N_-,
\end{eqnarray}
where we set $|X|^2:={}^t\! X\mJ X$  and $|Z|^2:=Z\mJ {}^t\! Z.$

\subsection{The representations $d\pi_\la$ and $d\tilde{\pi}_\la.$}
\label{3.2}
We are going now to apply the F-method
 explained in Section \ref{sec:2 correction} to the conformal case
of the general signature
 $(p,q)$, 
\[
  (G,G')=(SO_o(p,q), SO_o(p,q-1)).  
\]
The first goal is to describe the action of elements in  $\gon_+$ in terms of
differential operators acting
 on the \textquotedbl{Fourier image}\textquotedbl\,  of the generalized Verma module.
This can be deduced from the explicit form of the (easily described)
action of the induced representation in the non-compact picture.
We shall then find singular vectors in $M^\gog_\gop(\mC_\lambda)$
by using the F-method. Later on, we use them to obtain equivariant 
differential operators from $\Ind^G_P(\mC_\lambda)$ to $\Ind^{G'}_{P'}(\mC_{\lambda+K})$
 for some $K \in {\mathbb{N}}$
after switching $\lambda$ to $-\lambda$ corresponding to the dual representation.

For $\lambda \in {\mathbb{C}}$, 
 we define a family of differential operators
 on ${\mathfrak {n}}_-(\mathbb{R}) \simeq {\mathbb{R}}_x^n$
 by 
$$
Q_j(\la):=-\frac{1}{2} \epsilon_j |X|^2\partial_{x_j}
         +x_j(\lambda +\sum_kx_k\partial_{x_k}),\, j=1,\ldots, n, 
$$
and on its dual space ${\mathbb{R}}_{\xi}^n$
 by 
\begin{equation}
\label{dpti}
P_j(\lambda) := -i\left(\frac{1}{2}\epsilon_j\xi_j\square + (\lambda-E)\partial_{\xi_j}\right),\, j=1,\dots ,n,
\end{equation}
where
$$\square :=\partial^2_{\xi_1}+\dots +\partial^2_{\xi_{p-1}}-\partial^2_{\xi_{p}}-\dots -\partial^2_{\xi_{p+q-2}}$$
is the Laplace--Beltrami operator of signature $(p-1,q-1)$
 and  $E=\sum_{k=1}^{n}\xi_k\partial_{\xi_k}$ is 
the Euler homogeneity operator.
The mutually commuting operators $P_j(\lambda)$ $(1 \le j \le n)$
 were introduced in \cite[Chapter 1]{KM}, 
 and include 
 \textquotedbl{fundamental differential operators}\textquotedbl\, on the isotropic cone
 as a special case (i.e., \ $\lambda=1$).  

Let ${\mathbb{C}}_{\lambda}$ be the one-dimensional representation
 of $P$ given by $p \mapsto a^{\lambda}$,
 with the notation of \eqref{eqn:pma}.  
By a little abuse 
 of notation
 we shall use ${\mathbb{C}}_{\lambda}$
 to stand for the one-dimensional representation space
 ${\mathbb{C}}$.  
Let $\pi_{\lambda}\, (\equiv \pi_{\lambda,+})$ be the complex representation
 of $G$ on the unnormalized induced representation
$$
\operatorname{Ind}_P^G({\mathbb{C}}_{\lambda}):=\{
f\in \lC^{\infty}(G): f(gp)=a^{-\lambda}f(g)\quad\text{for any}\quad p\in P
\}  
$$
with $p$ in the notation (\ref{eqn:pma}).
The infinitesimal representation $d \pi_{\lambda}$
 of the Lie algebra ${\mathfrak {g}}$
 acts on 
$
   \lC^{\infty}({\mathfrak {n}}_-(\mathbb{R})) \otimes {\mathbb{C}}_{\lambda}.  
$
We write ${\mathbb{C}}_{\lambda}^{\vee}$
 for the contragredient representation of 
${\mathbb{C}}_{\lambda}$.  

\begin{lemma}
\label{lem:4.1}
The  elements $E_j\in\gon_+({\mathbb R})$ act on
$\lC^\infty(\gon_-(\mathbb{R})) \otimes \mC_{\lambda}$ by 
\begin{eqnarray}\label{dpio}
d\pi_\lambda(E_j)(g\otimes v)=Q_j(\la)(g)\otimes v
\quad
 \text{for } g\in
\lC^\infty(\gon_-(\mathbb{R})), v\in \mC_{\lambda}.  
\end{eqnarray}

The action of $d\tilde{\pi}_\lambda$
 on $\Pol[\xi_1,\ldots,\xi_n]\otimes \mC^\ch_{\la}$  is 
given by
\begin{eqnarray}\label{dtpio}
d\tilde{\pi}_\la(E_j)(f\otimes v)=P_j(\la)(f)\otimes v
\quad
\text{for }
f\in\Pol[\xi_1,\ldots,\xi_n],
v\in \mC_{\lambda}^\ch.  
\end{eqnarray}
\end{lemma}

\vskip 1mm
\begin{proof}
For $n=\exp Z \in N_+$
 and $x=\exp X \in N_-$
we have from \eqref{eqn:pma} 
\begin{eqnarray*}
n^{-1} x =
\left( 
\begin{array}{ccc}
a & -Z+\frac{1}{2}|Z|^2\,{}^t\! X\mJ& -\frac{1}{2}|Z|^2 \\ 
 X-\frac{1}{2}|X|^2\mJ {}^t\! Z & \Id - \mJ {}^t\! Z\otimes {}^t\! X\mJ & \mJ {}^t\! Z\\
 -\frac{1}{2}|X|^2 & -{}^t\! X\mJ & 1 
\end{array}
\right), 
\end{eqnarray*}
where 
\begin{equation}\label{eqn:aZX}
   a:=1-Z\cdot X+\frac{|Z|^2|X|^2}{4}.
\end{equation}
If $Z$ and $X$ are sufficiently small, 
 then $a \ne 0$
 and we can define
\begin{eqnarray*}
\tilde x:=
 \left(
\begin{array}{ccc}
1 & 0 & 0 \\ 
a^{-1}(X-\frac{1}{2}|X|^2\mJ {}^t\! Z) & \Id & 0\\
-\frac{1}{2}a^{-1} |X|^2 & a^{-1}({}^t\! X\mJ+\frac{1}{2}|X|^2Z) & 1 
\end{array}
\right)\in N_-.  
\end{eqnarray*}
Then $p:=\tilde{x}^{-1}n^{-1}x\in P$,
and in the expression 
 \eqref{eqn:pma},
$a$ is given by \eqref{eqn:aZX} and $m$ satisfies $\epsilon (m)=1$.  
If we let $Z$ tend to be zero,
then the elements $\tilde{x},a$ and $m$ behave up to the first order in
$\|Z\|=(\sum\limits_{i=1}^{n}|z_i|^2)^{\frac{1}{2}}$ as 
\begin{eqnarray}
\label{eqn:4.8}
a\sim 1-Z\cdot X,\;\;\; m\sim \Id-\mJ {}^t\! Z\otimes {}^t\! X\mJ+X\otimes Z;\; 
\end{eqnarray}
\begin{eqnarray}
\tilde{x}=\exp \tilde{X},\;
\tilde{X}\sim(1+Z\cdot X)\left(X-\frac{|X|^2\mJ {}^t\! Z}{2}\right).
\end{eqnarray}
Taking $Z=tE_j$, we have for $X={}^t\! (x_1, \ldots, x_n)\in\gon_-({\mathbb R})\simeq\mR^n$,
\begin{align*}
& a =1-tx_j+o(t), 
\\
& \widetilde{X} =X+tx_j{}^t\! (x_1, \ldots, x_n)-\frac{1}{2}\epsilon_jt|X|^2\, {}^t\! E_j +o(t),
\end{align*}
where $o(t)$ denotes
the Landau symbol. Therefore, for $F\in \operatorname{Ind}_P^G({\mathbb{C}}_{\lambda})$
and $x=\exp (X)\in N_-$, we have
\begin{eqnarray}
(\mathrm{d}\pi_\lambda(E_j)F)(x) & = & \frac{d}{dt}|_{t=0}\,\, F(\exp(-tE_j)x) 
\nonumber \\ \nonumber 
& = &  \frac{d}{dt}|_{t=0}\,\, F(\tilde{x}p) 
\nonumber \\ \nonumber
& = &  \frac{d}{dt}|_{t=0}\,\, a^{-\lambda}F(\exp(\tilde{X})) 
\nonumber \\ \nonumber
& = &   \big (Q_j(\lambda)(F\circ \exp)\big )(X).
\end{eqnarray}

Thus we have proved the formula (\ref{dpio}) 
for the action $d\pi_\la(E_j).$  
 
 The action of $ d\tpi_\lambda(E_j)$ is computed in two steps.  The first step is to compute  the dual action $d\pi^\ch$ 
reversing the order in the composition of operators and adding sign changes
 depending on the order of the operator. In the second step we apply the distributional Fourier transform 
$$
x_j\mapsto -i\partial_{\xi_j},\, \partial_{x_j}\mapsto -i\xi_j
$$ 
preserving the order of operators in the composition.  
\end{proof}

\subsection
{The case $(G,G^\prime)=(SO_o(p,q),SO_o(p,q-1)).$ }
\label{juhlexample}

We realize $G^\prime=SO_o(p,q-1)$
 as the subgroup of $G=SO_o(p,q)$
 which leaves the basis vector $e_{n}={}^t\!(0,\ldots, 0, 1, 0)$ invariant. We recall $n=p+q-2$.
Then the parabolic subalgebra ${\mathfrak {p}}$
 is ${\mathfrak {g}}'$-compatible 
 in the sense of Definition \ref{def:3.2}, 
 and therefore $P':=P \cap G'$
 becomes a parabolic subgroup of $G'$
 with a Levi part $L':= L \cap G'$.  

The nilpotent radical 
$
   \gon'_-(\mathbb{R})\simeq\{{}^t\! X: X=(x_1,\ldots,x_{n-1})\}\simeq \mR^{n-1},
$
has codimension one in $\gon_-(\mathbb{R})$.
We endow $\gon_-(\mathbb{R}) \simeq \mR^{p+q-2}$
 with the standard flat quadratic form $(p-1,q-1)$, 
 denoted by $\mR^{p-1,q-1}$, 
 so that $G$ acts by local conformal transformations  
 on $\gon_-(\mathbb{R})$.  
The subspace $\gon_-'(\mathbb{R}) \simeq \mR^{n-1}$ has signature $(p-1,q-2).$

According to the recipe of the F-method in Section \ref{sec:2 correction},
 we begin by finding the $L'$-module
 structure on the space $\mbox{\rm Sol}$ 
 (see \eqref{eqn:sol2l} for the definition)
 in this case.
 
\subsubsection{The space of singular vectors}
Recall from \eqref{eqn:phi} the isomorphism 
 between the space of singular vectors
 $M_\gop^\gog(V^\ch)^{\gon'_+}$
 and the space  $\mbox{\rm Sol}\equiv\mbox{\rm Sol}(\mathfrak{g},\mathfrak{g}';V^{\ch})$
  of polynomial solutions
 to a system 
 of partial differential equations.  
We are now going to determine the set $\mbox{\rm Sol},$
 and thus we describe completely the set of singular vectors.

For $k \in {\mathbb{N}}$, 
 we denote by ${\mathcal{H}}^k({\mathbb{R}}^{p-1,q-1})$
 the space of harmonic polynomials
 of degree $k$, 
 namely, 
 homogeneous polynomials
 $f(\xi)$ of degree $k$
 satisfying
\[
  (\frac{\partial^2}{\partial \xi_1^2}
   +\cdots+
   \frac{\partial^2}{\partial \xi_{p-1}^2}
   -
   \frac{\partial^2}{\partial \xi_p^2}
   -\cdots-
   \frac{\partial^2}{\partial \xi_{p+q-2}^2}
)f=0.  
\]
Then the indefinite orthogonal group $O(p-1,q-1)$
 acts irreducibly
 on ${\mathcal{H}}^k({\mathbb{R}}^{p-1,q-1})$, 
 which then decomposes
\begin{equation}
\label{eqn:brsph}
{\mathcal{H}}^k({\mathbb{R}}^{p-1,q-1})
 \simeq
 \bigoplus_{j=0}^{k} {\mathcal{H}}^j({\mathbb{R}}^{p-1,q-2})
\end{equation}
when restricted to the subgroup $O(p-1,q-2)$.  

If $C^{\alpha}_\ell(x)$ is the Gegenbauer polynomial, 
then $x^{\ell}C^{\alpha}_\ell(x^{-1})$ is an even polynomial.
Hence we can define another polynomial
 ${\mathcal C}_\ell^\alpha(s)$
 by the relation (see Appendix for more details):
\begin{equation}
\label{eqn:Ctilde}
x^{\ell}C^{\alpha}_\ell(x^{-1})=\mathcal{C}_\ell^\alpha(x^2).  
\end{equation}

We define a homogeneous polynomial
 $f_K(\xi) \equiv f_{K,\lambda}(\xi_1, \cdots, \xi_n)$
 of degree $K$ ($K \in {\mathbb{N}}$) by 
\begin{equation}
\label{eqn:FK}
 f_{K,\lambda}(\xi_1, \cdots, \xi_n):=
\xi_n^{K} {\mathcal C}_K^{-\lambda -\frac{n-1}{2}}
\left(-
\frac
{\epsilon_n\sum_{i=1}^{n-1}\epsilon_i \xi_i^2}
{\xi_n^2}
\right).
\end{equation}
Since $f_{K,\lambda}$ vanishes when $\lambda\in\{\frac{1-n}{2}, \frac{3-n}{2}, \ldots,
[\frac{K-1}{2}]+\frac{1-n}{2}\}$ for $K\geq 1$, we renormalize a non-zero element
\begin{equation}
\label{eqn:wN}
  w_K\equiv w_{K,\lambda}\in \operatorname{Pol}[\xi_1,\cdots, \xi_n] 
  \otimes {\mathbb{C}}_{\lambda}
\end{equation}
by
\begin{align*}
& w_{2N,\lambda}\quad :=\frac{N!}{(-\lambda-\frac{n-1}{2})_N}f_{2N,\lambda}(\xi_1, \cdots, \xi_n)\otimes 1_\lambda,
\\
& w_{2N+1,\lambda}:=\frac{N!}{2(-\lambda-\frac{n-1}{2})_{N+1}}f_{2N+1,\lambda}(\xi_1, \cdots, \xi_n)\otimes 1_\lambda ,
\end{align*}
where $(\alpha)_k=\alpha(\alpha+1)\cdots (\alpha+k-1)$.
We let $A\simeq\mR_{>0}$ act on 
$\Pol[\gon_+]\otimes V^\ch\simeq\Pol[\xi_1, \cdots , \xi_n]\otimes {\mathbb C}_\lambda$
by
\[
f\otimes v\mapsto f(a^{-1}\cdot )\otimes a\cdot v\quad \text{for}\quad a>0 .
\]
We say $f\otimes v$ has weight $\mu$ if this action is given by the
multiplication of $a^\mu$.
Then the weight of $w_{K,\lambda}$ is $\lambda-K$. 
\begin{theorem}\label{basis}
Let $p\geq 1, q\geq 2, p+q>4$ and $(\mathfrak{g}(\mathbb R),\mathfrak{g}'(\mathbb R))=(\mathfrak{so}(p,q), \mathfrak{so}(p,q-1))$. 
We write $1_{\la}$ for a non-zero vector
 in the one-dimensional vector space $\mC_\lambda$ 
 with parameter $\lambda\in\mC$.

\begin{enumerate}
\item[(1)]
Let $P_j(\lambda)$ $(1\leq j\leq n)$ be the second-order differential 
operators defined in (\ref{dpti}). Then the space 
$\mbox{\rm Sol}(\mathfrak{g},\mathfrak{g}';{\mathbb C}_\lambda)$ 
(see (\ref{eqn:sol2l})) is given by 
\[
\mbox{\rm Sol}(\mathfrak{g},\mathfrak{g}';{\mathbb C}_\lambda)
 =\{f\otimes 1_\lambda \in \Pol[\xi_1, \cdots , \xi_n]\otimes {\mathbb C}_\lambda:\,
P_j(\lambda)f=0\, (1\leq j\leq n-1)\}.
\]
\item[(2)]
For $\lambda \not \in {\mathbb{N}}$, 
we have 
\[
\mbox{\rm Sol}(\mathfrak{g},\mathfrak{g}';\mathbb{C}_\lambda)
 =\bigoplus_{K=0}^{\infty}{\mathbb{C}} w_{K,\la}.
\]
In particular, the inverse Fourier transform (see (\ref{eqn:phi})) gives an $L'$-isomorphism:
\[
  M_{{\mathfrak {p}}}^{{\mathfrak {g}}}({\lambda})
  ^{{\mathfrak {n}}_+'} 
  \overset{\sim} {\underset{\varphi}{\leftarrow}}
 \mbox{\rm Sol}(\mathfrak{g},\mathfrak{g}';\mathbb{C}_\lambda)
 =\bigoplus_{K=0}^{\infty}{\mathbb{C}} w_{K,\la}.  
\]

\item[(3)]
 For $\la\in\mN$, we have 
\[
\mbox{\rm Sol}(\mathfrak{g},\mathfrak{g}';\mathbb{C}_\lambda)
 =\bigoplus_{K=0}^{\infty}{\mathbb{C}} w_{K,\la}
  \oplus \bigoplus_{j=1}^{\la +1 }H_j',
	\]
where $H'_j \equiv H_{j,\lambda}'$ $(1\leq j\leq \lambda +1)$ is the subspace
 of ${\mathcal{H}}^{\lambda+1}({\mathbb{R}}^{p-1,q-1})\otimes {\mathbb{C}}_{\lambda}$
 corresponding to the summand ${\mathcal{H}}^j({\mathbb{R}}^{p-1,q-2})$
 in \eqref{eqn:brsph}. In particular, the inverse Fourier transform induces  
 an $L'$-isomorphism:
\[
  M_{{\mathfrak {p}}}^{{\mathfrak {g}}}({\lambda})
  ^{{\mathfrak {n}}_+'} 
  \overset{\sim} {\underset{\varphi}{\leftarrow}}
 \mbox{\rm Sol}(\mathfrak{g},\mathfrak{g}';\mathbb{C}_\lambda)
 =\bigoplus_{K=0}^{\infty}{\mathbb{C}} w_{K,\la}
  \oplus \bigoplus_{j=1}^{\la +1 }H_j'.  
\]
Furthermore, for each $j=1,\ldots,\lambda+1,$ the image   
$\varphi(H'_j)$ is contained in the $\gog'$-submodule
 generated
by the  vector $\varphi(w_{\lambda+1-j,\lambda})$.  
\end{enumerate}
\end{theorem}

\vskip 1mm
\begin{proof}
We apply the F-method as follows.
By \eqref{dpti}, 
 the equation $d \tilde \pi(Z)f=0$
 for $Z \in \gon_+'$
 amounts to a system of differential equations
\begin{eqnarray}\label{dtpis}
P_j(\lambda)f = 0 \,\,\, \text{for }\,j=1,\dots ,n-1.
\end{eqnarray}
Hence we have shown the first statement. 

In order to analyze 
$\mbox{\rm Sol}(\mathfrak{g},\mathfrak{g}';\mathbb{C}_\lambda)$
explicitly, we first prove that if $\lambda \notin {\mathbb{N}}$
 then any polynomial solution $f$ to \eqref{dtpis}
 is $SO_o(p-1,q-2)$-invariant.  
Since the operators $P_j(\lambda)$ decrease the homogeneity
 by one, 
we can assume
 $f\in\mbox{\rm Sol}(\mathfrak{g},\mathfrak{g}';\mathbb{C}_\lambda)$ to be homogeneous without loss 
 of generality.
It follows from \eqref{dtpis}
 that 
\[
(\epsilon_i\xi_iP_j(\lambda)-\epsilon_j\xi_jP_i(\lambda))f=0, 
\]
 which amounts to 
\begin{eqnarray}
(E-\lambda-1)(\epsilon_j\xi_j\partial_{\xi_i}
-\epsilon_i\xi_i\partial_{\xi_j})f=0
\end{eqnarray}
for all $i,j=1,\dots ,n-1.$ We recall $n=p+q-2$.
Hence if $\la\not= \deg f-1$, then $(\epsilon_j\xi_j\partial_{\xi_i}
-\epsilon_i\xi_i\partial_{\xi_j})f=0$, namely, $f$ is a
polynomial invariant under $SO_o(p-1,q-2)$.  

Next,
 let us solve the equation \eqref{dtpis}.  
\vskip 1mm
\noindent
Case 1.   $SO_o(p-1,q-2)$-invariant solutions.

As we saw above, 
 this is always the case when $\lambda\not\in\mN.$
Since $p+q-3\geq 2$, classical invariant theory says
 that any $SO_o(p-1,q-2)$-invariant, homogeneous 
polynomial $f$ in $n$-variables of degree $K$
can be written in the form 
\[
     f(\xi_1,\ldots,\xi_n)=\xi_n^K\,h(t),
\]
where 
\[
t=\frac{\epsilon_n|\xi'|^2}{\xi_n^2},\,|\xi'|^2=\sum_{i=1}^{n-1}\epsilon_i\xi_i^2
\]
and $h$ is a polynomial of degree $N$
(depending on the parity of $K$,  either $K=2N$ or $K=2N+1$).

Hence we look for a solution
 of the form $f=\xi_n^{K}h(t),\;t=\frac{\epsilon_n|\xi'|^2}{\xi_n^2}.$
We get immediately
$$
\pa_j f=\xi_n^{K-2}2\epsilon_n \epsilon_j \xi_j\,h',\;
\pa_j^2 f
=\xi_n^{K-2}\left(4h''\frac{\xi_j^2}{{\xi_n}^2}+\epsilon_n\epsilon_j2h'\right),
\,j=1,\ldots,n-1,
$$
\begin{align*}
 \square'f&=\epsilon_n \xi_n^{K-2}\left(4h^{\prime\prime}t+2(n-1)\,h^\prime\right),\;
\pa_{\xi_n} f=\xi_n^{K-1}\left(K\,h-2\,h^\prime t\right),
\\
\eps_n \pa_{\xi_n}^2 f &=
 \eps_n\xi_n^{K-2}\left(K(K-1)\,h+(-4K+6)\,h't+4\,h''t^2\right),
\\
\square f&=\eps_n\xi_n^{K-2}\left(4t(1+t)\,h''+\left(2(n-1) +t(-4K+6)\right)\,
h^\prime+K(K-1)\,h\right),
\\
(\la-E)\pa_j f&=\epsilon_n\epsilon_j\xi_j\xi_n^{K-2}\left(2\la-2K+2\right)h^\prime.
\end{align*}
Collecting terms together and
 cancelling the scalar multiple
 by $\frac{1}{2}\epsilon_n\epsilon_j\xi_j\xi_n^{K-2}$, 
 the partial differential equation (\ref{dtpis})
 induces the following ordinary differential equation
 for $h(t)$ which is independent of $j$ ($1\leq j\leq n-1$):
 \begin{eqnarray}\label{obyc}
\label{hyperoutput}
R(K,-\lambda-\frac{n-1}{2})h(t)
=0,
\end{eqnarray}
where we define a differential operator $R(l,\alpha)$ of second order by 
\begin{eqnarray}
R(l,\alpha):=4t(1+t)\frac{d^2}{dt^2}+((6-4l)t+4(1-\alpha-l))\frac{d}{dt}+l(l-1).
\label{eqn:Rla}
\end{eqnarray}

Hence the resulting system of equations (\ref{dtpis}) (for $j=1,\ldots,n-1$) reduces
 to a single ordinary differential equation
 of second order for $h(t)$.

We set $g(x):=x^Kh(-\frac{1}{x^2})$, 
then $g(x)\in\mathrm{Pol}_K[x]_{\mathrm{even}}$ belongs to
\[
\mathrm{Pol}_K[x]_{\mathrm{even}}:={\mathbb C}\text{-span}\Big\{ x^{K-2j}:
\, 0\leq j\leq \Big[\frac{K}{2}\Big] \Big\}.
\] 
It follows from Lemma \ref{lem:ghGegen} that $g(x)$ is a scalar multiple of 
the renormalized Gegenbauer polynomial
\begin{align*}
\widetilde{C}^\alpha_K(x):=\frac{1}{(\alpha)_{[\frac{K+1}{2}]}}C^\alpha_K(x)\quad 
\text{for}\quad \alpha=-\lambda-\frac{n-1}{2},
\end{align*}
and $h(t)=\widetilde {\mathcal C}_{K}^{-\lambda-\frac{n-1}{2}}(-t)$
 up to scalar (see \eqref{eqn:Ctilde}.)  
Hence $f(\xi_1,\dots ,\xi_n)$ is a scalar multiple of 
$\frac{1}{(-\lambda-\frac{n-1}{2})_{[\frac{K+1}{2}]}}f_{K,\lambda}$.
\vskip 1mm
\noindent
Case 2. $\la\in \mN.$

We set $k:=\lambda+1$.  
Suppose $f$ is a homogeneous polynomial.  
As we saw,
 if $f\otimes 1_\lambda \in \mbox{\rm Sol}(\gog, \gog'; {\mathbb C}_\lambda)$, 
then $f$ is $SO_o(p-1,q-2)$-invariant
 as far as 
 $\deg f \ne k$.  
Suppose now
 $\deg f=k$.  
Then $(\la-E)\pa_{\xi_j}\,f$ vanishes, because $\partial_{\xi_j}f$ is homogeneous 
of degree $k-1\, (=\lambda)$.  
In view of the formula (\ref{dpti}) of $P_j(\lambda)$ for homogeneous polynomials $f$
 of degree $k$, 
 $f\otimes 1_\lambda \in \mbox{\rm Sol}(\mathfrak{g},\mathfrak{g}';\mathbb{C}_\lambda)$
 if and only if $f$ satisfies
 the single equation $\square f=0$,  
namely,
 $f \in {\mathcal{H}}^k({\mathbb{R}}^{p-1,q-1})$. Hence we have shown 
\begin{eqnarray}\nonumber
\mbox{\rm Sol}(\mathfrak{g},\mathfrak{g}';\mathbb{C}_\lambda)
  &=& \bigoplus_{K=1}^{\infty}{\mathbb{C}} w_{K,\la}
  \oplus {\fam2 H}^{\lambda+1}({\mathbb R}^{p-1,q-1})\otimes{\mathbb C}_\lambda
\\ \nonumber
 &=& \bigoplus_{K=0}^\infty{\mathbb C}w_{K,\lambda}\oplus\bigoplus_{j=1}^{\lambda+1}H'_j .
\end{eqnarray}
For $ j=0,\ldots,k\, (= \lambda+1),$ let $M_j$  be the  $\gog'$-submodules in $M^\gog_\gop(\lambda)$ 
 generated by the singular vectors  $\varphi(w_{\lambda+1-j,\lambda}).$ 
Finally, let us prove
\[
\varphi(H'_j)\subset M_j\quad (j=0,1, \cdots, \lambda +1).
\]
This is deduced from the following three claims, which
we are going to prove:
\begin{itemize}
\item
$\varphi(H'_j)\subset M_0+M_1+\cdots +M_{\lambda+1}.$
\item
$\sum\limits_{i=0}^{\lambda+1}M_i$ is a direct sum of $\gog'$-modules.
\item
Among $\{M_0, M_1,\cdots ,M_{\lambda+1}\}$, $M_j$ is the unique $\gog'$-submodule
that has the same ${\mathfrak Z}(\gog')$-infinitesimal character
with that of the $\gog'$-module generated by $\varphi(H'_j)$.
\end{itemize}
In the F-method, the Lie algebra ${\mathfrak {n}}_-'$
 acts on $\operatorname{Pol}[\xi_1, \cdots, \xi_{n-1}]
 \otimes {\mathbb{C}}_{\lambda}$
 by the multiplication 
of a linear function of $\xi_1, \cdots, \xi_{n-1}$. Hence we have
\[
M_j=\varphi (\operatorname{Pol}[\xi_1, \cdots, \xi_{n-1}]w_{\lambda+1-j, \lambda}).
\]
Thus the first claim will be proved if we show
\begin{eqnarray}
\label{eqn:Polw}
\operatorname{Pol}^{\lambda+1}[\xi_1, \cdots, \xi_{n}]\otimes 1_\lambda\subset
\sum\limits_{i=0}^{\lambda +1}\operatorname{Pol}[\xi_1, \cdots, \xi_{n-1}]w_{\lambda+1-i, \lambda},
\end{eqnarray}
where $\operatorname{Pol}^{l}[\xi_1, \cdots, \xi_{n}]$ stands for the space of 
homogeneous polynomials of degree $l$. To see (\ref{eqn:Polw}), we observe that the coefficient 
of $\xi_n^K\otimes 1_\lambda$ in the polynomial $w_{K,\lambda}$ (see (\ref{eqn:wN}))
is a non-zero multiple of
\begin{eqnarray}
\nonumber
\frac{(\alpha)_K}{(\alpha)_{[\frac{K+1}{2}]}}=(\alpha+[\frac{K+1}{2}])\cdots(\alpha+K-2)(\alpha+K-1)
\quad \text{with}\quad \alpha=-\lambda-\frac{n-1}{2},
\end{eqnarray}
which does not vanish if $K\leq \lambda+1$. Hence we see by induction on $l$ that 
$\sum\limits_{K=0}^{l}\operatorname{Pol}[\xi_1, \cdots, \xi_{n}]w_{K, \lambda}$
coincides with 
$\sum\limits_{K=0}^{l}\operatorname{Pol}[\xi_1, \cdots, \xi_{n-1}]\xi_n^K\otimes 1_\lambda$
for all $l\leq \lambda+1$. In particular, (\ref{eqn:Polw}) is proved and the first claim
is shown.

To see the second and third claims, let
$\mathfrak{j}'$ be a Cartan subalgebra of 
$\gog'=\mathfrak{so}(n+1,{\mathbb C})\simeq\mathfrak{so}(p,q-1)\otimes_{\mathbb R}{\mathbb C}$.
We identify
 $({\mathfrak {j}}')^{\ast}
 \simeq {\mathbb{C}}^{[\frac{n+1}{2}]}$
 via the standard basis
 $\{e_1, e_2, \cdots, e_{[\frac{n+1}{2}]}\}$
 so that ${\mathbb{C}}_{\lambda}$ is given by 
 $\lambda e_1$ ($\lambda \in {\mathbb{C}}$)
 and similarly for ${\mathfrak {j}}^{\ast}
 \simeq {\mathbb{C}}^{[\frac{n+2}{2}]}$.  
In the coordinates
 we have 
\[
  \rho'=
  (\frac{n-1}{2}, \frac{n-3}{2}, \cdots, 
   \frac{n-1}{2}-[\frac{n-1}{2}])
  \in 
  ({\mathfrak {j}}')^{\ast}.   
\]
Then the $\mathfrak{Z}(\gog')$-infinitesimal character of $M_j$ is given by 
\[
(\lambda-(\lambda+1-j),0, \cdots ,0)+\rho'=
(j+\frac{n-3}{2}, \frac{n-3}{2}, \frac{n-5}{2}, \cdots, 
   \frac{n-1}{2}-[\frac{n-1}{2}])
\]
which are distinct for $j=0,1, \cdots , \lambda+1(=k)$ in 
$({\mathfrak {j}}')^{\ast}/W(\gog')$ if $n\geq 3$. Therefore,
the sum $M_0+M_1+\cdots + M_{\lambda+1}$ is a direct sum if $n\geq 3$.

Let us now consider  one fixed summand $H'_j$ for $j=0,1,\ldots,k.$  
Since such $f$ is homogeneous of degree $k(=\lambda+1)$, the  
$\mathfrak{Z}(\gog')$-infinitesimal
character of the $\gog'$-module
 generated by $\varphi(H'_j)$ is given by 
\[
(-1, j, 0, \cdots , 0)+\rho'=
(\frac{n-3}{2}, j+\frac{n-3}{2}, \frac{n-5}{2}, \cdots, 
   \frac{n-1}{2}-[\frac{n-1}{2}]),
\]
which coincides with that of $M_j$. 
Hence we have shown $\varphi(H'_j)\subset M_{j}$.
\end{proof}
 
Theorem \ref{basis} gives a complete description
 of the set of singular vectors invariant with respect to
 $SO_o(p-1,q-2). $ 
In the positive signature, which corresponds to the 
 $p=1$ case,
 $SO(q-2)$-invariant
 singular vectors   were found by A. Juhl
 in \cite[Chapter 5]{ju} by a heavy combinatorial computation
 using recurrence relations.  
 Notice that the Juhl's computations
 depend on the parity of $K$, and the higher dimensional
components of singular vectors are more difficult
 to detect by algebraic methods.  
We would like to emphasize that
 our approach is very different and allows us 
 a uniform treatment.  

An important point is that the F-method gives a complete
description of the set of singular vectors and its structure
as an $L'$-module.  
In the second part of Theorem \ref{basis},
we describe also all $L'$-submodules of higher dimensions,
 which will be used in the complete description
 of the composition series
  when the restriction is not completely reducible
 (see Theorem \ref{T.3.8}). 

\subsubsection{Equivariant differential operators for the conformal group}

For $ 0\leq j\leq N-1$, we set 
\begin{eqnarray}\label{acoef}
& & a_j(\la)\equiv a_j^{N,n}(\lambda):=\frac{(-2)^{N-j} N!}{j!(2N-2j)!}
\prod_{k=j}^{N-1}(2\la-4N+2k+n+1),  
\\
& & \nonumber
b_j(\la)\equiv b_j^{N,n}(\lambda):=\frac{(-2)^{N-j} N!}{j!(2N-2j+1)!}\prod_{k=j}^{N-1}(2\la-4N+2k+n-1), 
\\ \label{bcoef}
\end{eqnarray}
and $a_N(\lambda)=b_N(\lambda):=1$.
Here we have adopted the same notation with Juhl's book \cite{ju}, i.e.
$[\text{loc. cit.}, \,\text{Corollary}\,\, 5.1.1]$ for $a_j(\lambda)$ and 
$[\text{loc. cit.}, \,\text{Corollary}\,\, 5.1.3]$ for $b_j(\lambda)$ for 
the convenience of the reader.

Then it follows from (\ref{eqn:Caj}) and  (\ref{eqn:Cbj}) in Appendix that
the formula \eqref{eqn:FK} amounts to
\begin{eqnarray}\label{eqn:weven}
& & w_{2N,\lambda}\quad =(\sum\limits_{j=0}^Na_j(\lambda)(-|\xi'|^2)^j\xi_n^{2N-2j})\otimes 1_\lambda, 
\\ \label{eqn:wodd}
& & w_{2N+1,\lambda}=(\sum\limits_{j=0}^Nb_j(\lambda)(-|\xi'|^2)^j\xi_n^{2N-2j+1})\otimes 1_\lambda,
\end{eqnarray}
because $\epsilon_n=-1$.

As stated in Theorem {\ref{folkloreg}, 
the homomorphism between the generalized Verma modules 
of the Lie algebras $\gog'$ and $\gog$ induces
an equivariant differential operator
 acting on local sections of induced homogeneous bundles 
on the generalized flag manifolds of the Lie groups $G'$ and $G$. 
We shall describe these differential operators
 using the non-compact picture
of the induced representation.
The restriction from $G$ to $N_-\,P$ 
induces the non-compact model of the induced representation
by the map
$$
\beta: {\Ind}_P^G(\mC_\lambda)\hookrightarrow \lC^\infty(N_-)\simeq \lC^\infty(\mR^{p-1,q-1}).
$$
Via the injection $\beta$, we get the following explicit form of $G'$-equivariant
differential operator by replacing $\lambda$ with $-\lambda$.

\begin{theorem}
\label{thm:4.3}
Let $(G,G^\prime)=(SO_o(p,q),SO_o(p,q-1))$, $p\geq 1$, $q\geq 2$, $n=p+q-2>2$, and $\lambda\in\mC$.
\begin{enumerate}
\item[(1)]
The singular vectors $\varphi(w_{2N,-\lambda})
 \in M_{{\mathfrak {p}}}^{{\mathfrak {g}}}({-\lambda})$
 in Theorem \ref{basis} (1) induce
(in the non-compact picture) the family 
$
\mathscr{D}_{2N}(\la): \lC^\infty(\mR^{p-1,q-1})\to \lC^\infty(\mR^{p-1,q-2})
$
 of $G'$-equivariant differential operators
 given by $\mathscr{D}_{2N}f=(D_{2N}(\lambda)f)|_{x_n=0}$, where
 $D_{2N}(\lambda)$ is a differential operator 
 of order $2N$ defined as follows:
$$
D_{2N}(\la):=\sum_{j=0}^N a_j(-\la)(-\square')^j
(\frac{\pa}{\pa\,x_n})^{2N-2j} .  
$$
Here
  $\square'=
	\frac{\partial^2}{\partial x_1^2}
   +\cdots+
   \frac{\partial^2}{\partial x_{p-1}^2}
   -
   \frac{\partial^2}{\partial x_p^2}
   -\cdots-
   \frac{\partial^2}{\partial x_{n-1}^2}
	$ 
	is the (ultra)-wave operator on $\mR^{p-1,q-2}$.

The infinitesimal intertwining property of 
$\mathscr{D}_{2N}(\lambda)$ is given explicitly as
\begin{eqnarray}
\mathscr{D}_{2N}(\lambda)d\pi^G_{\lambda}(X)
=d\pi^{G'}_{\lambda+2N}(X)\mathscr{D}_{2N}(\lambda)
\quad\text{for }\,\,
 X\in\gog'.  
\end{eqnarray}

Similarly,  
the singular vectors $\varphi(w_{2N+1,-\lambda})\in M_\gop^\gog({-\lambda})$
 in Theorem \ref{basis} (1) define the family $\mathscr{D}_{2N+1}(\la):
 \lC^\infty(\mathbb{R}^{p-1,q-1}) \to \lC^\infty(\mathbb{R}^{p-1,q-2})$
 of $G'$-equivariant differential operators induced by
$$
D_{2N+1}(\la):=\sum_{j=0}^N b_j(-\la)(-\square')^j
(\frac{\pa}{\pa\,x_n})^{2N-2j+1}.  
$$
The operator $\mathscr{D}_{2N+1}(\la)$ intertwines
 $d \pi_{\lambda}^G$ and $d \pi_{\lambda+2N+1}^{G'}$.  
 
\item[(2)] The branching law
${\mathcal{H}}^k({\mathbb{R}}^{p-1,q-1})
 \simeq
 \bigoplus_{j=0}^{k} {\mathcal{H}}^j({\mathbb{R}}^{p-1,q-2})$
(see (\ref{eqn:brsph})) is multiplicity-free, 
and we denote by $\pi^k_j$ the corresponding orthogonal projections.
Note that the $O(p-1,q-1)$-modules ${\mathcal{H}}^k({\mathbb{R}}^{p-1,q-1})$ are 
isomorphic to the $k$-th symmetric and trace-free part
$\odot^k_0(\mathbb{R}^{p-1,q-1})$ of the defining representation of $O(p-1,q-1)$.

{For $-\lambda\in\mathbb{N},$ set  $k=|\lambda|+1.$
Then we have $G$-equivariant differential operator
\[
\mathscr{D}_k: \lC^\infty(\mathbb{R}^{p-1,q-1}) \to
  \lC^\infty(\mathbb{R}^{p-1,q-1}) \otimes
  \mathcal{H}^k(\mathbb{R}^{p-1,q-1}),
\]
 given by 
the set $f\otimes 1_{-\lambda},\,f\in\mathcal{H}^k(\mathbb{R}^{p-1,q-1})$ of singular vectors.
In the non-compact picture, the operator $\mathscr{D}_k$ corresponds to
the (symmetric) trace-free part of the multiple gradient $\sigma\mapsto\nabla_{(a}\ldots\nabla_{b)_0}\sigma$ (number of indices being $k$).
}

Moreover, for each $j=0, 1,2,\dots,k,$ there are $G'$-equivariant differential operators  
\[
\mathscr{D}_{k,j}: \lC^\infty(\mathbb{R}^{p-1,q-1}) \to
  \lC^\infty(\mathbb{R}^{p-1,q-2}) \otimes
  \mathcal{H}^j(\mathbb{R}^{p-1,q-2}); 
  \]
 given by the composition $\mathscr{D}_{k,j}=\pi^k_j\circ   \mathscr{D}_k,
$ restricted to $\mathbb{R}^{p-1,q-2}.$
 
\end{enumerate}
\end{theorem}

\vskip 1mm
\begin{proof}
The first part of the theorem follows immediately from Theorem \ref{basis} (2)
 and from the fact that an element $X\in \gon_-(\mathbb{R})$
acts on functions in $\lC^\infty(\gon_-(\mathbb{R}))$ by the derivative in the direction $X.$

The second part
follows from Theorem \ref{basis} (3)
and the well-known classification of  differential
operators on the sphere $S^{n}\simeq G/P$ equivariant with respect to the action of the conformal group  (see, e.g.,  \cite[Sections 8.6--8.9]{slovak}).
\end{proof}

\begin{remark}\label{rem:4.4.}
Denote the induced representation ${\Ind}_P^G(\mC_\lambda)$ 
by $\pi_{\lambda,+} (\equiv \pi_\lambda)$ and denote 
the representation ${\Ind}_P^G(\mathrm{sgn}\otimes \mC_\lambda)\equiv{\Ind}_P^G((-1)\otimes \mC_\lambda)$
 induced from $p\mapsto \epsilon(m) a^\lambda$
 by $\pi_{\lambda,-}$
 (see \eqref{eqn:pma} for notation). They
give rise to the same action of the Lie algebra, but on the level of the induced representations 
we have the following intertwining relation
\begin{eqnarray}
\mathscr{D}_K(\lambda)\pi^G_{\lambda,\epsilon_1}(g')
=\pi^{G'}_{\lambda+K,\epsilon_2}(g')\mathscr{D}_K(\lambda),
\end{eqnarray}
where $g'\in G'$ and $\epsilon_1\cdot\epsilon_2=(-1)^K$, $K\in{\mathbb N}$.

The results obtained in the first part of the Theorem \ref{thm:4.3} generalize those obtained
 in \cite[Chapter 5]{ju} for the positive definite signature. 
Our proof based on the F-method is completely 
 different from \cite{ju},
 and is significantly
 shorter even in the $p=1$ case.

The operator $\mathscr{D}_k$, defined in 
the second part of Theorem \ref{thm:4.3},
  is the first BGG operator in the BGG complex corresponding to the 
	$G$-module given by the $k$-th symmetric traceless power
of its fundamental vector representation.  
There are also explicit formulae for a majority of
operators appearing in the BGG complexes in the compact picture 
(as well as for their curved versions in the BGG sequences), 
  but the expressions are more 
complicated and contain many lower order curvature terms (see, \cite{cds}).

The second part of Theorem \ref{thm:4.3} is an example of a more general principle,
which can be formulated as follows. Every $G$-equivariant differential operator $\mathscr{D}$, acting
between sections of homogeneous bundles over $G/P$,
induces by restriction to $G'/P'$
a $G'$-equivariant differential operator. 
Moreover, we may compose $G$-equivariant differential operators or
$G'$-equivariant differential operators to get $G'$-equivariant differential operators.
This composition may be possible 
for a discrete set of $\lambda.$ 
Such possibilities called factorization identities will be discussed 
(for densities) in Section \ref{4.3.5} in more details.
\end{remark}
\subsubsection{The branching rules for Verma modules --- generic case}

The branching rules for generic parameters are obtained
as a special case
 of the general theory stated in 
 Section \ref{sec:3}.  
The proof does not require the F-method.  

\begin{theorem}  \label{generic}
For $\lambda \in\mC\setminus \{\frac 1 2(k-n):k=2,3,4,\cdots\}$, 
the Verma module 
$M^\gog_\gop(\lambda)$ 
decomposes 
 as a direct sum of generalized Verma modules of ${\mathfrak {g}}'$:
\begin{eqnarray}
\label{eqn:4.9}
M^{\gog}_{\gop}(\lambda)|_{\gog'}
\simeq 
\bigoplus_{b\in\mN}M^{\gog'}_{\gop'}(\lambda -b).
\end{eqnarray}
\end{theorem}

\vskip 1mm
\begin{proof}
Apply Theorem \ref{thm:3.5} to the special case
 where $\gog^\tau=\gog'$ and  
\[
  ({\mathfrak {g}}, {\mathfrak {g}}', 
   {\mathfrak {p}}/{\mathfrak {n}}_+)
  \simeq
  ({\mathfrak {so}}(n+2,{\mathbb{C}}), 
   {\mathfrak {so}}(n+1,{\mathbb{C}}),
   {\mathfrak {so}}(n,{\mathbb{C}}) \oplus {\mathbb{C}}).  
\]
Then $l=1$ and $\nu_1 = -e_1$.  
Hence we get \eqref{eqn:4.9} from Theorem \ref{thm:3.5}
 as the identity 
 in the Grothendieck group
 for all $\lambda \in {\mathbb{C}}$.  
The infinitesimal characters
 of $M_{{\mathfrak {p}}'}^{{\mathfrak {g}}'}(\lambda-b)$
 are given by 
\[
  (\lambda-b+\frac{n-1}{2},
   \frac{n-3}{2},
   \frac{n-5}{2},
   \cdots,
   \frac{n-1}{2}-[\frac{n-1}{2}]),
\]
 which are all distinct 
 in $({\mathfrak {j}}')^{\ast}/W({\mathfrak {g}}')$ for $b \in {\mathbb{N}}$
 if and only if
 $2\lambda + n \ne 2,3,4, \cdots$, 
 wh\textsc{}ence the last statement.  
\end{proof}
 
We shall give in Corollary \ref{cor:girred}
 a necessary and sufficient condition
 on $\lambda$ 
 for the irreducibility 
 of $M_{{\mathfrak {p}}'}^{{\mathfrak {g}}'}(\lambda)$.  
The branching law in the singular case $\lambda \in \frac 1 2 (k-n)$
 $(k=2,3,4, \cdots)$
 will be treated in Theorem \ref{T.3.8}.  
 
\subsubsection{The branching rules --- exceptional cases}
For integral values of the inducing parameter, 
 the branching law is not always a direct sum decomposition
 but may involve extensions.  
To understand this delicate structure,
 we shall apply the F-method again
 and use an explicit form of singular vectors. 
The description of the branching rules
 for exceptional parameters
 was earlier studied
in a special case corresponding to Juhl's operator in \cite{pso}. 
 
 Let us first notice
 that the $\gog$-module
 $M^\gog_\gop(\lambda)\equiv M^\gog_\gop(\mC_\lambda)$ decomposes for all
 $\la$ into an even and odd part as ${\mathfrak {g}}'$-modules. 
In the F-method, we take the Fourier transform of the 
Verma module $M^\gog_\gop(\mC_\lambda)$, and this decomposition is 
described as the decomposition of
 $\Pol[\xi_1,\ldots,\xi_n]\otimes \mC_\lambda$
into
$$
\left(\bigoplus_{k=0}^\infty \Pol[\xi_1,\ldots,\xi_{n-1}]\xi_n^{2k}\right)\otimes \mC_\lambda
 \oplus\left(\bigoplus_{k=0}^\infty \Pol[\xi_1,\ldots,\xi_{n-1}]\xi_n^{2k+1}\right)\otimes \mC_\lambda.  
$$
By the formula of $d\pi_\lambda(E_j)$ (see Lemma \ref{lem:4.1}), it 
is easy to see that both summands are $\gog'$-submodules.

Any singular vector vector $\varphi(w_{\lambda,K})$ ($K\in\mN$)
in $M^\gog_\gop(\lambda)^{\gon'_+}$ (see Theorem \ref{basis})
generates a $\gog'$-submodule of $M^\gog_\gop(\lambda)$, which 
we denote by $V_K\equiv V_K(\lambda)$. Since $\gon_-$ acts
freely on $M^\gog_\gop(\lambda)$, the $\gog'$-module $V_K$
is isomorphic to the $\gog'$-Verma module 
$M^{\gog'}_{\gop'}({\mathbb C}_{\lambda-K})$. We note that
the $\gog'$-submodule $\sum\limits_{K\in\mN}V_K$ in
$M^\gog_\gop(\mC_\lambda)$ is not necessarily a direct sum
for exceptional parameter $\lambda$ (see (\ref{eqn:lmdj}) below),
where two $\gog'$-modules $V_K$ and $V_{K'}$ may have the
same ${\mathfrak Z}(\gog')$-infinitesimal character. In order
to understand what happens about the $\gog'$-module structure
of the $\gog$-module $M^\gog_\gop(\mC_\lambda)$ in this 
case, we apply again the F-method -- take the inverse Fourier 
transform (see (\ref{eqn:phiPM}))
\[
V_K=\varphi(\operatorname{Pol}[\xi_1, \cdots, \xi_{n}]w_{K, \lambda}) 
\]
and use an explicit formula for the singular vector $w_{K,\lambda}$
 (see \eqref{eqn:wN}).   

In cases where the submodules  $V_{2N}$ and $V_{2N'}$ have the same infinitesimal
character, we shall see (due to the knowledge of the explicit form of the singular 
vectors) that one of them
is submodule of the other. 
In this case, we find a $\gog'$-submodule $M^\gog_\gop(\mC_\lambda)$ which allows 
a non-splitting extension. 
We shall illustrate it in a number of examples.
For this, we begin with explicit formulas of $w_{K, \lambda}$ $(K\leq 4)$
as follows (see Appendix for the formula for $\widetilde{C}^\lambda_K(t)$
($K=0,1,\cdots , 4$)):
\begin{eqnarray}
& & w_0\equiv w_{0, \lambda} = 1\otimes 1_\lambda,
\nonumber \\
& & w_1\equiv w_{1, \lambda} = \xi_n\otimes 1_\lambda,
\nonumber \\
& & w_2\equiv w_{2, \lambda} = (-(2\lambda+n-3)\xi_n^2-|\xi'|^2)\otimes 1_\lambda,
\nonumber \\
& & w_3\equiv w_{3, \lambda} = \xi_n(-\frac{1}{3}(2\lambda+n-5)\xi_n^2-|\xi'|^2)\otimes 1_\lambda,
\nonumber \\
& & w_4\equiv w_{4, \lambda} = (\frac{1}{3}(2\lambda+n-5)(2\lambda+n-7)\xi_n^4+2
(2\lambda+n-5)\xi_n^2|\xi'|^2+|\xi'|^4)\otimes 1_\lambda, \nonumber
\end{eqnarray}
where we recall $|\xi'|^2=\sum\limits_{i=1}^{n-1}\epsilon_i\xi_i^2$.

We set 
\begin{equation}
\label{eqn:lmdj}
\lambda_j :=\frac 1 2 (-n+1+j).  
\end{equation}

\begin{example}\label{Ex1}
{\it The case $\la=\la_1=\frac{-n+2}{2}.$}
In this case, all the infinitesimal characters of the $\gog'$-submodules
 generated by singular vectors $w_K$ $(K\in\mN)$ are mutually different 
except those corresponding to $w_0$ and $w_1,$ which coincide. 
 Due to the fact that the whole $\gog$-module splits into
a direct sum of even and odd parts, there is no extension among these 
$\gog'$-modules and therefore,
the branching is
 the same as in the generic case.
 \end{example}

\begin{example}\label{Ex2} {\it The case $\la=\la_2=\frac{-n+3}{2}.$}
In this case, the infinitesimal characters of $\gog'$-submodules
 generated by singular vectors $w_K$ $(K\in\mN)$ coincide only for
$w_0$ and $w_2$, and 
all others are mutually different. We compare $w_0$ and $w_2$. For 
$\lambda=\lambda_2$, the first term of $w_2\equiv w_{2,\lambda}$
vanishes, and $w_2$ reduces to $-|\xi'|^2\otimes 1_\lambda$. Hence 
\[
\operatorname{Pol}[\xi_1, \cdots, \xi_{n-1}]w_{0, \lambda} 
\supset
\operatorname{Pol}[\xi_1, \cdots, \xi_{n-1}]w_{2, \lambda} 
\]
for $\lambda=\lambda_2$. In turn, we have 
$V_0(\lambda_2)\supset V_2(\lambda_2)$.

Thus
for $\la=\la_2=\frac{-n+3}{2},$  the $\gog'$-submodule $V_0(\lambda_2)$ generated by $w_0$
contains the unique nontrivial 
submodule $V_2(\lambda_2)$ , generated by $w_2$. On the other hand, the direct sum $M_{02}$ of 
the $U(\gon'_-)$-span of $w_0=1_\la$ and 
the $U(\gon'_-)$-submodule generated by the vector  $\xi_n^2\otimes 1_\la$
 is invariant under the action of $\gog',$ and it is a
(non-split) extension
\begin{eqnarray}   
0\to {M}^{\gog^\prime}_{\gop^\prime}({\la_2})\to M_{02}
\to {M}^{\gog^\prime}_{\gop^\prime}({\la_2}-2)\to 0.
\end{eqnarray}
All the other infinitesimal characters are mutually different,
 hence the branching rule is now given by
$$
M^\gog_\gop(\la_2)\simeq
M_{02}\oplus\bigoplus_{b\in\mN,b\not=0,2}M^{\gog'}_{\gop'}(\la_2 -b).
$$ 
 \end{example}

\begin{example}\label{Ex3} {\it The case $\la=\la_3=\frac{-n+4}{2}.$}
In this case, the infinitesimal characters of $\gog'$-submodules
 generated by singular vectors $w_0, w_3$ respectively
$w_1,w_2$ coincide, and both characters are different from each other,  and differ
from all others (which are also mutually different). But again due to the fact that the whole $\gog$-module splits into
a direct sum of even and odd parts, the whole branching is again the same as in generic case.
 \end{example}
 
\begin{example}\label{Ex4} 
Let $\la=\la_4=\frac{-n+5}{2}.$ The explicit formula of singular vectors $w_{K,\lambda}$
shows that in this case the vectors shows the first two terms of $w_4=w_{4, \lambda}$
vanish, and  $w_4$ reduces to $|\xi'|^4\otimes 1_\lambda$. Hence
\[
\operatorname{Pol}[\xi_1, \cdots, \xi_{n-1}]w_{0, \lambda} 
\supset
\operatorname{Pol}[\xi_1, \cdots, \xi_{n-1}]w_{4, \lambda} 
\]
for $\lambda=\lambda_4$. In turn, we have 
$V_0(\lambda_4)\supset V_4(\lambda_4)$.
Of course, these two $\gog'$-modules have the same infinitesimal character. 
Another couple with the same infinitesimal characters (but different from
the previous couple)
are the two $\gog'$-submodules $V_1(\lambda_4)$ and $V_3(\lambda_4)$ 
generated by $w_1$ 	and $w_3$, respectively.

Returning to 
$w_0,w_4\in\operatorname{Pol}[\xi_1, \cdots, \xi_{n-1}]\otimes 1_\lambda$ 
for $\lambda=\lambda_4$, we consider 
\[
M_{04}:=\varphi(\operatorname{Pol}[\xi_1, \cdots, \xi_{n-1}]\, {\mathbb C}\text{-}\mathrm{{span}}
\{1,\xi_n^2,\xi_n^4\}\otimes 1_\lambda). 
\]
It turns out that the $U(\gon'_-)$-submodule $M_{04}$ in
$M^\gog_\gop(\mC_\lambda)$ is $\gog'$-invariant. Clearly,
$V_4(\lambda_4)\subset V_0(\lambda_4)\subset M_{04}$.
Furthermore, it is possible to show we have a non-splitting
exact sequence of $\gog'$-modules:
\begin{eqnarray}   
0\to {M}^{\gog^\prime}_{\gop^\prime}({\la_4})\to M_{04}
\to {M}^{\gog^\prime}_{\gop^\prime}({\la_4}-4)\to 0.
\end{eqnarray} 

Similarly, there is a non-trivial extension
\begin{eqnarray}   
0\to {M}^{\gog^\prime}_{\gop^\prime}({\la_4}-1)\to M_{13}
\to {M}^{\gog^\prime}_{\gop^\prime}({\la_4}-3)\to 0
\end{eqnarray} 
 of the modules generated by $w_1$ and $w_3,$ denoted by $M_{13}$, and
the branching rule is

  $$
M^\gog_\gop(\la_4)\simeq
M_{04}\oplus M_{13}\oplus\bigoplus_{b\in\mN,b\not=0,1,3,4}M^{\gog'}_{\gop'}(\la_4 -b).
$$  
\end{example}

We generalize these observations and obtain
 the following precise description
of the extensions
among branching laws
 for integral parameters: 
\begin{theorem} \label{T.3.8}
Recall \eqref{eqn:lmdj}
for the definition of $\lambda_j$.  
\begin{enumerate}
\item[{\rm{(1)}}] 
Suppose $j=2k+1$ with $k\in\mathbb{N}_+$.  
Then the branching is the same as in the generic case:
 \begin{eqnarray}
M^{\gog}_{\gop}(\lambda_{2k+1})|_{{\mathfrak {g}}'}
\simeq 
\bigoplus_{b\in\mN}M^{\gog'}_{\gop'}(\lambda_{2k+1} -b)
\quad
\text{{\rm{(direct sum)}}}.
\end{eqnarray}

\item[{\rm{(2)}}] 
Suppose $j=2k$ with $k\in\mathbb{N}_+$.  
Then  
there exists a ${\mathfrak {g}}'$-submodule
 $M_{a,2k-a}\subset M^{\gog}_{\gop}(\lambda_{2k})$
 for each $a=0,\ldots,k-1$
with the following two properties:
The restriction $M_{\mathfrak{p}}^{\mathfrak{g}}(\lambda_{2k})$
decomposes into a direct sum of\/ $\mathfrak{g}'$-modules:
$$
M^\gog_\gop(\la_{2k})|_{\mathfrak{g}'}\simeq
\bigoplus_{a=0}^{k-1}M_{a,2k-a}\oplus M^{\gog'}_{\gop'}(\la_{2k} -k)\oplus\bigoplus_{b=2k+1}^\infty
M^{\gog'}_{\gop'}(\la_{2k} -b),
$$  
and there exists a non-split exact sequence of\/ $\mathfrak{g}'$-modules:
\begin{eqnarray}   
0\to {M}^{\gog^\prime}_{\gop^\prime}({\la_{2k}}-a)\to M_{a,2k-a}
\to {M}^{\gog^\prime}_{\gop^\prime}({\la_{2k}}-(2k-a))\to 0.
\end{eqnarray} 
\end{enumerate}
\end{theorem}

In order to give a proof of the theorem, 
 we begin with the following elementary but useful observation.

\begin{lemma}
\label{lem:4.9}
Suppose $N$ is a ${\mathfrak {g}}$-module in the category ${\mathcal{O}}$, 
 and $V_1$, $V_2$ are two submodules of $N$ 
 satisfying the following three conditions:
\begin{enumerate}
\item[1)] 
(character identity)
$\operatorname{Ch}(V_1) + \operatorname{Ch}(V_2)
=\operatorname{Ch}(N)$, 
\item[2)]$V_1$ is irreducible, 
\item[3)]$\dim \operatorname{Hom}_{\mathfrak {g}}(V_1, V_2)
= 
\dim\operatorname{Hom}_{\mathfrak {g}}(V_1, N)=1.    
$
\end{enumerate}
Then there exists a non-split exact sequence
 of ${\mathfrak {g}}$-modules:
\begin{equation}
\label{eqn:short}
0\to V_2\to N\to V_1\to 0.  
\end{equation}
\end{lemma}

\begin{proof}
We note
 that $V_1 \subset V_2 \subset N$ from 
 the third condition.  
Since $V_1$ is irreducible,
$V_1$ is isomorphic to the quotient $N/V_2$ by the first condition.  
Thus we have an exact sequence
 \eqref{eqn:short} of ${\mathfrak {g}}$-modules.  
Then the third condition implies 
 that \eqref{eqn:short} does not split.  
\end{proof}

\medbreak
The next lemma analyzes a relationship between two singular vectors
with the same infinitesimal characters.
\begin{lemma}
\label{lem:dualC}
For $a, k \in {\mathbb{N}}$ 
 such that $a \le 2k$, 
 the Gegenbauer polynomials satisfy:
\[
  C_{a}^{-k}(s)=C_{2k-a}^{-k}(s)
\,\,
\text{and}
\,\,
{\mathcal C}_{a}^{-k}(s)=s^{a-k} {\mathcal C}_{2k-a}^{-k}(s).  
\]
\end{lemma}
\begin{proof}
By using the identity $\Gamma(\alpha)\Gamma(1-\alpha)=\frac{\pi}{\sin \pi \alpha}$, 
 the polynomial expression \eqref{eqn:Cpoly}
 of the Gegenbauer polynomial can be
 written as 
\[
  C_{a}^{-k}(s)
  =
  (-1)^a k ! \sum_{i=0}^{[\frac{a}{2}]}\frac{(2s)^{a-2i}}{\Gamma(1-a+i+k)i! (a-2i)!}.  
\]
Switching $a$ with $2k-a$, 
 we have 
\[
C_{2k-a}^{-k}(s)
  =
  (-1)^{2k-a} k ! \sum_{i=0}^{[\frac{2k-a}{2}]}
  \frac{(2z)^{2k-a-i}}{\Gamma(1-k+a+i)i! (2k-a-2i)!}, 
\]
where the terms for $i=0,\cdots,k-a-1$
 vanish if $a<k$.  
Putting $j=i+a-k$, 
 we have
\[
C_{2k-a}^{-k}(s)
  =
  (-1)^a k ! \sum_{j=0}^{[\frac{a}{2}]}
  \frac{(2z)^{a-2j}}{\Gamma(1+j)(j+k-a)! (a-2j)!}.  
\]
Hence $C_{a}^{-k}(s)=C_{2k-a}^{-k}(s)$.  
The last assertion follows immediately from the definition \eqref{eqn:Ctilde}.  
\end{proof}

\begin{lemma}\label{lem:4.14}
Let $k \in {\mathbb{N}}$.  
We denote by $f_a(\xi) \equiv f_{a,\lambda_{2k}}(\xi)$
 ($a \in {\mathbb{N}}$)
 the polynomials defined in \eqref{eqn:FK}.
Then for any $a \in {\mathbb{N}}$
 such that $a \le k$, 
 we have
\[
  f_{2k-a}(\xi)
  =
  (\sum_{i=1}^{n-1}\epsilon_i \xi_i^2)^{k-a}f_{a}(\xi).  
\]
\end{lemma}
\begin{proof}
Let $t=\xi_n^{-2} \sum_{i=1}^{n-1}\epsilon_i \xi_i^2$.  
By \eqref{eqn:FK}
 and Lemma \ref{lem:dualC},
 we have 
\begin{equation*}
f_a(\xi) = \xi_n^a {\mathcal C}_a^{-k}(t)
          = \xi_n^a t^{a-k} {\mathcal C}_{2k-a}^{-k}(t).  
\end{equation*}
Hence $f_{2k-a}(\xi)=\xi_{n}^{2k-a}{\mathcal C}_{2k-a}^{-k}(t)
=(\sum_{i=1}^{n-1}\epsilon_i \xi_i^2)^{k-a}f_a(\xi)$.  
\end{proof}

\begin{proof}
[Proof of Theorem \ref{T.3.8}]
Since there is no extension between two modules with distinct generalized
infinitesimal characters,
we can collect the terms
 of the same generalized infinitesimal characters
 as direct summands.

Let us examine this decomposition in our setting.
First we observe that
the identity \eqref{eqn:4.9} holds
 in the Grothendieck group
 for all $\lambda \in {\mathbb{C}}$ by Theorem \ref{thm:3.5}.  
We recall from \eqref{eqn:lmdj} that
$\lambda_j = \frac{1}{2}(-n+1+j)$.
This shows that
 generalized
$\mathfrak{Z}(\mathfrak{g}')$-infinitesimal characters
 decomposition of $\mathfrak{g}'$-modules:
\[
M_{\mathfrak{p}}^{\mathfrak{g}}(\lambda_j)|_{\mathfrak{g}'}
= \bigoplus_{\substack{b\in\mathbb{N} \\ 2b\ge j}} N_b,
\]
that appear in the direct summands of the restriction
$M_{\mathfrak{p}}^{\mathfrak{g}}(\lambda_j)|_{\mathfrak{g}'}$
are of the form
\[
(\frac{j}{2}-b, \frac{n-3}{2}, \frac{n-5}{2}, \dots, \frac{n-1}{2} - [\frac{n-1}{2}]),
\]
for some $b\in\mathbb{N}$ with $\frac{j}{2}\le b$.
Correspondingly, in the Grothendieck group,
or equivalently,
as the character identity,
we have
\[
N_b = \begin{cases} M_{\mathfrak{p}'}^{\mathfrak{g}'} (\lambda_j - b)
                           & (j < b), \\
                           M_{\mathfrak{p}'}^{\mathfrak{g}'} (\lambda_j - b)
                              + M_{\mathfrak{p}'}^{\mathfrak{g}'} (\lambda_j - j+b)
                           & (\frac{j}{2} < b \le j), \\
                           M_{\mathfrak{p}'}^{\mathfrak{g}'} (\lambda_j - b)
                           & (b = \frac{j}{2}).
         \end{cases}
\]
On the other hand, 
the ${\mathfrak {g}}'$-module
 $M_{\mathfrak{p}'}^{\mathfrak{g}'}(\lambda_j-b)$
 is irreducible for any $b \in \mathbb{N}$
with $\frac{j}{2} \le b$ by Corollary \ref{cor:girred} below.
Let us consider the ${\mathfrak {g}}'$-module structure
 of $N_b$ for $\frac j 2 < b \le j$.  
It follows from Theorem \ref{basis}  that
$w_{b,\lambda_j}$ and $w_{j-b,\lambda_j}$
 $\operatorname{Sol}
({\mathfrak {g}}, {\mathfrak {g}}'; {\mathbb{C}}_{\lambda_j})$
generate two $\mathfrak{g}'$-submodules in
$\operatorname{Pol}(\mathfrak{n}_+)\otimes\mathbb{C}_{\lambda_{j}}$,
to be denoted by $M_b$ and $M_{j-b}$,
which are isomorphic to
$M_{\mathfrak{p}'}^{\mathfrak{g}'}(\lambda_j-b)$
and $M_{\mathfrak{p}'}^{\mathfrak{g}'}(\lambda_j-j+b)$,
respectively.

Furthermore, if $j=2k$ then by Lemma \ref{lem:4.14}
with $b:=2k-a$ on an explicit knowledge of singular vectors
we have
\[
f_b(\xi) = (\sum_{j=1}^{n-1} \varepsilon_i \xi_i^2)^{b-k} f_{2k-b}(\xi)
\]
for $b\ge k$.
Thus $M_{2k-b}$ is a submodule of $M_b$.
The theorem follows by application of Lemma \ref{lem:4.9}.
Here we take $V_1$ to be $M_b$
 and $V_2$ to be $M_{j-b}$. 
\end{proof}

\subsubsection{Factorization identities} \label{4.3.5}
Let us return back to the Example \ref{Ex2}.
For $\la_2=\frac{-n+3}{2},$ the action of $\gog'$ on the top two singular vectors $w_0$ and $w_2$
generate the $\gog'$-submodules $V_0$ and $V_2$
 in the $\gog$-module $M_\gop^\gog(\la_2)$, respectively.
 The second one is a submodule of the first one. The corresponding inclusion
is a $\gog'$-homomorphisms $\psi$, whose dual differential operator is the conformally
invariant Yamabe operator. 
If we denote by $\phi_0$ and $\phi_2$ the inclusions of
$V_0$ and $V_2$, respectively, into $M^{\gog}_{\gop}(\la_0)$,
 we get the relation
$$
\phi_2=\phi_0\circ\psi.
$$

The F-method explains this factorization as 
\[
f_2(\xi)=(\sum\limits_{i=1}^{n-1}\epsilon_i\xi_i^2)f_0(\xi)
\]
by applying Lemma \ref{lem:4.14} with $a=0$ and $k=1$.

As another example, let us consider the weight $\la_4=-\frac{n}{2}+\frac{5}{2}.$
Then the $\gog'$-submodule $V_1$ generated by the singular vector $w_1$ and the 
$\gog'$-submodule $V_3$
generated by the singular vector $w_3$ have the same infinitesimal character.
There exists a $\gog'$-homomorphism $\psi $ from $V_3$ to $V_1.$
The homomorphism $\phi_3$ from $V_3$ to $M_\gop^\gog(\la_0)$ can be factorized
as $\phi_1\circ \phi.$ The F-method explains this factorization as
\[
f_3(\xi)=(\sum\limits_{i=1}^{n-1}\epsilon_i\xi_i^2)f_1(\xi)
\]
 by applying Lemma \ref{lem:4.14} with $a=1$ and $k=2$.

 Hence for some particular discrete subset of values for $\la,$ there is a possibility
to factor an element in $\Hom_{\gog'}(M^{\gog'}_{\gop'}(\la'),M^\gog_\gop(\la))$ as 
a composition of an element
in  the space $\Hom_{\gog'}(M^{\gog'}_{\gop'}(\la'),M^{\gog'}_{\gop'}(\la''))$ and an element in
$\Hom_{\gog'}(M^{\gog'}_{\gop'}(\la''),M^\gog_\gop(\la))$.
There is also another possibility to factor an element in
$\Hom_{\gog'}(M^{\gog'}_{\gop'}(\la'),M^\gog_\gop(\la))$ as a composition of an element
in $\Hom_{\gog'}(M^{\gog'}_{\gop'}(\la'),M^\gog_\gop(\la''))$ and 
in $\Hom_{\gog}(M^{\gog}_{\gop}(\la''),M^{\gog}_{\gop}(\la))$.

 The fact that such a behaviour
 can happen only for discrete values
 of $\la$ is a consequence
of classification of homomorphisms of $\gog$-generalized Verma modules. These properties were discovered 
and used effectively
for curved generalizations by A. Juhl (see \cite[Chapter 6]{ju}) under the name factorization identities.
 It is not a special
feature of this particular example 
with $G=SO_o(1,n+1)$
but it is a more general fact.  
It holds not only in Juhl's case (the scalar case) but also in
spinor-valued case.   
If we consider not only differential intertwining operators for the 
restriction but also  continuous intertwining operators ("symmetry breaking operators"),
then the factorization identity may hold for continuous values of parameters
see \cite{KEast}, \cite[Chapters 8, 12]{KSpeh}. 
In the F-method,
 the factorization identities are derived from the identities of
 two polynomials in $\operatorname{Sol}$ such as the formula given in Lemma \ref{lem:4.14}.
This viewpoint will be pursued in the second part of the series \cite{KM}.
See also \cite[Sect.9]{KP2}.
  
 In the dual language of differential operators the factorization is described as follows:
The first example above expresses the Juhl operator $D_2$ as the composition of the operator $D_0$ and the Laplace operator.
The second example shows that the operator $D_3$ is given by the composition of $D_1$ and the Laplace operator.  

 \subsection
{The case $G=G^\prime=SO_o(p,q)$.}

We consider an application
 of the F-method to the special case $G'=G$.  
In this case, we do not need branching laws.
Even in this classical situation, 
 we shall observe
 that the F-method yields a simple and new independent construction 
of all differential intertwining operators for $G=SO_o(p,q)$-modules induced from densities. 

For $\lambda\in{\mathbb C}$ we recall from Section \ref{3.1} that
${\mathbb C}_\lambda$ the one-dimensional representation of the parabolic 
subgroup $P$. We write $M_{\mathfrak{p}}^{\mathfrak{g}}(\lambda)$ for
$M_{\mathfrak{p}}^{\mathfrak{g}}({\mathbb C}_\lambda)$ as before.
With the notation as in Theorem \ref{basis},
we give a classification of all singular vectors via the bijection
$
M_{\mathfrak{p}}^{\mathfrak{g}}(\lambda)^{\mathfrak{n}_+}
  \overset{\sim} {\underset{\varphi}{\leftarrow}}
  \operatorname{Sol}(\mathfrak{g},\mathfrak{g};\mathbb{C}_\lambda)
$
by the next proposition.

\begin{proposition}\label{prop:girr}
Let $\mathfrak{g}(\mathbb{R})=\mathfrak{so}(p,q)$
and $\lambda\in\mathbb{C}$.
Recall $n=p+q-2$ and $w_{K,\lambda}$ is defined by the formula in \eqref{eqn:wN}.
Then we have:
\begin{align*}
&\text{$n$: even}
\\
&&&
\operatorname{Sol}(\mathfrak{g},\mathfrak{g};\mathbb{C}_\lambda)
 \simeq
   \begin{cases}  \mathbb{C}w_{0,\lambda}\oplus\mathbb{C}w_{2\lambda+n,\lambda}
                       &\text{if $\lambda+\frac{n}{2}\in\mathbb{N}_+$ },
                       \\
                       \mathbb{C}w_{0,\lambda}\oplus\mathbb{C}w_{2\lambda+n,\lambda}
                            \oplus \mathcal{H}^{\lambda+1} (\mathbb{R}^{p-1,q-1})
                       &\text{if $\lambda \in \mathbb{N}$},
                       \\
                       \mathbb{C}w_{0,\lambda}
                       &\text{otherwise}.
   \end{cases}
\\[1ex]
&\text{$n$:odd}
\\
&&&
\operatorname{Sol}(\mathfrak{g},\mathfrak{g};\mathbb{C}_\lambda)
\simeq
   \begin{cases}  \mathbb{C}w_{0,\lambda}\oplus\mathbb{C}w_{2\lambda+n,\lambda}
                       &\text{if $\lambda+\frac{n}{2}\in\mathbb{N}_+$},
                       \\
                       \mathbb{C}w_{0,\lambda}
                            \oplus \mathcal{H}^{\lambda+1} (\mathbb{R}^{p-1,q-1})
                       &\text{if $\lambda \in \mathbb{N}$},
                       \\
                       \mathbb{C}w_{0,\lambda}
                       &\text{otherwise}.
   \end{cases}
\end{align*}
\end{proposition}

\begin{proof}
[Proof of Proposition \ref{prop:girr}]
Most of the proof was already given in that of Theorem \ref{basis}.
In particular, we see
\[
\mathcal{H}^{\lambda+1}(\mathbb{R}^{p-1,q-1})
\subset \operatorname{Sol} (\mathfrak{g},\mathfrak{g};\mathbb{C}_\lambda)
\quad\text{for all $\lambda\in\mathbb{N}$}.
\]
In view of the obvious inclusion
$
\operatorname{Sol}(\mathfrak{g},\mathfrak{g};\mathbb{C}_\lambda)
\subset \operatorname{Sol}(\mathfrak{g},\mathfrak{g}';\mathbb{C}_\lambda),
$
it suffices to determine 
for which $K$ and $\lambda$ the vector $w_{K,\lambda}$
belongs to $\operatorname{Sol}(\mathfrak{g},\mathfrak{g};\mathbb{C}_\lambda)$
when $K\in\mathbb{N}_+$.
This is equivalent to the condition that $f_{K,\lambda}$ defined in \eqref{eqn:FK}
is $\mathfrak{so}(p,q)$-invariant.
The form  of the polynomial $f_{K,\lambda},$ the relation
$\epsilon_n=-1,$  and the invariance of $f_{K,\lambda}$ with respect to
$\mathfrak{so}(p,q)$ imply that
$C_K^{-\lambda-\frac{n-1}{2}}(s)$
is a multiple of $(1-s^2)^m$ for some $m\in\mathbb{N}$.
This happens if and only if
\[
K=2m \quad\text{and}\quad
\lambda+\frac{n}{2}=m.
\]
To see this,
we verify whether or not $(1-s^2)^m$ satisfies the Gegenbauer differential equation like
$C_{2m}^\alpha(s)$ for $\alpha=-\lambda-\frac{n-1}{2}$.
Since
\[
((1-s^2)\frac{d^2}{ds^2} - (1+2\alpha s^2)\frac{d}{ds} + 4m(m+\alpha)s^2)
(1-s^2)^m
= 2m(2m+2\alpha-1) s^2(1-s^2)^{m-1},
\]
this is zero if and only if $2m+2\alpha-1=0$,
namely, $\lambda+\frac{n}{2}=m$.
Hence Proposition \ref{prop:girr} is proved.  
\end{proof}
Forgetting a concrete description
 of singular vectors
 in Proposition \ref{prop:girr},
we still have the following abstract result as a corollary:
\begin{corollary}\label{cor:girred}
The generalized Verma module $M_{\mathfrak{p}}^{\mathfrak{g}}(\lambda)$
is irreducible
 if and only if
$\frac{n}{2}+\lambda \not \in\mathbb{N}_+$ ($n$ even);
$\lambda \not\in \mathbb{N}$
 and $\frac{n}{2}+\lambda \not\in \mathbb{N}_+$ ($n$ odd).
\end{corollary}

 If $F$ is a homomorphism from a
generalized Verma module $M_\gop^\gog(V)$
 to $M_\gop^\gog(\mC_\lambda),$
 then the image of $1\otimes V$ by $F$ 
is an $\gol$-irreducible subspace of 
$M_\gop^\gog(\lambda)^{\gon_+}=\varphi(\mbox{\rm Sol}(\gog,\gog;{\mathbb C}_\lambda))$.
In such a way, Proposition \ref{prop:girr} gives not only a new proof
 of the well-known classification of all homomorphisms from a generalized Verma module  $M_\gop^\gog(\mC_\lambda)$
for this specific pair $(\mathfrak{g},\mathfrak{p})$,
 but also an explicit construction
 of such homomorphisms
 by special values of the Gegenbauer polynomials.
  
In the dual language, the singular vector $w_{0,\lambda}$ gives the identity operator 
on the induced representations 
${\Ind}_{P}^{G}(\mC_{-\lambda})$, whereas $w_{2\lambda+n,\lambda}$ and ${\fam2 H}^{\lambda+1}(\mR^{p-1,q-1})$
give rise to $G$-intertwining differential operators  
\begin{alignat}{2}
&
{\Ind}_{P}^{G}(\mC_{\frac{n}{2}-m})\to\,
&&{\Ind}_{P}^{G}(\mC_{\frac{n}{2}+m}), 
\label{eqn:GG1}
\\
&
{\Ind}_{P}^{G}(\mC_{1-k})\to\,
&&{\Ind}_{P}^{G}((-1)^k\otimes{\fam2 H}^k(\mR^{p-1,q-1})\otimes\mC_{1}), 
\label{eqn:GG2} 
\end{alignat}
with $m=\lambda+\frac{n}{2}, k=1+\lambda\,\in\mN$, respectively. In the non-compact picture, \eqref{eqn:GG1}
is given by the $m$-th power of the Laplacian
\[
\square^m: \lC^\infty(\mR^{p-1,q-1})\to\lC^\infty(\mR^{p-1,q-1}), 
\]
where $\square = \frac{\partial^2}{\partial x_1^2}+\cdots +\frac{\partial^2}{\partial x_{p-1}^2}
-\frac{\partial^2}{\partial x_p^2}-\cdots -\frac{\partial^2}{\partial x_{p+q-2}^2}$. 
For the operator \eqref{eqn:GG2}, we use the fact that the representation of 
$SO_o(p-1,q-1)$ on ${\fam2 H}^k(\mR^{p-1,q-1})$ is self-dual. Then \eqref{eqn:GG2}
is described in the non-compact picture by 
\begin{eqnarray}
& \lC^\infty(\mR^{p-1,q-1})\otimes{\fam2 H}^k(\mR^{p-1,q-1}) \to  \lC^\infty(\mR^{p-1,q-1}),
\nonumber \\ \nonumber
&  (u(x),f(\xi)) \mapsto f(\frac{\partial}{\partial x_1},\cdots , \frac{\partial}{\partial x_{p+q-2}})u .
\end{eqnarray}
Note that the powers $\square^m, m=\lambda+\frac{n}{2}$
with $\lambda\leq 0$ are special in the following sense. There exists their `curved' versions, i.e., on any manifold with  a given conformal structure, there are conformally invariant operators of
order $m=1,\ldots, \frac{n}{2}$ with symbol equal to $\square^m.$
These operators were constructed in $\cite{GJMS},$ and they are usually called
the GJMS operators. The structure of lower order curvature terms is very
complicated and was studied in many publications (see, e.g., \cite{ju2} 
and the references
in \cite{ju}). On the contrary, curved analogues of $\square^m$ for $m>\frac{n}{2}$ do not exist ($\cite{GH}$).

The series $\mathscr{D}_k$ of operators constructed above are examples 
of the first BGG operators (equations for the conformal Killing tensors in Theorem \ref{thm:4.3}) 
given by the projection
 to the symmetric trace-free part of the multiple gradient $\nabla_{(a}\ldots\nabla_{b)_0}\sigma$ 
(number of indices being $k$).
 As the trace-free condition translates by the Fourier transform to the harmonicity condition, the 
 operators $\mathscr{D}_k$ correspond by Proposition \ref{prop:girr} to $\mathcal{H}^k(\mathbb{R}^{p-1,q-1})$
 with $k:=\lambda+1 \in \mathbb{N}_+$.
 

\section{Dirac operators and $Spin_o(p,q)$} \label{sec:5}
 \numberwithin{equation}{section}
 \setcounter{equation}{0}

In the present section 
 we extend the scalar-valued results 
 considered in Section \ref{juhlexample} 
 to those for spinor-valued sections.  
The symmetries for the base manifolds
 remain the same, given by the pair of Lie groups 
$(\widetilde{G},\widetilde{G}')=(Spin_o(p,q), Spin_o(p,q-1))$.  
The main results are Theorems \ref{thm:5.6} and \ref{finalspinor}. 

\subsection{Notation}\label{4.1}

We shall use the same convention as in Section \ref{sec:4}. 
Let $p\geq 1, q\geq 2, n=p+q-2,n=n'+1,$ and
we suppose that the quadratic form
\eqref{eqn:pq}
 on ${\mathbb{R}}^{p+q}={\mathbb{R}}^{n+2}$ is given
 as in Section \ref{3.1}. Let us consider the associated Clifford algebra
 $\lC_{p,q}$.  
It is generated
 by an orthonormal basis $e_0,\ldots,e_{p+q-1}$ with
the relations 
\[
     e_i^2=-\epsilon_i \text{ for }
     i=1,\dots ,p+q-2
     \text{ and }
     e_0e_{p+q-1}+e_{p+q-1}e_0=1.  
\]
Let $\lC_{p-1,q-1}$ be its subalgebra generated by 
$e_1,\ldots,e_{p+q-2}.$ 
We realize $Spin(p,q)$ in $\lC_{p,q}$
 and define ${\tilde G}$ to be the 
 identity component $Spin_o(p,q)$. 
We write 
\[
   \Pi:\, Spin_o(p,q)\to SO_o(p,q)
\] 
 for the canonical homomorphism, 
 which is a double covering.  
Via $\Pi$, 
 $\tilde{G}$ acts on $\mR^{p,q}$ preserving
the null cone ${\mathcal{N}}_{p,q}$ of $\mR^{p,q}$ 
 and the projective null cone ${\mathbb{P}}{\mathcal{N}}_{p,q}$.  
We shall keep the notation as in Section  \ref{3.1}.
The subgroup $\tilde{P}\subset \tilde{G}$ is  defined 
as the stabilizer of the chosen null line generated by 
the vector $(1,0,\ldots,0),$
namely, 
 $\tilde{P}=\Pi^{-1}(P).$ 
According to the Langlands decomposition 
$P=L N_+=MAN_+$ in $G=SO_o(p,q)$, 
 we have a Langlands decomposition 
\[
   \tilde P=\tilde L N_+=\tilde MAN_+
\]
by setting $\tilde{L}:=\Pi^{-1}(L)$
 and $\tilde{M}:=\Pi^{-1}(M) \simeq Spin(p-1,q-1)$.  
Here by a little abuse of notation,
 we regard $A$ and $N_+$ as subgroups of $\tilde G$.  

The Lie algebras $\tilde{\gog}$ and $\tilde{\gop}$
 are isomorphic to ${\mathfrak {g}}$ and ${\mathfrak {p}}$, 
respectively, 
 considered in the case $G=SO_o(p,q)$.
We take a Cartan subalgebra $\goh$ in $\gog$ so that
$\goh\subset\gol$.  
Let us denote by $\mS_{\pm}^n\equiv \mS_\pm^{p-1,q-1}$
 the irreducible half-spin representations
for $\tilde{M}\simeq{Spin}(p-1,q-1)$ with $n=p+q-2$ even, and $\mS^{n}\equiv \mS^{p-1,q-1}$ 
the spin representation
for $\tilde{M}\simeq{Spin}(p-1,q-1)$ with $n=p+q-2$ odd. We have $\mS_\pm^{n}\simeq \mS^{n-1}$ for $n$ even 
and $\mS^{n}\simeq \mS_+^{{n-1}}\oplus \mS_-^{{n-1}}$ for $n$ odd.
By an abuse of notation, 
 we write $\mS$ for ${\mS}_\pm$ in the proof, since the differential action is the same. 
The differential action of the spinor representation
 is given by 
\[
  {\mathfrak {so}}(p-1,q-1) \to {\mathcal{C}}_{p,q}, 
\quad
  \epsilon_i\epsilon_jE_{ij}-E_{ji}
  \mapsto -\frac{1}{2}\epsilon_i e_i e_j
  \quad
  (1\leq i\not= j\leq n).  
\]
Here $E_{ij}$ stands for the matrix unit
 with $1$ at the $(i,j)$-component, 
 and we recall that 
$
   \{ X_{ij}=\epsilon_i\epsilon_jE_{ij}-E_{ji}:
 1 \le i < j \le n=p+q-2\}$ forms 
 a basis of ${\mathfrak {so}}(p-1,q-1)$.

\subsection{Representations $d\pi_\la$ and $d\tilde{\pi}_\la.$}
\label{srsection}

For $\lambda \in {\mathbb{C}}$,  
 we define the twisted spinor representation 
 $\mS_\la$ 
 of the Levi factor $\widetilde L=\tilde MA$
 as the outer tensor product
 $\mS \otimes \mC_\lambda$,
 where ${\mathbb{C}}_{\lambda}$ is the one-dimensional 
 representation of $A$
 with the same normalization
 as in Section  \ref{3.2}. 
The differential representation of 
 $\widetilde{\mathfrak {l}}\simeq \mathfrak {l}$
 has the highest weight
 $(\la+\frac{1}{2},\frac{1}{2},\ldots,\frac{1}{2})$
 (and $(\la+\frac{1}{2},\frac{1}{2},\ldots,-\frac{1}{2})$
 for $n$ even).
We extend $\mS_\la$ 
 to a representation of $\widetilde P$
 by letting the unipotent radical $N_+$ act trivially.  
The (unnormalized) induced representation $\Ind^{\tilde G}_{\tilde P}\mS_\la$,
 is denoted by $\pi_{S,\lambda}$, 
 or simply by $\pi_{\lambda}$.  
The representation space
 is identified with $\lC^\infty(\tilde{G},\mS_\la)^{\tilde P}$,
 consisting of smooth functions
 $F:\widetilde G \to \mS_\lambda$
 subject to 
$$
  F(g\,\tilde{m}an)=a^{-\la}
  \mS(m)^{-1}F(g),\quad 
  \text{for all }g\in \tilde{G},\, \tilde{p}=\tilde{m}an\in \tilde{P}, 
$$
where 
$$
\Pi(\tilde{m}an)=\left(
\begin{array}{ccc}
\epsilon(m) a& \star & \star \\ 
0 & m & \star\\
0 & 0 & \epsilon(m) a^{-1} 
\end{array}
\right),\;
m \in SO(p-1,q-1),\,a>0,   
$$
and $\epsilon(m)=+1$ or $-1$ according to whether or not $m$ belongs
to the identity component of $SO_o(p-1,q-1)$.

Similarly, we may treat ${\Ind}_{\widetilde P}^{\widetilde G}(\mathrm{sgn}\, \otimes\, \mS_\lambda)$
by the condition $F(g\tilde{m}\tilde{a}n)=\epsilon(m)^{-1}a^{-\lambda}{\mS}(m^{-1})F(g)$
as in the scalar case (see Remark \ref{rem:4.4.}), but we omit it because the differential
action is the same and the main results hold by a small modification in the signature 
cases. 

Restricting to the open Bruhat cell, 
 we have the non-compact picture
 $\lC^{\infty}({\mathbb{R}}^{p-1,q-1}, \mS_{\lambda})$ 
 of the induced representation 
 $\operatorname{Ind}_{\tilde P}^{\tilde G}(\mS_{\lambda})$.  
In order to calculate the action
 of $d\pi_{\lambda}(Z)$
 in the non-compact picture 
 for $Z\in \gon_+({\mathbb R})$,
 we apply
the previous computation in the scalar case. 
We already observed
 in \eqref{eqn:4.8} 
 that if $Z \in {\mathfrak {n}}_+({\mathbb R})$
 and $X \in {\mathfrak {n}}_-({\mathbb R})$
 are sufficiently small, 
then the element $\tilde m \in \tilde M$
 determined by the condition 
\[
  (\exp Z)^{-1} \exp X \in N_-\tilde m A N_+
\]
 behaves as
\begin{eqnarray}
\label{eqn:5.1}
\operatorname{Id}-\Pi(\tilde m)
 \sim \mJ {}^t\! Z\otimes {}^t\! X\mJ - X\otimes Z=\sum_{i,j=1}^n z_ix_j (\epsilon_i\epsilon_jE_{ij}-E_{ji}),
\end{eqnarray}
up to the first order in $\| Z \|$.  
The right-hand side of \eqref{eqn:5.1}
 acts as the multiplication
by the element 
\begin{eqnarray}\label{eqn:Clzx}
-\frac{1}{2}\sum_{i\not= j}^nz_ix_j\epsilon_ie_ie_j &=&
-\frac{1}{2}\Big((\sum_{i=1}^n\epsilon_iz_ie_i)(\sum_{j=1}^nx_je_j) -\sum_{j=1}^nz_ix_i\epsilon_ie_i^2\Big)
\nonumber \\
&=&-\frac{1}{2}\big(\uz\, \ux+\sum_{i=1}^nx_iz_i\big)
\end{eqnarray}
in the corresponding 
Clifford algebra, where $\ux=\sum_1^n x_ie_i,$ and $\uz=\sum_1^n \epsilon_iz_ie_i.$

We define differential operators on $\gon_+({\mathbb R})\simeq {\mathbb R}^n$
in the coordinates $(\xi_1, \cdots, \xi_{n-1}, \xi_n)$ by
\begin{alignat}{2}
D:=&\sum_{k=1}^{n} e_k \partial_{\xi_k}
  =D' + e_n \partial_{\xi_n}
&&\text{(the Dirac operator on ${\mathbb{R}}^{p-1,q-1}$)}, 
\label{eqn:5.6}
\\
E:=
&\sum_{j=1}^{n-1} \xi_j \partial_{\xi_j}
 + \xi_n \partial_{\xi_n}
\qquad
&&\text{(the Euler homogeneity operator)}, 
\notag
\\
\square:=&-D^2=\square'-\pa_{\xi_n}^2, 
\quad
\square'=\sum_{j=1}^{n-1}\epsilon_j \pa_{\xi_j}^2.  
&&
\notag
\end{alignat}

Summarizing the information obtained so far,
 we get the following claim
 as in Lemma \ref{lem:4.1}:
\begin{lemma}\label{lem:5.1}
The basis element $E_j \in {\mathfrak {n}}_+({\mathbb R})$
 $(1 \le j \le n)$ acts 
 on $\lC^\infty(\mR^{p-1,q-1},\mS_\la)$
(i.e., in the non-compact picture for $\lC^\infty(\tilde{G},\mS_\la)^{\tilde P}$) 
in the coordinates $\{x_1, \cdots , x_n\}$ of ${\gon_-({\mathbb R})}\simeq {\mathbb R}^{p-1,q-1}$ 
by 
\begin{eqnarray}\label{eqn:spinaction}
d\pi_{S,\lambda}(E_j)=
-\frac{1}{2}\epsilon_j|X|^2\partial_{x_j}
+x_j(\lambda 
+\sum_k x_k\partial_{x_k}+\frac{1}{2})
+\frac{1}{2}(\epsilon_je_j\ux) .
\end{eqnarray}
 The dual action composed with the Fourier transform is given by 
\begin{eqnarray} \label{eqn:spinactiondual}
d\tilde{\pi}_{S,\lambda}(E_j)
=i\left(\frac{1}{2}\epsilon_j\xi_j\square 
+(\lambda-E-\frac{1}{2})\partial_{\xi_j}-\frac{1}{2}\epsilon_je_jD\right).  
\end{eqnarray}
\end{lemma}
\begin{proof}
It follows from (\ref{eqn:Clzx}) that 
\[
d\pi_{S,\lambda}(E_j)=d\pi_{\lambda}(E_j)+\frac{1}{2}(x_j+\epsilon_je_j\ux),
\]
whence the formula (\ref{eqn:spinaction}). In turn, 
\[
d\tilde{\pi}_{S,\lambda}(E_j)
=P_j(\lambda)-\frac{i}{2}(\partial_{\xi_j}+\epsilon_je_jD),
\]
whence the formula (\ref{eqn:spinactiondual}) owing to Lemma \ref{lem:4.1}.
\end{proof}

\subsection{The space  $\mbox{\rm Sol}$ of singular vectors}

In the scalar case, we proved in Theorem \ref{basis} that $\mathrm{Sol}(\gog, \gog'; {\mathbb C}_\lambda)$
consists of polynomials which are invariant under $\gom'\simeq\mathfrak{so}(n-1,{\mathbb C})$ as far as 
$\lambda\notin\mN$. In the spinor case, we shall consider first such invariant solutions. For 
this, we work with $\lC_n^{\mathbb C}(\simeq\lC_{p-1,q-1}\otimes_{\mR}\mC)$-valued polynomials in 
$\xi_1,\cdots , \xi_n$, on which $\widetilde{M}\simeq Spin(p-1,q-1)$ acts as 
\[
s\mapsto gs(g^{-1}\cdot)\quad \text{for} \quad 
s \in \mathrm{Pol}[\xi_1, \cdots , \xi_n]\otimes\lC^\mC_n .
\]
It is obvious that the following elements 
$$
\uxi':=\sum_{j=1}^{n-1}\epsilon_j\,e_j\,\xi_j,\,\underline{\xi_n}
    :=\epsilon_n\, e_n\xi_n, \xi_n
$$
belong to $(\mathrm{Pol}[\xi_1, \cdots , \xi_n]\otimes\lC^\mC_n)^{\widetilde{M'}}$,
and so does any polynomial generated by these three elements. We note
 $(\uxi')^2=-|\xi'|^2=-\sum_{j=1}^{n-1}\epsilon_j\xi_j^2$
 and $(\underline{\xi_n})^2=-\epsilon_n\xi_n^2=\xi_n^2$ as $\epsilon_n=-1$. 
We set
\[
t:=\epsilon_n\frac{|\xi'|^2}{\xi_n^2}.
\]
Then homogeneous polynomials $F_K$ of degree $K$ 
generated by $\uxi', \underline{\xi_n}$ and $\xi_n$
are written as follows:
for $K=2N$ it is of the form
\begin{eqnarray}\label{pol:even}
F_{2N}(\xi_1,\cdots ,\xi_n)=\xi_n^{2N}P(t)+ \xi_n^{2N-2} Q(t)\underline{\xi'}\underline{\xi_n}\, , 
\end{eqnarray}
where $P(t)$ and $Q(t)$ are polynomials in the variable $t,$ 
$P(t)$ is  of degree
 $N$ and $Q(t)$ is of degree $N-1$;
for $K=2N+1$, $F_K$ is of the form:
\begin{eqnarray}\label{pol:odd}
F_{2N+1}(\xi_1,\cdots ,\xi_n)=\xi_n^{2N}(P(t)\uxi'+Q(t)\underline{\xi_n}),
\end{eqnarray}
where both  $P(t)$ and $Q(t)$ are polynomials of degree $N.$ 

Let us consider the question when the spinor
$F_K\cdot s_{\lambda}$, $s_\lambda\in\mS_\lambda$, belongs to
$\mathrm{Sol}(\gog, \gog'; {\mathbb S}_\lambda)$, namely, 
$F_K\cdot s_\lambda$ is annihilated by the operators
$$
-2i d \tilde \pi_{S,\lambda}(E_j)
=\epsilon_j\xi_j\square
-(2E-2\la+1)\pa_{\xi_j}-\epsilon_je_j\,D,\,j=1,\ldots,n-1, 
$$
where the notation \textquotedbl$\cdot$\textquotedbl$\,$ in $F_K\cdot s_\lambda$ means 
the tensor product followed by
the Clifford multiplication.
This leads us
 to a system of ordinary differential equations
 for the polynomials $P(t)$ and $Q(t).$
We shall first treat the case of even homogeneity $K=2N.$

\begin{lemma}
\label{spinorhypergeom}
Let $\lambda \in {\mathbb{C}}$, $N \in {\mathbb{N}}$ and let $s_\lambda\in\mS_\lambda$ be a non-zero vector.  
For polynomials $P(t)$ and $Q(t)$, we set 
\[
F_{2N} =\xi_n^{2N}P(t)+\xi_n^{2N-2}Q(t)\uxi'\underline{\xi_n}  .
\]
Then $F_{2N}\cdot s_\lambda\in\mathrm{Sol}(\gog,\gog';\mS_\lambda)$ 
 if and only if
 the following system of ordinary differential equations is satisfied:
\begin{align}
& R(2N, -\lambda-\frac{n}{2}+1)P=0
\label{5.7}
\\
& R(2N-1, -\lambda-\frac{n}{2}+1)Q=0,
\label{5.8}
\\
&-2\,N\,P+2\,t\,P'+(4N-2\la-n)\,Q-2\,tQ'=0,
\label{5.9}
\\
&2\,P'-(2N-1)\,Q+2\,tQ'=0.
\label{5.10}
\end{align}
The first two equations actually 
 follow from the last two equations.
\end{lemma}
\begin{proof}
By Lemma \ref{lem:5.1}, $F_{2N}\cdot s_\lambda\in\mathrm{Sol}(\gog,\gog';\mS_\lambda)$
if and only if 
\[ 
d\tilde{\pi}_{S,\lambda}(E_j)(F_{2N}\cdot s_\lambda)=0\quad\mbox{for}\quad 1\leq j\leq n-1.
\]
We set $F_{2N}=v_1+v_2$ with $v_1=\xi_n^{2N}P(t)$ and 
$v_2=\xi_n^{2N-2}Q(t)\underline{\xi'}\underline{\xi_n}.$ Then
\begin{align*}
\pa_{\xi_j}v_1=& 2\epsilon_n\epsilon_j\xi_j\xi_n^{2N-2}P'(t),
\\
\pa_{\xi_j}^2v_1=&2\epsilon_n\epsilon_j\xi_n^{2N-2}P'(t)+ 4\xi_n^{2N-4} \xi_j^2P''(t),
\\
\square'v_1=&\epsilon_n\xi_n^{2N-2}(4\,t\,P''+(2n-2)\,P'),
\\
\pa_{\xi_n} v_1=&\xi_n^{2N-1}(2N\,P-2\,t\,P'),
\\
\pa_{\xi_n}^2 v_1=&
\xi_n^{2N-2}(2N(2N-1)\,P+(-8N+6)\,t\,P'+4\,t^2\,P''),
\\
e_jD'v_1=&2\epsilon_n\,e_j\xi_n^{2N-2}\, P'\uxi',
\\
e_je_{n}\pa_{\xi_n} v_1=&\epsilon_n\, e_j\xi_n^{2N-2}(2N\,P-2\,t\,P')\underline{\xi_n}.
\end{align*}
Similarly,
\begin{align*}
\pa_{\xi_j}v_2=&\epsilon_n\epsilon_j\xi_n^{2N-4}2Q'\xi_j\underline{\xi'}\underline{\xi_n}+
\xi_n^{2N-2}\epsilon_je_jQ(t)\underline{\xi_n},
\\
\square'\,v_2=&\xi_n^{2N-4}\epsilon_n(4\,t\,Q''+(2n+2)\,Q')\underline{\xi'}\underline{\xi_n},
\\
\pa_{\xi_n} v_2=&\xi_n^{2N-3}((2N-1)\,Q-2\,t\,Q')\underline{\xi'}\underline{\xi_n},
\\
\pa_{\xi_n}^2 v_2=&
\xi_n^{2N-4}((2N-1)(2N-2)\,Q+t\,(-8N+10)\,Q'+4\,t^2\,Q'')\underline{\xi'}\underline{\xi_n},
\\
-e_jD'v_2=&-\epsilon_ne_j\xi_n^{2N-2}(2\,t\,Q'+n'\,Q)\underline{\xi_n},
\\
-e_je_{n}\pa_{\xi_n} v_2=&\epsilon_ne_j\xi_n^{2N-2}((2N-1)\,Q-2\,t\,Q')\underline{\xi'}.
\end{align*}

Collecting all terms of $d\tilde{\pi}_{S,\lambda}(E_j)(F_{2N}\cdot s_\lambda)$
with respect to the basis 
$$
\epsilon_j\epsilon_n\xi_j\xi_n^{2N-2},\;\;\epsilon_j\epsilon_n\xi_j\xi_n^{2N-4}\underline{\xi'}\underline{\xi_n},
\;\;\epsilon_j\,e_j\,\epsilon_n\xi_n^{2N-2}\uxi',\;\;
\epsilon_j\,e_j\,\epsilon_n\xi_n^{2N-2}\underline{\xi_n},
$$
we see that the the system of equations 
 $d \tilde \pi_{S,\lambda} (E_j)(F_{2N}\cdot s_\lambda)=0$ $(1\leq j\leq n-1)$
 is equivalent to the four equations in the lemma.

It can be easily checked that a suitable linear combination
 of the last two equations (\ref{5.9}) and (\ref{5.10})
 and their differentials implies the first two equations 
(\ref{5.7}) and (\ref{5.8}). For 
example, the application of $\frac{d}{dt}$ and $\frac{d^2}{dt^2}$ to the equation (\ref{5.10})
gives 
\begin{eqnarray}\nonumber
2P'=-2tQ'+(2N-1)Q,\quad 2\, P''=-2\,Q'-2\, tQ'' + (2N-1)\, Q',
\end{eqnarray}
and their substitution into the equation (\ref{5.9})
yields the equation (\ref{5.8}). Hence the proof of Lemma \ref{spinorhypergeom}
is complete.
\end{proof}


%
\begin{lemma} \label{defPQ}
Let $N \in {\mathbb{N}}$ and $\lambda\in\mC$, 
and $s_\lambda\in\mS_\lambda$ a non-zero vector.  
We set 
\begin{eqnarray}
\tilde{F}_{2N}(\xi_1,\cdots ,\xi_n) 
= \xi_n^{2N}\widetilde{\mathcal C}^{-\lambda-\frac{n}{2}+1}_{2N}\left(\frac{|\xi'|^2}{\xi_n^2}\right)
+\xi_n^{2N-2}\widetilde{\mathcal C}^{-\lambda-\frac{n}{2}+1}_{2N-1}\left(\frac{|\xi'|^2}{\xi_n^2}\right)
\underline{\xi'}\underline{\xi_n}. 
\label{5.13}
\end{eqnarray}  
Then $\tilde{F}_{2N}\cdot s_\lambda$ is a spinor--valued 
homogeneous polynomial of $\xi_1$, $\cdots$, $\xi_{n-1}$, 
 and $\xi_n$ of degree $2N$, and belongs to 
${\rm{Sol}}({\mathfrak {g}}, {\mathfrak {g}}'; \mS_{\lambda})$.
\end{lemma}  
\noindent
\begin{proof}
By Lemma \ref{spinorhypergeom}, it is sufficient to solve the system of
ordinary differential equations \eqref{5.7}--\eqref{5.10} for 
polynomials $P(t)$ and $Q(t)$.

It follows from Lemma \ref{lem:ghGegen} in Appendix that the 
polynomial solutions $P(t)$ and $Q(t)$ are of the form 
\[
P(t)=A\,\widetilde{\mathcal C}^{\alpha}_{2N}(-t),\, 
Q(t)=B\,\widetilde{\mathcal C}^{\alpha}_{2N-1}(-t)\quad \text{with} \quad 
\alpha=-\lambda-\frac{n}{2}+1,
\]
for some $A,B\in{\mathbb C}$. Let us show that 
\eqref{5.9} and \eqref{5.10} are fulfilled for 
this pair ($P(t), Q(t)$) if and only if $A=B$.
We shall deal with \eqref{5.10} below, and omit 
\eqref{5.9} which gives the same conclusion by 
a similar argument using the formula \eqref{eqn:Cidentity}
in Appendix.

Suppose $t=-\frac{1}{x^2}$. We note that if two functions 
$g(x)$ and $h(t)$ are related by the formula
\[
g(x)=x^lh(t)\quad (\equiv x^lh(-\frac{1}{x^2})),
\]
then $\frac{dx}{dt}=\frac{1}{2}x^3$ and thus 
\begin{eqnarray}  
h'(t)=\frac{1}{2}(x^{-l+3}g'(x)-lx^2h(t)).  \label{eqn:ghdiff}
\end{eqnarray} 
Applying \eqref{eqn:ghdiff} to 
\[
g_P(x):=x^{2N}P(t) \quad\text{and}\quad g_Q(x):=x^{2N-1}Q(t),
\]
we get 
\begin{align}
& 2P'(t)  \quad\quad\quad\quad\quad\quad\quad\,\,\, = x^{-2N+3}g_P'(x)-2Nx^2P(t),\\
& -(2N-1)Q(t)+2tQ'(t) = -x^{-2N+2}g_Q'(x). 
\end{align}
Since $g_P(x)=A\widetilde{C}^{\alpha}_{2N}(x)$ and 
$g_Q(x)=B\widetilde{C}^{\alpha}_{2N-1}(x)$ with 
$\alpha=-\lambda-\frac{n}{2}$ (see \eqref{eqn:Ctilde} or Lemma
\ref{lem:ghGegen}), \eqref{5.10} amounts to 
\[
A((\alpha+N)x\,\widetilde{ C}^{\alpha+1}_{2N-1}(x)
-N\widetilde{ C}^{\alpha}_{2N}(x))
-B\widetilde{ C}^{\alpha+1}_{2N-2}(x)=0
\]
by \eqref{der_ap}. Using the identity
\eqref{eqn:C3even}, we see that this holds if and only if
$A=B$. Thus the proof of Lemma \ref{defPQ} is completed.  

\end{proof}
 
The case of odd homogeneity is contained in the next two lemmas,
 for which the proof is similar and omitted.
\begin{lemma}
\label{spinorhypergeomodd}
Let $N \in {\mathbb{N}}$, $\lambda\in\mC$, and $s_\lambda\in\mS_\lambda$ a non-zero
vector. 
Then 
\[
{F}_{2N+1}\cdot s_\lambda=\xi_n^{2N}(P(t)\underline{\xi'}+Q(t)\underline{\xi_n})\cdot s_\lambda
\]
is annihilated
 by $d \tilde \pi_{S,\lambda} (E_j)$
 $(1 \le j \le n-1)$
if and only if
the polynomials
 $P(t)$ and $Q(t)$ satisfy
 the following system of ordinary differential equations:
\begin{align}
& R(2N, -\lambda-\frac{n}{2}+1)P=0,	\label{5.14}\\
& R(2N+1, -\lambda-\frac{n}{2}+1)Q=0, \label{5.15}\\
& (4N-2\lambda-n+2)\,P-2\,t\,P'-(2N+1)\,Q+2\,tQ'=0,\\
& 2NP-2\, tP'-2\,Q'=0.
\end{align}
Furthermore,
 the first two equations follow from the last two equations.
\end{lemma}

\begin{lemma}\label{lem:svodd} 
Let $N\in\mN$, $\lambda\in\mC$, and $s_\lambda\in\mS_\lambda$ a non-zero
vector. We set
\begin{eqnarray}
\tilde{F}_{2N+1}(\xi_1,\cdots ,\xi_n) 
= \xi_n^{2N}\left(\widetilde{\mathcal C}^{-\lambda-\frac{n}{2}+1}_{2N}\left(\frac{|\xi'|^2}{\xi_n^2}\right)\underline{\xi'}
+(-\lambda-\frac{n}{2}+N+1)\widetilde{\mathcal C}^{-\lambda-\frac{n}{2}+1}_{2N+1}\left(\frac{|\xi'|^2}{\xi_n^2}\right)
\underline{\xi_n}\right). \nonumber \\
\label{oddhomvectors}
\end{eqnarray}  
Then $\tilde{F}_{2N+1}\cdot s_\lambda$ is of homogeneous degree $2N+1$
 and is annihilated
 by $d \tilde \pi_{\lambda}(E_j) (1 \le j \le n-1)$, and therefore belongs to ${\mathrm{Sol}}(\gog,\gog';\mS_\lambda)$
for any $s_\lambda\in\mS_\lambda$.
\end{lemma}   
\begin{remark}
In contrast to the even case in Lemma \ref{defPQ}, the 
coefficient $-\lambda-\frac{n}{2}+N+1 \, (=\alpha+N)$ shows 
up in the odd case in Lemma \ref{lem:svodd} with respect to 
the renormalized Gegenbauer polynomials. We note
\[
{\mathcal C}^{\alpha}_{2N}(-t)\underline{\xi'}+{\mathcal C}^{\alpha}_{2N+1}(-t)\underline{\xi_n}=
(\alpha)_N(\widetilde{\mathcal C}^{\alpha}_{2N}(-t)\underline{\xi'}
+(\alpha+N)\widetilde{\mathcal C}^{\alpha}_{2N+1}(-t)\underline{\xi_n}).
\] 
\end{remark}
As in Section 4, the homomorphisms of generalized Verma modules defined by the singular vectors described  
above induce equivariant differential operators acting on local sections of induced homogeneous vector bundles 
on the generalized flag manifolds. We describe these differential operators in the non-compact picture
of the induced representations. 

In the following theorem, we retain the notation of 
Section \ref{4.1}: $\mS^n\equiv\mS^{p-1,q-1}$ is the
spin representation of $Spin(p-1,q-1)$, and 
$\mS_\pm^n\equiv\mS_\pm^{p-1,q-1}$ are the half-spin 
representations when $n=p+q-2$ is even. They 
are extended to the representations 
$\mS_\lambda^n\equiv\mS_\lambda^{p-1,q-1}$,
$\mS_{\pm,\lambda}^n\equiv\mS_{\pm,\lambda}^{p-1,q-1}$,
respectively, of the parabolic subgroup $\widetilde P=\widetilde MAN_+$
with $\widetilde M\simeq Spin(p-1,q-1)$ by letting $A$
act as the one-dimensional representation ${\mathbb C}_\lambda$
and $N_+$ trivially. Then the branching law of the restriction 
with respect to the pair of the parabolic subgroups 
$\widetilde P\supset {\widetilde P}'$ of $Spin_o(p,q)\supset Spin_o(p,q-1)$
is given as 
\begin{eqnarray}
 & & \mS_{\lambda}^n \simeq \mS_{+,\lambda}^{n-1}\oplus \mS_{-,\lambda}^{n-1},  \quad  n: \text{odd},
\nonumber \\
 & & \mS_{\pm,\lambda}^n  \simeq \mS_{\lambda}^{n-1},  \quad\quad\quad\,\,\,  n: \text{even}.
\end{eqnarray}
We recall from \eqref{acoef} and \eqref{bcoef} the polynomials
$a_j(\lambda)\equiv a^{N,n}_j(\lambda)$ and 
$b_j(\lambda)\equiv b^{N,n}_j(\lambda)$. 
Then we have
\begin{eqnarray}
& & N!\,\widetilde{\mathcal C}^{-\lambda-\frac{n}{2}+1}_{2N}(-t)=
\sum\limits_{j=0}^{N}a_j(\lambda-\frac{1}{2})t^j, \nonumber 
\\
& & N!\,\widetilde{\mathcal C}^{-\lambda-\frac{n}{2}+1}_{2N-1}(-t)=
2N\sum\limits_{j=0}^{N-1}b_j(\lambda-\frac{1}{2})t^j, \nonumber 
\\
& & N!\,\widetilde{\mathcal C}^{-\lambda-\frac{n}{2}+1}_{2N+1}(-t)
      =
      2\,\sum\limits_{j=0}^{N}b_j(\lambda-\frac{1}{2})t^j.\nonumber
\end{eqnarray}
Therefore, by applying Theorem \ref{folkloreg}  to Lemmas \ref{defPQ} and \ref{lem:svodd}
with $\lambda$ replaced by $-\lambda$, we obtain the following theorem:

\begin{theorem}\label{thm:5.6}
Let $({\widetilde G}, {\widetilde G}')=(Spin_o(p,q), Spin_o(p,q-1))$ and $\lambda\in{\mathbb C}$. We decompose
the Dirac operator $D$ on $\mR^{p-1,q-1}=\mR^{p-1,q-2}\oplus\mR^{0,1}$ as 
\[
D=D'+\underline{\partial}_n\equiv \sum_{i=1}^{n-1}e_i\frac{\partial}{\partial x_i}
+e_n\frac{\partial}{\partial x_n},
\]
and write the Laplace--Beltrami operator on $\mR^{p-1,q-2}$ as
\[
\square'=-(D')^2=\sum_{i=1}^{n-1}\epsilon_i\frac{\partial^2}{\partial x_i^2}.
\]
We introduce a family of $\mathrm{End}(\mS^n)$-valued differential operators $D_K^{\mS}(\lambda)$
of order $K$ $(K\in\mN)$ by 
\begin{align*}
 D_{2N}^{\mS}(\lambda)&:=
  \sum\limits_{j=0}^N {a}_j(-\la -\frac{1}{2})(\square')^j\frac{\partial^{2N-2j}}{\partial x_n^{2N-2j}}
                  +2N\,\sum\limits_{j=0}^{N-1} {b}_j(-\la -\frac{1}{2})(\square')^j
									\frac{\partial^{2N-2j-2}}{\partial x_n^{2N-2j-2}}{D}'\underline{\pa}_n,
\\
 D_{2N+1}^{\mS}(\lambda)&:=
                  \,\sum\limits_{j=0}^N {{ a}}_j(-\la-\frac{1}{2})(\square')^j
									 \frac{\partial^{2N-2j}}{\partial x_n^{2N-2j}}{D}'
                  +(-2\lambda-n+2N+2)\sum\limits_{j=0}^{N} {{ b}}_j(-\la-\frac{1}{2})(\square')^j
									\frac{\partial^{2N-2j}}{\partial x_n^{2N-2j}} \underline{\pa}_n.
%
%
\end{align*}
\begin{enumerate}
\item[(1)]
The differential operators $D_K^\mS(\lambda)$ $(K\in\mN)$ induce ${\widetilde G}'$-homomorphisms 
\begin{align*}
& \mathrm{Ind}^{\widetilde G}_{\widetilde{P}}(\mS^n_\lambda)\to\bigoplus_{\epsilon=\pm} 
\mathrm{Ind}^{{\widetilde G}'}_{{\widetilde P}'}(\mS^{n-1}_{\epsilon, \lambda+K}) \quad \text{for $n$ odd},
\\
 & \mathrm{Ind}^{\widetilde G}_{\widetilde{P}}(\mS^n_{\pm, \lambda})\to
\mathrm{Ind}^{{\widetilde G}'}_{{\widetilde P}'}(\mS^{n-1}_{\lambda+K})  \quad\quad\,\, \text{for $n$ even},
\end{align*}
by the following formulas
\[
\mathscr{D}_K^\mS(\lambda)f:=(D_K^\mS(\lambda)f)|_{x_n=0}
\]
in the non-compact picture
\begin{align*}
& \mathscr{D}_K^\mS(\lambda): 
\lC^\infty(\mR^{p-1,q-1}, \mS^{n})\to \lC^\infty(\mR^{p-1,q-2}, \mS^{n-1}_{+}\oplus\mS^{n-1}_{-})
& & \text{for $n$ odd},
\\
& \mathscr{D}_K^\mS(\lambda): 
\lC^\infty(\mR^{p-1,q-1}, \mS^{n}_\pm)\to \lC^\infty(\mR^{p-1,q-2}, \mS^{n-1})
 & & \text{for $n$ even},
\end{align*}
In particular, the following infinitesimal relations are satisfied in both cases:
\begin{eqnarray}
\mathscr{D}_K^\mS(\lambda)d\pi^{\widetilde G}_{\mS, \lambda}(X)
=d\pi^{{\widetilde G}'}_{\mS, \lambda+K}(X)\mathscr{D}_K^\mS(\lambda)
\quad
\text{ for all $X\in\gog'$}.  
\end{eqnarray} 
\item[(2)]
Conversely, if $\lambda\in{\mathbb C}$ satisfies
\[
-\lambda+n-\frac{3}{2}\notin\mN_+\quad \text{and}\quad -2\lambda+n-1\notin\mN_+,
\]
and if there exist an irreducible finite-dimensional representation $W$ of
$\widetilde{P'}$ and a non-trivial differential ${\widetilde G}'$-homomorphism $T$, then
\begin{align*}
&
W\simeq
   \begin{cases}  \mS^{n-1}_{\epsilon, \lambda+K}
                       &\text{($n$ odd)},
                       \\
                   \mS^{n-1}_{\lambda+K}    
                       &\text{($n$ even)},
 \end{cases}
\end{align*}
for some $K\in\mN$ and $\epsilon=\pm$ ($n$ odd case) and $T$ is given by 
a scalar multiple of $\mathscr{D}^\mS_K(\lambda)$ in the non-compact picture.
\end{enumerate}
\end{theorem}
\begin{remark}\label{rem:5.7}
For $n=p+q-2$ odd, the infinitesimal action ${d}\pi^{\widetilde G}_{\mS,\lambda}$
of the Lie algebra $\gog$ on the non-compact picture decomposes into a direct 
sum of two $\gog'$-modules:
\[
\lC^\infty(\mR^{p-1,q-1}, \mS^{n})\simeq \lC^\infty(\mR^{p-1,q-1}, \mS^{n-1}_{+})\oplus
\lC^\infty(\mR^{p-1,q-1}, \mS^{n-1}_{-}).
\]
If $f\in \lC^\infty(\mR^{p-1,q-1}, \mS^{n-1}_\epsilon)$ $(\epsilon=\pm)$, then 
$\mathscr{D}_K^\mS(\lambda)f\in \lC^\infty(\mR^{p-1,q-2}, \mS^{n-1}_\delta)$
where $\delta=(-1)^K\epsilon$.
\end{remark}
\begin{proof}[Proof of Theorem \ref{thm:5.6}.] 
\begin{enumerate}
\item[(1)]
It follows from 
Lemmas \ref{defPQ} and \ref{lem:svodd} that 
\[
\widetilde{F}_K\cdot s_\lambda \in \mathrm{Sol}(\gog, \gog'; \mS^n_\lambda)
\quad \text{for all $K\in\mN$ and $s_\lambda\in\mS^n_\lambda$}.
\]
Furthermore, since $\widetilde{F}_K$ is an $\gom'$-invariant element in
$\mathrm{Pol}[\gon_+]\otimes\mathrm{End}(\mS^n_\lambda)$, the subspace 
$\widetilde{F}_K(\mS^n_\lambda)$ (or $\widetilde{F}_K(\mS^n_{\pm,\lambda})$
for $n$ odd) belongs to the same $\gom'$-isotypic component, namely,
\[
\mS^n|_{\gom'}\simeq \mS^{n-1}_+\oplus \mS^{n-1}_-\, (n:\text{even}), 
\quad\text{or}\quad \mS^n_\pm|_{\gom'}\simeq \mS^{n-1}\, (n:\text{odd}). 
\]
Now the F-method in Section \ref{sec:2 correction} (with $\lambda$
replaced by $-\lambda$) leads us to Theorem \ref{thm:5.6} 
as in the scalar case (Theorem \ref{basis}).
\item[(2)]
The second statement follows from Theorem \ref{folkloreg} and Proposition
\ref{prop:Sbranch} below.
\end{enumerate}
\end{proof}
\begin{proposition}\label{prop:Sbranch}
(Branching laws $\gog\downarrow\gog'$)
We recall $n=p+q-2$ and $(\gog, \gog')=(\mathfrak{so}(n+2,{\mathbb C}), \mathfrak{so}(n+1,{\mathbb C}))$.
\begin{enumerate}
\item[(1)]
In the Grothendieck group of $\gog'$-modules, we have 
\begin{alignat}{2}
M^\gog_\gop(\mS^n_\lambda)|_{\gog'} &\simeq\bigoplus_{\epsilon=\pm}\bigoplus_{b=0}^\infty 
M^{\gog'}_{\gop'}(\mS^{n-1}_{\epsilon,\lambda -b})
\quad &&\text{for $n$ odd}, \nonumber
\\
M^\gog_\gop(\mS^n_{\pm,\lambda})|_{\gog'} &\simeq\bigoplus_{b=0}^\infty 
M^{\gog'}_{\gop'}(\mS^{n-1}_{\lambda -b})
\quad &&\text{for $n$ even}, 
\notag
\end{alignat}
\item[(2)]
If $\lambda\in{\mathbb C}$ satisfies the following two conditions:
\begin{eqnarray}
& & \lambda+n-\frac{3}{2}\notin\mN_+, \label{eqn:1lmd}
\\ 
& & 2\lambda +n-1\notin\mN_+, \label{eqn:2lmd}
\end{eqnarray}
then the first statement gives an irreducible decomposition 
as $\gog'$-modules.
\end{enumerate}
\end{proposition}
\begin{proof}
\begin{enumerate}
\item[(1)]
We apply Theorem \ref{groth} with
\begin{align*}
&&&
F_\lambda :=
   \begin{cases}  \mS^n_\lambda\equiv \mS^n\otimes\mC_\lambda
                       &\text{for $n$ odd},
                       \\
                   \mS^n_{\epsilon,\lambda}\equiv \mS_\epsilon^n\otimes\mC_\lambda    
                       &\text{for $n$ even}.
 \end{cases}
\end{align*}
Then we have 
\begin{align*}
&&&
F_\lambda|_{\gol'}\simeq
   \begin{cases}  \mS^{n-1}_{+,\lambda}\oplus \mS^{n-1}_{-,\lambda}
                       &\text{for $n$ odd},
                       \\
                   \mS^{n-1}_\lambda   
                       &\text{for $n$ even}.
 \end{cases}
\end{align*}
Since the symmetric tensor algebra 
$S(\gon_-/\gon_-\cap\gog')\simeq\bigoplus\limits_{b\in\mN}{\mathbb C}_{-b}$
as a module of $\gol'\simeq\mathfrak{so}(n-1,{\mathbb C})\oplus\mathfrak{so}(2,{\mathbb C})$,
we have
\begin{align*}
&&&
F_\lambda|_{\gol'}\otimes S(\gon_-/\gon_-\cap\gog')\simeq
   \begin{cases}  \bigoplus\limits_{\epsilon=\pm}\bigoplus\limits_{b=0}^\infty 
\mS^{n-1}_{\epsilon,\lambda -b}
\quad &\text{for $n$ odd},
                       \\
\bigoplus\limits_{b=0}^\infty 
\mS^{n-1}_{\lambda -b}
\quad &\text{for $n$ even}.                					
 \end{cases}
\end{align*}
Here the first statement follows from Theorem \ref{groth}. 
\item[(2)]
Suppose $n$ is odd. Then the ${\mathfrak Z}(\gog')$-infinitesimal character
of the $\gog'$-module $M^{\gog'}_{\gop'}(\mS^{n-1}_{\epsilon,\lambda-b})$ is
given by
\begin{eqnarray*}
(\lambda-b+\frac{n-1}{2}, \frac{n}{2}-1, \cdots , \frac{5}{2}, \frac{3}{2}, \epsilon\frac{1}{2})
\in{\mathbb C}^{\frac{n+1}{2}} / W(D_{\frac{n+1}{2}}).
\end{eqnarray*}
They are distinct when $b$ runs over $\mN$ and $\epsilon=\pm$ if and only if 
\[
\lambda-b +\frac{n-1}{2}\not= -(\lambda-b' +\frac{n-1}{2})
\]
for all $(b,b')\in\mN^2$ with $b\not=b'$, namely, $\lambda$ satisfies
(\ref{eqn:2lmd}). Furthermore, the $\gog'$-module $M^{\gog'}_{\gop'}(\mS^{n-1}_{\epsilon,\lambda-b})$
is irreducible if 
\[
\lambda-b +n+\frac{3}{2}\notin\mN_+
\] 
by  (\ref{eqn:anti}), and  in particular, if (\ref{eqn:1lmd}) is satisfied.

Therefore, if  both (\ref{eqn:1lmd}) and (\ref{eqn:2lmd}) are fulfilled,
then there is no extension among the irreducible
$\gog'$-modules $M^{\gog'}_{\gop'}(\mS^{n-1}_{\epsilon,\lambda-b})$,
and hence the formula in $(1)$ gives a direct sum of irreducible
$\gog'$-modules.
 
Suppose $n$ is even. Then the ${\mathfrak Z}(\gog')$-infinitesimal character
of the $\gog'$-module $M^{\gog'}_{\gop'}(\mS^{n-1}_{\lambda-b})$ is
\begin{eqnarray*}
(\lambda-b+\frac{n-1}{2}, \frac{n}{2}-1, \cdots , 2, 1)
\in{\mathbb C}^{\frac{n}{2}} / W(B_{\frac{n}{2}}).
\end{eqnarray*}
They are distinct when $b$ runs over $\mN$ if and only if 
(\ref{eqn:2lmd}) is satisfied. Furthermore, the condition (\ref{eqn:anti})
amounts to 
\[
2(\lambda-b +\frac{n-1}{2})\notin\mN_+\quad\text{and}\quad \lambda-b +\frac{n-3}{2}\notin\mN_+
\]
which are satisfied if $\lambda$ fulfills (\ref{eqn:1lmd}) and (\ref{eqn:2lmd}). Thus the second statement 
also follows for $n$ even.
\end{enumerate}
\end{proof}
As in the scalar case,
 we have for $\lambda\in \mN+\frac{1}{2}$ an additional set of singular vectors 
in ${\rm{Sol}}({\mathfrak {g}}, {\mathfrak {g}}';\mS_{\lambda})$.
To see this, we retain the notation of Section \ref{4.1} and define the space of
monogenic spinors of degree ($j\in\mN$) by
\[
{\fam2 M}^j(\mR^{p-1,q-1}, \mS^{n}):=\{s\in\mathrm{Pol}[\xi_1, \cdots , \xi_n]\otimes\mS^n:Ds=0, Es=js \},
\]
where $D=\sum\limits_{k=1}^ne_k\frac{\partial}{\partial\xi_k}$ is the Dirac operator and
$E=\sum\limits_{k=1}^n\xi_k\frac{\partial}{\partial\xi_k}$ is the Euler homogeneity operator.
For $n=p+q-2$ even, we also define for $\epsilon=\pm$ 
\[
{\fam2 M}^j(\mR^{p-1,q-1}, \mS_\epsilon^{n}):=\{s\in\mathrm{Pol}[\xi_1, \cdots , \xi_n]\otimes\mS_\epsilon^n:Ds=0, Es=js \}.
\]
Then ${\fam2 M}^j(\mR^{p-1,q-1}, \mS^{n})$ for $n$ odd (or 
${\fam2 M}^j(\mR^{p-1,q-1}, \mS_\epsilon^{n})$ for $n$ even)
is an irreducible $Spin(p-1,q-1)$-submodule 
of $\mathrm{Pol}[\xi_1, \cdots , \xi_n]\otimes\mS^n$
with highest weight
$(j+\frac{1}{2}, \frac{1}{2}, \cdots , \frac{1}{2})$ for $n$ odd
(or $(j+\frac{1}{2}, \frac{1}{2}, \cdots , \epsilon\frac{1}{2})$ for $n$ 
even), respectively.
\begin{lemma}\label{lem:MkSol}
Suppose $\lambda\in -\frac{1}{2}+\mN$. Then 
\begin{eqnarray}
& & {\fam2 M}^{\lambda+\frac{1}{2}}(\mR^{p-1,q-1}, \mS^{n})\subset \mathrm{Sol}(\gog, \gog; \mS^n_\lambda),
\nonumber \\ \nonumber
& & {\fam2 M}^{\lambda+\frac{1}{2}}(\mR^{p-1,q-1}, \mS_\epsilon^{n})\subset \mathrm{Sol}(\gog, \gog; \mS^n_{\epsilon, \lambda})\quad \text{for $n$ even}.
\end{eqnarray}
\end{lemma}
\begin{proof}
Since $D^2=-\square$, we have 
\[
\xi_k\square - e_kD=-e_k(\mathrm{Id}-\epsilon_k\xi_ke_kD).
\]
Hence, by Lemma \ref{lem:5.1}, we have
\begin{eqnarray} \label{eqn:spinactiondualrewritten}
d\tilde{\pi}_{S,\lambda}(E_k) \nonumber
=i\left((\lambda-E-\frac{1}{2})\partial_{\xi_k}-\frac{1}{2}e_k(\mathrm{Id}-\epsilon_k\xi_ke_kD)D\right)\quad\text{for}\quad k=1,\cdots ,n.
\end{eqnarray}
If $s\in{\fam2 M}^{\lambda+\frac{1}{2}}(\mR^{p-1,q-1}, \mS^{n})$, then 
\begin{align}\nonumber 
 E(\partial_{\xi_k}s)=(\lambda-\frac{1}{2})\partial_{\xi_k}s\quad (1\leq k\leq n),\quad Ds=0.
\end{align}
Therefore, $d\tilde{\pi}_{S,\lambda}(E_k)s=0$ for all $k$ ($1\leq k\leq n$). Thus the lemma is proved.
\end{proof}
We also need the branching laws of these modules when restricted 
from $Spin(p-1,q-1)$ to the subgroup $Spin(p-1,q-2)$, which are given by 
\begin{eqnarray}\label{eqn:MBodd}
& & {\fam2 M}^{j}(\mR^{p-1,q-1}, \mS^{n})\simeq 
\bigoplus_{\epsilon=\pm}\bigoplus_{i=0}^j{\fam2 M}^{i}(\mR^{p-1,q-2}, \mS_\epsilon^{n-1})
\quad \text{for $n$ odd}, \\ \label{eqn:MBeven}
& & {\fam2 M}^{j}(\mR^{p-1,q-1}, \mS_\epsilon^{n})\simeq 
\bigoplus_{i=0}^j{\fam2 M}^{i}(\mR^{p-1,q-2}, \mS^{n-1})
\quad\quad\,\,\,\, \text{for $n$ even}.
\end{eqnarray} 

Let us summarize our results for spinor representation 
in the following theorem analogous to Theorem \ref{basis}.

\begin{theorem}\label{finalspinor}
Let $({\widetilde G}, {\widetilde G}')=(Spin_o(p,q), Spin_o(p,q-1))$, $n=p+q-2$, 
and $\lambda \in {\mathbb{C}}$. 
We recall from 
Lemmas \ref{defPQ} and \ref{lem:svodd} that 
$\widetilde{F}_K\equiv \widetilde{F}^\lambda_K\in\mathrm{Pol}[\xi_1,\dots ,\xi_n]\otimes\lC^{\mathbb C}_n$
($K\in\mN$) is defined by
\begin{eqnarray}
\label{eqn:spinwN}
&  & \tilde{F}_{2N} = 
\xi_n^{2N}\widetilde{\mathcal C}^{\alpha}_{2N}(-t)+
\xi_n^{2N-2}\widetilde{\mathcal C}^{\alpha}_{2N-1}(-t)
\underline{\xi'}\underline{\xi_n},
\nonumber \\
& &  \tilde{F}_{2N+1} = \xi_n^{2N}\left(\widetilde{\mathcal C}^{\alpha}_{2N}(-t)\underline{\xi'}
+(-\lambda-\frac{n}{2}+N+1)\,\widetilde{\mathcal C}^{\alpha}_{2N+1}(-t)\underline{\xi_n}\right),
\end{eqnarray}
where $\alpha=-\lambda-\frac{n}{2}+1$ and $t=\frac{\sum\limits_{j=1}^{n-1}\epsilon_j\xi_j^2}{\epsilon_n\xi_n^2}$. 
Then
\begin{enumerate}
\item[{\rm{(1)}}]
For any $\lambda \in {\mathbb C}$, we have 
\begin{eqnarray}
 & & \mathrm{Sol}(\gog, \gog'; \mS_\lambda)\supset \bigoplus_{K=0}^\infty
\{ \widetilde{F}_K\cdot s_\lambda: \, s_\lambda\in\mS_\lambda\},
\nonumber \\ \nonumber
 & & \mathrm{Sol}(\gog, \gog'; \mS_{\epsilon, \lambda})\supset \bigoplus_{K=0}^\infty
\{\widetilde{F}_K\cdot s_\lambda: \, s_\lambda\in\mS_{(-1)^K\epsilon,\lambda}\}
\quad \text{for $n$ even}.
\end{eqnarray}
Moreover, if $\lambda$ satisfies (\ref{eqn:1lmd}) and (\ref{eqn:2lmd}),
then the above inclusions are the equalities and the inverse Fourier 
transform (see (\ref{eqn:phi})) gives all the singular vectors via the 
isomorphisms as modules of
$\widetilde{L'}\simeq Spin(p-1,q-2)\times\mR$:
\begin{eqnarray}
 & &  M_{{\mathfrak {p}}}^{{\mathfrak {g}}}({\mS_\lambda})
  ^{{\mathfrak {n}}_+'} 
  \overset{\sim} {\underset{\varphi}{\leftarrow}}
 \mbox{\rm Sol}(\mathfrak{g},\mathfrak{g}';\mathbb{S}_\lambda)
 \simeq\bigoplus_{K=0}^{\infty}\{\tilde{F}_K\cdot s_\lambda:\, s_\lambda\in\mS_\lambda\},  
\nonumber \\ \nonumber 
& &  M_{{\mathfrak {p}}}^{{\mathfrak {g}}}({\mS_{\epsilon, \lambda}})
  ^{{\mathfrak {n}}_+'} 
  \overset{\sim} {\underset{\varphi}{\leftarrow}}
 \mbox{\rm Sol}(\mathfrak{g},\mathfrak{g}';\mathbb{S}_{\epsilon, \lambda})
 \simeq\bigoplus_{K=0}^{\infty}\{\tilde{F}_K\cdot s_\lambda:\, s_\lambda\in\mS_{(-1)^K\epsilon, \lambda}\} 
\quad \text{for $n$ even}. 
\end{eqnarray}

\item[{\rm{(2)}}]
 For $\la\in\mN+\frac{1}{2},$ 
 we have ${\mathfrak {l}}'$-injective maps
\begin{eqnarray}
 & &  M_{{\mathfrak {p}}}^{{\mathfrak {g}}}({\mS^n_\lambda})
  ^{{\mathfrak {n}}_+'} 
  {\underset{\varphi}
   {\overset \sim \leftarrow}
   }
 \,\,
  \mbox{\rm Sol}(\mathfrak{g},\mathfrak{g}';\mathbb{\mS}^n_\lambda)
 \quad\supset 
  \bigoplus_{K=0}^{\infty}\{\tilde{F}_K\cdot s_\lambda:\, s_\lambda\in\mS^n_\lambda\}
  \oplus \bigoplus_{\epsilon=\pm}\bigoplus_{i=1}^{\la + \frac 1 2}{\fam2 M}_{i,\epsilon},
	\quad\,\,\,\,\text{for $n$ odd},
	\nonumber\\ \nonumber
	& &  M_{{\mathfrak {p}}}^{{\mathfrak {g}}}({\mS^n_{\epsilon, \lambda}})
  ^{{\mathfrak {n}}_+'} 
  {\underset{\varphi}
   {\overset \sim \leftarrow}
   }
 \,\,
  \mbox{\rm Sol}(\mathfrak{g},\mathfrak{g}';\mathbb{\mS}^n_{\epsilon, \lambda})
 \supset 
  \bigoplus_{K=0}^{\infty}\{\tilde{F}_K\cdot s_\lambda:\, s_\lambda\in\mS^n_{(-1)^K\epsilon, \lambda}\}
  \oplus\bigoplus_{i=1}^{\la + \frac 1 2}{\fam2 M}_{i},
	\quad\text{for $n$ even},
\end{eqnarray}

Here
 ${\fam2 M}_{i,\epsilon}$ and ${\fam2 M}_{i}$ 
are the summands in the branching laws 
(\ref{eqn:MBodd}) and (\ref{eqn:MBeven}), respectively.    
\end{enumerate}
\end{theorem}

\vskip 0.8pc
In the second part \cite{KOSS2} of the series, we shall prove the existence of lifts 
of homomorphisms corresponding to singular vectors in generalized Verma 
modules induced from spinor representations to homomorphisms of semi-holonomic
generalized Verma modules covering them. According to the philosophy of
parabolic geometries, \cite{cs}, we get curved versions of our equivariant 
differential operators acting on sections of spinor bundles on manifolds 
with conformal structure.   


\section{Appendix: Gegenbauer polynomials}

 \numberwithin{equation}{section}
 \setcounter{equation}{0}
 
In the Appendix we summarize for reader's convenience a few basic 
conventions and properties of the Gegenbauer polynomials.

The Gegenbauer polynomials are defined in terms of their generating function
\begin{eqnarray}
\frac{1}{(1-2xt+t^{2})^{\alpha}}=\sum_{l=0}^{\infty}C_l^{\alpha}(x) t^{l},
\end{eqnarray}
and satisfy the recurrence relation
\begin{eqnarray}
C_{l}^\alpha(x) = \frac{1}{l}\left(2x(l+\alpha-1)C_{l-1}^\alpha(x) -
(l+2\alpha-2)C_{l-2}^\alpha(x)\right)
\end{eqnarray}
with $C_0^\alpha(x)  = 1,\, C_1^\alpha(x)  = 2 \alpha x$.
The Gegenbauer polynomials are solutions of the Gegenbauer differential equation
\begin{eqnarray}
\label{eqn:Gdiff}
\left((1-x^{2})\frac{d^2}{dx^2}-(2\alpha+1)x\frac{d}{dx}
+l(l+2\alpha)\right) g=0.
\end{eqnarray}
They are given as special values of the Gaussian hypergeometric series
when the series is finite:
\begin{eqnarray}
C_l^{\alpha}(x)=\frac{(2\alpha)_{{l}}}{l!}
\,_2F_1\left(-l,2\alpha+l;\alpha+\frac{1}{2};\frac{1-x}{2}\right)=
\label{eqn:Cpoly}
\sum_{k=0}^{[\frac{l}{2}]}
(-1)^k\frac{\Gamma(l-k+\alpha)}{\Gamma(\alpha)k!(l-2k)!}(2x)^{l-2k}.
\end{eqnarray}
We renormalize the Gegenbauer polynomials by 
\begin{eqnarray}\label{eqn:Gegen2}
\widetilde{C}_l^{\alpha}(x):=\frac{\Gamma(\alpha)}{\Gamma(\alpha+[\frac{l+1}{2}])}
C^\alpha_l(x)=\frac{1}{(\alpha)_{[\frac{l+1}{2}]}}C^\alpha_l(x).
\end{eqnarray}
Then ${\widetilde C}^\alpha_l(x)$ is a non-zero solution to 
\eqref{eqn:Gdiff} for all $\alpha\in{\mathbb C}$ and $l\in{\mathbb N}$.

We observe that $x\mapsto x^lC_l^\alpha(x^{-1})$ is an even
polynomial of degree $2[\frac{l}{2}]$. Therefore there exists 
uniquely a polynomial of degree $[\frac{l}{2}]$, to be denoted
by ${\mathcal C}^\alpha_l(t)$, such that 
\begin{eqnarray}\label{eqn:Cntilde}
{\mathcal C}^\alpha_l(x^2)=x^lC_l^\alpha(x^{-1}).
\end{eqnarray}
It follows from (\ref{eqn:Cpoly}) and (\ref{eqn:Cntilde}) that 
\begin{equation}
\label{eqn:tildeCpoly}
{\mathcal C}_l^{\alpha}(-t)=\sum_{k=0}^{[ \frac{n}{2}]}
\frac{\Gamma(l-k+\alpha)2^{l-2k}}{\Gamma(\alpha)k!(l-2k)!}t^{k}.
\end{equation}
Similarly to the renormalization \eqref{eqn:Gegen2},
we set
\[
\widetilde{\mathcal C}_l^{\alpha}(t)=
\frac{1}{(\alpha)_{[\frac{l+1}{2}]}}{\mathcal C}_l^{\alpha}(t).
\] 
Then
\begin{eqnarray}
& & \widetilde{\mathcal C}^\alpha_0(t)=1, \nonumber
\\
& & \widetilde{\mathcal C}^\alpha_1(t)=2, \nonumber
\\
& & \widetilde{\mathcal C}^\alpha_2(t)=2\, (\alpha+1)-t, \nonumber
\\
& & \widetilde{\mathcal C}^\alpha_3(t)=2\,(\frac{2}{3}(\alpha+2)-t), \nonumber
\\
& & \widetilde{\mathcal C}^\alpha_4(t)=\frac{1}{2}\,(\frac{4}{3}(\alpha+2)(\alpha+3)
-4(\alpha+2)t+t^2). \nonumber
\end{eqnarray}
In particular, for $\alpha=-\lambda-\frac{n}{2}$
\begin{eqnarray}
& & {N!}\, \widetilde{\mathcal C}^\alpha_{2N}(-t) =
\frac{N!}{(\alpha)_N}\,{\mathcal C}^\alpha_{2N}(-t)=
\sum\limits_{j=0}^Na_j(\lambda)t^j,
\label{eqn:Caj}
\\
& & \frac{1}{2}{N!}\, \widetilde{\mathcal C}^\alpha_{2N+1}(-t)
=\frac{N!}{2(\alpha)_{N+1}}\,{\mathcal C}^\alpha_{2N+1}(-t)
=\sum\limits_{j=0}^Nb_j(\lambda)t^j,
\label{eqn:Cbj}
\end{eqnarray}
where the coefficients
$a_j(\lambda)\equiv a^{N,n}_j(\lambda)$ and 
$b_j(\lambda)\equiv b^{N,n}_j(\lambda)$ are defined in 
(\ref{acoef}) and (\ref{bcoef}), respectively.

We recall from \eqref{eqn:Rla} that 
$$
R(l,\alpha):=4t(1+t)\frac{d^2}{dt^2}+((6-4l)t+4(1-\alpha-l))\frac{d}{dt}+l(l-1).
$$
\begin{lemma}\label{lem:ghGegen}
Suppose $l\in{\mathbb N}$ and $g(x)=x^lh(-\frac{1}{x^2})$.
\begin{enumerate}
\item[{\rm{(1)}}]
$h(t)$ satisfies $R(l,\alpha)h(t)=0$ if and only if $g(x)$ 
satisfies the Gegenbauer differential equation \eqref{eqn:Gdiff}.
\item[{\rm{(2)}}]
If $h(t)$ is a polynomial of degree $[\frac{l}{2}]$ and satisfies 
$R(l,\alpha)h(t)=0$, then 
g(x) is a scalar multiple of  the renormalized Gegenbauer polynomial
$\widetilde{C}^\alpha_l(x)$ and $h(t)$
is a scalar multiple of
$\widetilde{\mathcal C}^\alpha_l(-t)$.  
\end{enumerate}
\end{lemma}

The Gegenbauer polynomials satisfy the Rodrigues formula
$$
C_l^{\alpha}(x) =
\frac{(-2)^l}{l!}
\frac{\Gamma(l+\alpha)\Gamma(l+2\alpha)}
{\Gamma(\alpha)\Gamma(2l+2\alpha)}(1-x^2)^{-\alpha+1/2}\frac{d^l}{dx^l}
\left[(1-x^2)^{l+\alpha-1/2}\right],
$$
a basic formula for derivative  
\begin{eqnarray}
\label{der_ap}
 \frac{d}{dx}C^{\alpha}_{l}(x)=2\alpha\, C^{\alpha +1}_{l-1}(x)
\end{eqnarray}
and the following identities:
\begin{eqnarray}\label{eqn:Gpolyspec}
& & l C_{l}^{\alpha}(x)-2\alpha\, xC^{\alpha +1}_{l-1}(x)
+2\alpha C^{\alpha +1}_{l -2}(x)=0,
\\
\label{eqn:Cidentity}
& & -2\alpha\, C_{l}^{\alpha+1}(x)+(l+2\alpha)\, C^{\alpha}_{l}(x)
+2\alpha x\, C^{\alpha +1}_{l -1}(x)=0.
\end{eqnarray}

The formulas \eqref{der_ap} and \eqref{eqn:Gpolyspec} are
restated in terms of the renormalized Gegenbauer polynomials
$\widetilde{C}^\alpha_l$ as below:
\begin{eqnarray}
& & \frac{d}{dx}\widetilde{C}^{\alpha}_{2N}(x)=2(\alpha + N)\, \widetilde{C}^{\alpha +1}_{2N-1}(x), 
 \\
& & \frac{d}{dx}\widetilde{C}^{\alpha}_{2N+1}(x)=2\, \widetilde{C}^{\alpha +1}_{2N}(x), 
 \\
& & N \widetilde{C}_{2N}^{\alpha}(x)-(\alpha+N)\, x\widetilde{C}^{\alpha +1}_{2N-1}(x)
+ \widetilde{C}^{\alpha +1}_{2N -2}(x)=0, \label{eqn:C3even} \\
& & (2N+1)\, \widetilde{C}_{2N+1}^{\alpha}(x)-2\, x\widetilde{C}^{\alpha +1}_{2N}(x)
+ 2\,\widetilde{C}^{\alpha +1}_{2N -1}(x)=0. \label{eqn:C3odd}
\end{eqnarray}


\subsection*{Acknowledgments}

T. Kobayashi was partially supported by 
Institut des Hautes \'{E}tudes Scientifiques, France and
        Grant-in-Aid for Scientific Research (B) (22340026) and (A) (25247006), Japan
        Society for the Promotion of Science.  
P. Somberg and V. Sou\v cek acknowledge the financial support from the grant GA P201/12/G028.
 All four authors are grateful
 to the Max Planck Institut f\"ur Mathematik in Bonn,
where a part of the work was done.


 
  \end{document}